\documentclass[11pt]{article}
\usepackage[T1]{fontenc}
\usepackage[ansinew]{inputenc}
\usepackage{amssymb,amsmath,amsthm,eucal,hyperref,mathrsfs,array,graphicx,xcolor,bbm,enumerate,dsfont}
\usepackage[all]{xy}
\usepackage[labelfont=bf]{caption}
\usepackage[labelfont=bf]{subcaption}
\usepackage[top=2cm, bottom=2cm, left=2cm, right=2cm]{geometry}

\numberwithin{equation}{section}
\numberwithin{figure}{section}

\newtheorem{lemma}{Lemma}[section]
\newtheorem{propn}[lemma]{Proposition}
\newtheorem{thm}[lemma]{Theorem}
\newtheorem{cor}[lemma]{Corollary}
\newtheorem{defn}[lemma]{Definition}
\newtheorem{remark0}[lemma]{Remark}

\newtheorem{claim}[lemma]{Claim}

\newtheorem{question}[lemma]{Question}

\newenvironment{proof2}[1]{{\em Proof of #1.}}{\hspace*{\fill} $\square$}

\begin{document}
\title{\bf An invariance principle for branching diffusions in bounded domains} 
\author{Ellen Powell}
\date{}
\maketitle

\begin{abstract}
We study branching diffusions in a bounded domain $D$ of $\mathbb{R}^d$ in which particles are killed upon hitting the boundary $\partial D$. It is known that any such process undergoes a phase transition when the branching rate $\beta$ exceeds a critical value: a multiple of the first eigenvalue of the generator of the diffusion. We investigate the system at criticality and show that the associated genealogical tree, when the process is conditioned to survive for a long time, converges to Aldous' Continuum Random Tree under appropriate rescaling. The result holds under only a mild assumption on the domain, and is valid for all branching mechanisms with finite variance, and a general class of diffusions.
\end{abstract}

\thispagestyle{empty}

\newcommand{\test}{test}
\newcommand{\ellen}[1]{\textcolor{blue}{[#1]}}
\newcommand{\R}{\mathbb{R}}
\newcommand{\C}{\mathbb{C}}
\newcommand{\N}{\mathbb{N}}
\newcommand{\Z}{\mathbb{Z}}
\newcommand{\I}{\mathds{1}}
\newcommand{\dd}{\,{\mathrm d}}
\newcommand{\ddd}{{\mathrm d}}
\newcommand{\st}{\,\partial}
\newcommand{\stt}{\partial}
\newcommand{\e}{\operatorname{e}}
\newcommand{\im}{\operatorname{i}}
\newcommand{\ee}{\mathbf{e}}
\newcommand{\hate}{\hat{\mathbf{e}}}
\newcommand{\grad}{\operatorname{grad}}
\newcommand{\divg}{\operatorname{div}}
\newcommand{\bbd}{\mathcal{B}_b(D)}
\newcommand{\vp}{\varphi}
\newcommand{\td}{\tau^D}
\newcommand{\pd}{p^D_t(x,y)}
\newcommand{\nh}{\I_{\{\td>t\}}}
\newcommand{\nhe}{\I_{\{\td_\xi>t\}}}
\newcommand{\ip}[2]{\left( #1, #2 \right)}
 \newcommand{\eps}{\varepsilon}

\newcommand{\Prb}[2]{\mathbb{P}^{#1} \big({#2} \big)}
\newcommand{\Prbs}[2]{\mathbf{P}^{#1} \big({#2} \big)}
\newcommand{\Prbn}[1]{\mathbb{P} \left({#1} \right)}
\newcommand{\hl}{-\frac{1}{2}\Delta}
\newcommand{\vpcube}{\int_D \vp(y)^3 \, dy}
\newcommand{\E}[2]{\mathbb{P}^{#1}\big( #2 \big)}
\newcommand{\tP}[2]{\tilde{\mathbb{P}}^{#1}\big( #2 \big)}
\newcommand{\tQ}[2]{\tilde{\mathbb{Q}}^{#1}\big( #2 \big)}
\newcommand{\Estar}[2]{\mathbb{Q}^{#1}\big( #2 \big)}
\newcommand{\Estaro}[2]{\overline{\mathbb{Q}}_{#1}\left[ #2 \right]}
\newcommand{\En}[1]{\mathbb{P}\big( #1 \big)}
\newcommand{\Es}[2]{\mathbb{P}^{#1}\big(#2 \big)}
\newcommand{\Fs}[1]{\mathcal{F}_{#1}}
\newcommand{\qt}[2]{\mathbb{Q}_{#1}^t\left[#2 \right]}
\newcommand{\qbt}[2]{\mathbb{Q}_{#1}^{B,t}\left[#2 \right]}
\newcommand{\qrt}[2]{\mathbb{Q}_{#1}^{R,t}\left[#2 \right]}
\newcommand{\es}[2]{\mathbf{P}^{#1}\big(#2\big)}
\newcommand{\esbt}[2]{\mathbf{P}^{B,t}_{#1}\left[#2\right]}
\newcommand{\pct}[1]{\overline{\mathbb{P}}_x\left(\left. #1 \right| |N_t|>0 \right)}
\newcommand{\pctii}[1]{\mathbb{P}_x\left(\left. #1 \right| |N_t|>0 \right)}
\newcommand{\ect}[1]{\overline{\mathbb{P}}_x\left(\left. #1 \right| |N_t|>0 \right)}
\newcommand{\ectii}[1]{\mathbb{{P}}_x\left[\left. #1 \right| |N_t|>0 \right]}
\newcommand{\flm}{\frac{\lambda}{m-1}}
\newcommand{\pb}[2]{\bar{\mathbb{P}}^{#1}\big( #2 \big)} 
\newcommand{\pbu}[2]{\bar{\mathbb{P}}^{#1}\otimes \mu^U_t \big( #2 \big)} 
\newcommand{\giv}{\, \big| \,}


\section{Introduction}

This paper concerns branching diffusions in a bounded domain $D$ of $\R^d$. For us, these are processes in which individual particles move according to the law of some diffusion, are killed upon exiting the domain, and branch into a random number of particles (with distribution $A$, independent of position) at constant rate $\beta>0$. Whenever such a branching event occurs each of the offspring then stochastically repeats the behaviour of its parent, starting from the point of fission, and independently of everything else. The configuration of particles at time $t$ will be represented by the $D$-valued point process
$$\mathbb{X}_t = \{X_u(t)\,: \, u\in N_t\}, $$
where $N_t$ is the set of individuals alive at time $t$ (its size will be denoted by $|N_t|$). We write $\mathbb{P}^x$ for the law of the process initiated from a point $x\in D$.  Finally, we will always assume that the offspring distribution $A$ satisfies
\begin{equation}
\label{eqn::offspring} \mathbb{E}(A)=m>1 \;\;\;\; \text{and} \;\;\;\; \mathbb{E}(A^2)<\infty, \end{equation} and that the generator 
\begin{equation}
\label{eqn::generator}
L=\frac{1}{2} \sum_{i,j} \st_{x_j} (a^{ij}\st_{x_i})
\end{equation} of the diffusion is uniformly  elliptic (see Section \ref{sec::spec_diffusion}) with coefficients $a^{ij}=a^{ji}\in C^1(\bar{D})$ for $1\leq i,j\leq d$. 

It is known that such a system exhibits a phase transition in the branching rate: for large enough $\beta$ there is a positive probability of survival for all time, but for small $\beta$, including at criticality, there is almost sure extinction. The critical value of $\beta$ is equal to 
$\frac{\lambda}{m-1}$,
where $\lambda$ is the first eigenvalue of 
$-L$ on $D$ with Dirichlet boundary conditions (see (\ref{eqn::evalue})). The main goal of this paper will be to study the system at criticality and find a scaling limit for the resulting genealogical tree. This is the continuous plane tree that is generated purely by the birth and death times of particles in the system, and encodes no information about the spatial motion.

More precisely, for $t\ge 0$ we condition the critical branching diffusion to survive until time \emph{at least} $t$, and look at the associated genealogical tree $\mathbf{T}$, equipped with its natural distance $d$. Rescaling distances by a factor $t$ gives us a sequence of laws on random compact metric spaces: 
\begin{equation}\label{eqn::law_conditioned_trees}(\mathbf{T}_{t,x},d_{t,x})\overset{(\text{law})}{:=} (\mathbf{T}, \frac{1}{t} d) \text{ under } \mathbb{P}^x(\cdot\, |\, |N_t|>0).\end{equation}  We will prove that this sequence converges in distribution to a conditioned Brownian continuum random tree as $t\to \infty$, with respect to the Gromov-Hausdorff topology. Indeed, if we let $\mathbf{e}$ be a Brownian excursion conditioned to reach height at least $1$, and write $(\mathbf{T}_{\mathbf{e}},d_{\mathbf{e}})$ for the real tree whose contour function is given by $\mathbf{e}$, then we obtain the following result.

\begin{thm}\label{thm::crtconv} Suppose that $D\subset \R^d$ is a $C^1$ domain \footnote{\label{fn::ck} We say that $D\subset \R^d$ is a \{Lipschitz / $C^k$ / $C^{k,\alpha}$\} domain (for $k\in \Z_{\geq 0}$ and $\alpha\in [0,1]$) if, at each point $x_0\in \partial D$, there exists $r>0$ and a \{Lipschitz / $C^k$ / $C^{k,\alpha}$\} function $\gamma:\R^{d-1}\to \R$ such that relabelling and reorienting axes as necessary, $D\cap B(x_0,r)=\{x\in B(x_0,r): x^d>\gamma(x^1,\cdots, x^{d-1})\}$. } and that $L$ as in (\ref{eqn::generator}) is uniformly elliptic with coefficients $a^{ij}=a^{ji}\in C^1(\bar{D})$ for $1\leq i,j\leq d$. Further suppose that $A$ satisfies (\ref{eqn::offspring}), and that $\vp\in C^1(\overline{D})$ where $\vp$ is the first eigenfunction of $-L$ on $D$ (see Section \ref{sec::spec_diffusion}). Then, at the critical branching rate $\beta=\lambda/(m-1)$, and for any starting point $x\in D$,
	\[ (\mathbf{T}_{t,x}, d_{t,x}) \underset{t\to \infty}{\longrightarrow} (T_{\ee}, d_{{\ee}})\]
	in distribution, with respect to the Gromov--Hausdorff topology.
\end{thm}
\begin{remark0}\label{rmk::suffreg} 
One sufficient condition to ensure that the hypotheses of Theorem \ref{thm::crtconv} are satisfied is to assume that $D$ is $C^{2,\alpha}$ for some $\alpha\in [0,1]$ (see Lemma \ref{lem::c2a}). However, this is also satisfied in many other cases, so we leave the assumptions of Theorem \ref{thm::crtconv} as general as possible.	
\end{remark0}

On the way to proving Theorem \ref{thm::crtconv} we also obtain new proofs of several other results concerning critical branching diffusions, some of which are already known in various forms. The reason for detailing these proofs here is threefold: firstly, it allows us to pin down the regularity required on the domain $D$; secondly, it provides a new and somewhat more probabilistic approach to the theory, that we believe is interesting in its own right; and finally, the proofs serve to introduce many concepts and ideas that are crucial for the proof of Theorem \ref{thm::crtconv}.
\medskip

Let us first look at the phase transition. This result was originally proved by Sevast'yanov \cite{sevastyanov} and Watanabe \cite{watanabe}, in the case when $L$ is a constant multiple of $\Delta$. However, it has also been reworked and generalised since then. In \cite[Chapter 6]{ah}, a more general version of the result is given for branching Markov processes whose moment semigroup satisfies a certain criterion. One of the main examples discussed is when the process is a branching diffusion on a manifold with killing at the boundary. This is slightly more general than the set up of the present paper, in that the diffusion is on a manifold and the branching mechanism is allowed to be spatially dependent, however, a (fairly abstract) condition on the moment semigroup is required. In \cite{heringup} the criterion is shown to be satisfied, for example, if the manifold has $C^3$ boundary and the generator of the diffusion is uniformly elliptic with $C^2$ coefficients. Here we will prove the result under a weaker assumption on the domain and generator, but in our slightly more specific framework. Related results have also been shown in \cite{eng}, which studies local extinction versus local growth on compactly contained subsets of a (possibly infinite) domain $D$, and in \cite{hesse}, where $D$ is taken to be a bounded interval of the real line. 

\begin{thm}[\cite{sevastyanov}, \cite{watanabe}]
\label{phasetrans}
Let $D\subset \R^d$ be a bounded domain with Lipschitz boundary and suppose that $L$ and $A$ are as in Theorem \ref{thm::crtconv}. Then, for any starting position $x\in D$, if $\lambda$ is the first eigenvalue of $-L$ on $D$ with Dirichlet boundary conditions then,
\begin{itemize}
\item for $\beta > \frac{\lambda}{m-1}$ the process survives for all time with positive $\mathbb{P}^x$-probability.
\item for $\beta \leq \frac{\lambda}{m-1}$ the process becomes extinct $\mathbb{P}^x$- almost surely. 
\end{itemize} 
Moreover, if $\beta \leq \frac{\lambda}{m-1}$ then $\mathbb{P}^x(|N_t|>0)\to 0$ as $t\to \infty$, \emph{uniformly} in $D$. 
\end{thm}

The rest of the paper will focus on the behaviour of the system at criticality, starting with an asymptotic for the survival probability. The results of Theorems \ref{criticalsurvivalprob}, Theorem \ref{thm::convergenceconditionednt} and Corollary \ref{uniformparticle} have also been shown in \cite[Chapter 6]{ah} in the same framework discussed above (see also \cite{heringmulti,heringup}, and \cite{hering1d} for the one-dimensional case). Here we provide new proofs, which hold under fewer assumptions on the domain $D$, and which we can also modify to give key ingredients for the proof of Theorem \ref{thm::crtconv} (see for example Lemma \ref{lemma::average_lln}).

\begin{thm}
\label{criticalsurvivalprob}
Suppose that $D$ is $C^1$ and that $L$ and $A$ are as in Theorem \ref{phasetrans}. Then, in the critical case $\beta=\frac{\lambda}{m-1}$, for all $x\in D$ we have 
\begin{equation}
\label{asp}
\Prb{x}{|N_t|>0} \sim \frac{1}{t}\times \frac{2(m-1)\vp(x)}{\lambda \left(\mathbb{E}[A^2]-\mathbb{E}[A]\right) \int_D \vp(y)^3 \, dy }
\end{equation}
as $t \to \infty$. Here $\vp$ is the first eigenfunction of $L$ on $D$, normalised to have unit $L^2$ norm.
\end{thm}

This asymptotic then allows us to study the behaviour of the system when it is conditioned to survive for a long time, which is important for the proof of Theorem \ref{thm::crtconv}. One tool that we will use is a classical \emph{spine} change of measure, under which the process has a distinguished particle, the spine, which is conditioned to remain in $D$ forever (as in \cite{pin}). Along this spine, families of ordinary critical branching diffusions immigrate at rate $\frac{m}{m-1}\lambda$ according to a biased offspring distribution. Note that there is no extinction under this new measure, which we denote by $\mathbb{Q}^x$. We will prove that changing measure in this way is in fact somewhat close to conditioning on survival for all time, in the sense of the following proposition.

\begin{propn}\label{thm::conditionedsystem}
Assume the hypotheses of Theorem \ref{criticalsurvivalprob}. Then for any $T\geq 0$, $x\in D$ and $B \in \mathcal{F}_T$, where $\mathcal{F}$ is the natural filtration of the process, we have
\begin{equation}
\lim_{t\to \infty} \Prb{x}{B\giv |N_t|>0}=\mathbb{Q}^x(B). 
\end{equation}
\end{propn} 

To our knowledge, Proposition \ref{thm::conditionedsystem} does not appear in the existing literature, although the idea behind it is well-known: see \cite[Theorem 5]{harrissurvivalprobs}.
\\

Finally, we prove a Yaglom-type limit theorem for the positions of the particles in the system at time $t$, given survival. 

\begin{thm}
\label{thm::convergenceconditionednt}
For any measurable function $f$ on $D$ such that $\int_D f(x)^2 \vp(x) \, dx< \infty$, we have 
\[ \left( \left. t^{-1}\sum_{u\in N_t} f(X_u(t)) \right| |N_t|>0 \right) \to Z  \]
in distribution as $t \to \infty$, where $Z$ is an exponential random variable with mean 
\[\frac{\lambda \left(\mathbb{E}[A^2]-\mathbb{E}[A]\right) \ip{\vp}{f}_{L^2(D)}\int_D\vp^3}{2(m-1)}. \]
\end{thm}

One consequence of Theorem \ref{thm::convergenceconditionednt} (or rather its proof) is that it allows us to describe the limiting distribution of the particles in the system at time $t$, given survival. It turns out that this is the law with density $\vp$, normalised to be a probability distribution.

\begin{cor}
\label{uniformparticle}
For $t\ge 0$, let $\mu_t$ be the random point measure with law \[\mathcal{L}\left( \left.\frac{1}{|N_t|} \sum_{u\in N_t} \delta_{X_u(t)} \; \right| \; |N_t|>0 \right).\] That is, $\mu_t$ has the law of the empirical measure of particles alive at time $t$, given survival. Then for each $f$ as in Theorem \ref{thm::convergenceconditionednt}, we have that \[ \mu_t(f) \to \mu(f)\] in distribution, and hence in probability, as $t\to \infty$, where
\[ \mu(f)= \frac{\int_D \vp(x) f(x) \, dx}{\int_D \vp(x) \, dx }. \] 
\end{cor}

\subsection{Context}
It is interesting to note the analogy between Theorems \ref{phasetrans}-\ref{thm::convergenceconditionednt} and classical results from the theory of Galton--Watson processes. Indeed, for critical Galton--Watson processes, Kolmogorov \cite{kolmogorov} proved an asymptotic for the probability of survival up to time $n$:
\[ \mathbb{P}(Z_n>0) \sim \frac{c}{n}\]
where $Z_n$ is the population size at time $n$, and the constant depends on the variance of the offspring distribution. Moreover, Aldous \cite{crti,crtiii} and Duquesne and Le Gall \cite{crtconvgd} showed that if you condition a critical Galton--Watson process to reach a large generation or have a large total progeny, then you have a scaling limit for the resulting tree. This limit is in the Gromov--Hausdorff topology, after rescaling distances in the tree appropriately, and the limiting object is the Continuum Random Tree, \cite{crti}. In fact, this result can be extended to multitype Galton--Watson processes with a finite number of types, as in \cite{mie}, where the same scaling limit exists. Since a branching diffusion can be thought of as a limit of multitype Galton--Watson processes (considering the types to be positions and discretising the domain appropriately) it is reasonable to conjecture that such a process will have the same limiting genealogy when conditioned to survive for a long time.

On the other hand this result must be non-trivial, given what is known and expected for other types of domain. For example, \cite{kesten, nearcriticalabsorptionbbs, bbmabsorption, criticalabsorptionbbs, criticalabsorptionbbsparticles} have studied branching Brownian motion on the positive half line (with absorption at the origin) where each particle moves with a drift $-\mu$. In this set up there is a critical value of the branching rate $\beta$, equal to $\beta_c=\mu^2/2$, such that for $\beta\leq \beta_c$ extinction occurs with probability one. In \cite{criticalabsorptionbbs}, an asymptotic for the survival probability is calculated, which is very different from that of Theorem \ref{criticalsurvivalprob}. Moreover, results of \cite{bbmabsorption} in the \emph{near critical regime} suggest that the critical genealogy in this case should be closely related to the Bolthausen--Sznitman coalescent \cite{bscoa}. This is also related to non-rigourous physics predictions of Brunet, Derrida, Mueller and Munier, \cite{bdmm1, bdmm2}: see \cite{NBcoasurvey} for a survey of the area, further references and a general discussion of these predictions. This is of course not just a one-dimensional effect, as the behaviour will trivially be the same when the domain D is a half-space with a drift in the direction of the hyperplane. Hence Theorem \ref{thm::crtconv} is not likely to hold if we drop the boundedness assumption on the domain $D$. More generally, this raises the following question:

\begin{question}
If the domain $D$ is allowed to be unbounded, or very irregular, what other behaviours do we see appearing at criticality (whenever we can make sense of this notion)? 
\end{question}

In general it is an open, and we believe extremely interesting problem, to try and classify all the possible different behaviours that can occur at criticality depending on the geometry of the domain. 
\medskip

\subsection{Organisation of the Paper and Main Ideas}\label{subsec::outline} 
After setting up the relevant notations and preliminary theory, we begin with the proof of Theorem \ref{phasetrans}.
The main idea, which differs from the more analytic proofs given in \cite{sevastyanov, watanabe, ah}, is to exploit the existence of the martingale 
\begin{equation}\label{eqn::mgale_outline} (M_t)_{t\ge 0} := (\,\e^{(\lambda-\beta(m-1))t}\sum_{u\in N_t} \vp(X_t^i)\,)_{t\ge 0} \end{equation}
(also appearing in \cite{eng,hesse}) that arises naturally from the definition of the process. Since \\ $\sum_{u\in N_t} \vp(X_t^i)$ roughly tells us the size of the system at time $t$, and $M_t$ converges almost surely as $t\to \infty$, the behaviour of the exponential term in (\ref{eqn::mgale_outline}) governs the possible survival or extinction of the process. 

We then turn to the proofs of Theorems \ref{criticalsurvivalprob} to \ref{thm::convergenceconditionednt}. The proof of Theorem \ref{criticalsurvivalprob} proceeds by a combination of probabilistic arguments, and the analysis of a system of coupled ordinary differential equations. Naively, we expand the survival probability (as a function of $x$, for each fixed $t$) with respect to the orthonormal basis of $L^2(D)$ given by the eigenfunctions of $-L$. Then because the survival probability satisfies a certain partial differential equation (the FKPP equation for branching Brownian motion, \cite{mckean}) we get a family of coupled ODEs from the coefficients. In fact, we do not explicitly use that the survival probability satisfies this PDE (as we can derive the ODEs for the coefficients directly and avoid potentially complicated technical assumptions), but this is the motivation behind the proof. Unfortunately however, the system of ODEs is not immediately easy to analyse, and this is where the probabilistic line of reasoning comes into play. Changing measure using the martingale $(M_t)_{t\ge 0}$ (to get a spine characterisation of the system as discussed in the introduction) allows us to deduce that the survival probability actually just decays like $a(t) \vp(x)$ as $t\to \infty$, where $a(t)$ is the first coefficient in the expansion. Thus, our problem is reduced to the study of a single ODE.  From here elementary analysis, combined with some extra information obtained from the probabilistic arguments, yields the result. Proposition \ref{thm::conditionedsystem}, Theorem \ref{thm::convergenceconditionednt} and Corollary \ref{uniformparticle} then follow fairly straightforwardly. 

The remainder of the paper is devoted to the proof of Theorem \ref{thm::crtconv}. To do this, we take an i.i.d. sequence of critical processes, and concatenate the \emph{height functions} of their associated continuous genealogical trees. A \textbf{key idea} is to define a suitable analogue of the Lukazewicz path for Galton--Watson trees: that is, something that will approximate this concatenated height process well, and will converge after rescaling to a reflected Brownian motion. At first it seems too much to hope that such a precise combinatorial structure survives in this spatial context. However, it turns out that we can exploit the martingale $(M_t)_{t\ge 0}$ by ``exploring it in a different order''. Just as $(M_t)_{t\ge 0}$ roughly measures the size of our population when we let time evolve in the usual way, when we explore the genealogical tree in a ``depth first'' order and define a new process analogolously to $(M_t)_{t\ge 0}$, this process remains a martingale and, perhaps surprisingly, becomes a kind of spatial analogue of the height function. After strengthening Corollary \ref{uniformparticle}, we can prove that the quadratic variation of this new martingale is essentially linear, and thus obtain an invariance principle.

Of course, we have to prove that this new martingale is indeed a good approximation to the height process. This is one of the main difficulties, as the reversibility tools that are key to proving the analagous statement for the Lukasiewicz path in the Galton--Watson case are lost. Instead, we must use precise estimates, and a delicate ergodicity argument related to our spine change of measure. This is one of the reasons that our machinery from the proof of Theorem \ref{criticalsurvivalprob} is so essential. Tightness arguments then allow us to conclude.
\medskip

\noindent \textbf{Acknowledgements} I would particularly like to thank Nathana\"{e}l Berestycki, for suggesting this problem and for many helpful discussions. I am also grateful to the anonymous referee for numerous useful comments and suggestions.


\section{Preliminaries}

\subsection{Spectral Theory and Diffusions}\label{sec::spec_diffusion}

Let us first assume that $D\subset \R^d$ is a bounded domain satisfying a \emph{uniform exterior cone condition}. This means that: (1) $D$ is an open connected set of $\R^d$ with $|D|<\infty$ and boundary $\partial D$ ; and (2) there exist $r,\kappa>0$ such that $\forall \, y\in \partial D$, we can find $\eta\in \R^d$ with $|\eta|=1$ and 
$$ \{z\in B(y,r):\, \eta\cdot (z-y)>0 \text{ and } |\eta\cdot(z-y)|<\kappa|z-y|\}\subset D^c.$$
Such a condition is satisfied, for example, if $\partial D$ is Lipschitz, see eg. \cite[p.27]{davies}.  

Let
$$L=\sum_{i,j=1}^d \st_{x_i}(a^{ij}\st_{x_j}) $$
be a self-adjoint differential operator on $D$, which is \emph{uniformly elliptic}, meaning that there exists a constant $\theta>0$ such that for all $\xi\in \R^d$ and a.e. $x\in D$
\begin{equation*}\sum_{i,j=1}^n a^{ij}(x)\xi_i\xi_j \geq \theta |\xi|^2,\end{equation*} see \cite[\S 6.1]{evans}. We also assume that $a$ is symmetric, i.e. $a^{ij}=a^{ji}$, and that $a^{ij}\in C^1(\bar{D})$ for all $1\leq i,j\leq d$. 

We say that $\tilde{\vp} \in H_0^1(D)$ \footnote{We define $H_0^1(D)$ to be the closure of $C_c^\infty(D)$ (the space of infinitely differentiable functions with compact support strictly inside $D$) with respect to the norm $\|u||_{H^1(D)}:=\|u\|_{L^2} +\sum_{i=1}^d \|D^{x_i}(u)\|_{L^2}$, where $D^{x_i}u$ is the $i$th partial derivative of $u$ in the weak sense.} is an eigenfunction of $-L$ with Dirichlet boundary conditions, and associated eigenvalue $\tilde{\lambda}$, if 
\begin{equation}
	\label{eqn::evalue}
	\int_D \sum_{i,j=1}^d a^{ij}(x) \st_{x_i} \tilde{\vp}(x) \st_{x_j} v (x) \, dx = \int_D \tilde\lambda \tilde{\vp}(x)v(x) \, dx
\end{equation}
for every $v\in H_0^1(D)$, as in \cite[\S 6.3]{evans}. That is, $\tilde{\vp}$ is a weak solution of
$-Lu = \tilde{\lambda} u $ with zero boundary conditions. Given the assumptions made on $L$ and $D$, the following properties then hold (see  \cite[Theorem 1.6.8]{davies} and \cite[Theorem 9.30]{gilbargtrudinger}.): 

\begin{itemize}
	\item The eigenvalues of $-L$ are all real and can be written $0<\lambda:=\lambda_1< \lambda_2\leq \lambda_3 \cdots$ (repeated according to their finite multiplicity) with $\lambda_k \to \infty$ as $k\to \infty$.
	\item The associated eigenfunctions $\{\varphi_i\}_{i\geq 0}$  (normalised correctly) form an orthonormal basis of $L^2(D)$. Moreover, the first eigenfunction $\varphi_1=:\varphi$ is strictly positive in $D$, and $\varphi_i\in C(\bar{D})$ for all $i\geq 1$.
\end{itemize}

Now we consider the diffusion $(X(t))_{t\geq 0}$ associated to $L$ on $D$. This is the Markov process on $D\cup (\partial)$, where we identify the boundary $\st D$ of $D$ with the single isolated point $(\partial)$, such that: 
\begin{itemize}
	\item $X(t)$ evolves as a diffusion with generator $L$ for all $t<\tau^D:=\inf\{s\geq 0: X(s)\in \st D\}$; and 
	\item $X(t)=(\partial)$ for all $t\geq \tau^D$. 
\end{itemize}

Thus $(X(t))_{t\geq 0}$ is the diffusion with generator $L$, \emph{killed} or \emph{absorbed} upon hitting $\partial D$. 
We write $\mathbf{P}^x$ for its law when started from $x\in D$. Then by \cite[Theorems 2.1.4 and 2.3.6]{davies}, the function
\begin{equation}
	\label{eqn::kernel_decomp}
	p^D_t(x,y) := \sum_{n\geq 1} \exp(-\lambda_n t) \varphi_n(x)\varphi_n(y);\;\;\;\; t\in (0,\infty), \; x,y\in D
\end{equation}
is well defined as a uniform limit on $[\alpha,\infty)\times D \times D$ for any $\alpha>0$, and is the transition density of the process $X$, restricted to $(x,y)\in D\times D$. We also have the estimate, \cite[Corollary 3.2.8]{davies}
\begin{equation}\label{eqn::useless_kernel_bound} 
	0\leq p^D_t(x,y)\leq c t^{-d/2}
\end{equation}
for some constant $c>0$. 

In particular, for any $t>0$ and any $f\in L^1(D)$ we have 
\begin{equation}
	\label{eqn::exp_singleparticle}
	\int_D f(y) p^D_t(x,y) \, dy = \mathbf{P}^x(f(X(t))\I_{\{\tau^D>t\}}).
\end{equation}

The properties (\ref{eqn::kernel_decomp})-(\ref{eqn::exp_singleparticle}) of the killed diffusion and associated transition kernel $p_t^D$ above, are consequences of the fact that the symmetric Markov semigroup associated with the killed diffusion is \emph{ultracontractive} (see \cite[\S 2]{davies}) when $D$ satisfies a uniform exterior cone condition. In fact, if we assume some more regularity on the domain, then it will satisfy a certain stronger form of contractivity known as \emph{intrinsic ultracontractivity}, first defined in \cite{daviessimon}. Intrinsic ultracontractivity is satisfied by the semigroup of the killed diffusion, for example, if the domain $D$ is bounded and Lipschitz \cite[Theorem 1]{baniuc}. The key property of intrinsic ultracontractivity that we will use is the following. 

\begin{lemma}
	\label{lem::spineconvergence}
	Suppose that $D$ is a bounded Lipschitz domain (or more generally, a domain such that the semigroup of $X$ is intrinsically ultracontractive). If \[K^D_t(x,y):= \frac{\e^{\lambda t} p^D_t(x,y) \vp(y)}{\vp(x)}\] is the transition density for $(X(t))_{t\geq 0}$ conditioned to remain in $D$ for all time \cite{pin} then for any $\varepsilon>0$ there exists a constant $C_\varepsilon$ depending only on the domain such that 
	\[\left| \frac{K^D_t(x,y)}{\vp(y)^2} - 1 \right| \leq C_\varepsilon \e^{-\gamma t}\]
	for all $t>\varepsilon$ and $x,y\in D$, where $\gamma:=\lambda_2-\lambda_1>0$ is the \emph{spectral gap} for $-L$ on $D$.
\end{lemma}

\begin{proof}
	See for example \cite{daviessimon} or \cite[Equation (1.8)]{baniuc}. 
\end{proof}

We also have the following estimate:
\begin{lemma}
	\label{lem::iuc_kernel_estimate} Suppose we are in the set up of Lemma \ref{lem::spineconvergence}. Then for some constants $c_1,c_2$, we have 
	$$p_r^D(w,z)\leq c_1 r^{-d/2} \e^{-\frac{|w-z|^2}{c_2 r}},$$
	for all $z,w\in D$. 
\end{lemma}
\begin{proof}
	See \cite[Eq.(1.2)]{baniuc} or \cite[P.89]{davies}.
\end{proof}
We also have the following result, which gives us extra regularity on the eigenfunctions of $L$, if we assume some extra regularity on the domain. 

\begin{lemma}\label{lem::c2a}
	Suppose that the boundary of $D$ is $C^{2,\alpha}$ for some $\alpha>0$ and that $L$ is a generator satisfying the conditions assumed throughout this section. Let $\{\varphi_i\}_{i\geq 1}$ be an $L^2$ orthonormal basis of eigenfunctions for $-L$. Then $$\varphi_i\in C^1(\bar{D})$$ for all $i\geq 1$. 
\end{lemma}

\begin{proof}
	\cite[Theorem 6.31]{gilbargtrudinger}
\end{proof}

\subsection{Branching diffusions}\label{sec::bds}
As stated in the introduction, we can view a branching diffusion in $D\subset \R^d$ as a point process 
$$\mathbb{X}_t=\{X_u(t): u\in N_t\}$$ taking values in $D$. This is often all that we will need to speak about, but since we are eventually interested in the genealogy of such processes, it will be helpful at various points to to view them as elements of a larger state space: the space of \emph{marked trees}. The set up described in this section closely follows \cite{spineapproach,manytofew}.

We begin by recalling the \emph{Ulam--Harris labelling} system. Let $$\Omega:= \{\emptyset\}\cup \bigcup_{n\in \N} (\N)^n $$ be the set of finite labels on $\N=\{1,2,3...\}$. A subset $T\subset \Omega$ is called a \emph{tree} if: 
\begin{itemize}
	\item $\emptyset\in T$;
	\item $u,v\in \Omega$ and $uv\in T$ implies that $u\in T$; and
	\item for all $u\in T$ there exists $A_u\in \N \cup \{0\}$ such that $uj\in T \iff 1\leq j \leq A_u$. 
\end{itemize} 
We will refer to elements $u\in T$ as \emph{particles} or \emph{individuals} in $T$. We think of the element $\emptyset$ as representing an \emph{initial ancestor}, and individuals $u\in T$ as describing its descendants. For example, if $u\in T$ is given by the label $(2,1)$ then $u$ would be the first child of the second child of $\emptyset$. For  $u,v\in \Omega$ we write $uv$ for the concatenation of the words $u$ and $v$, so for example if $u=(1,2,3),v=(2,1)$, then $uv=(1,2,3,2,1)$. We also set $u\emptyset=\emptyset u = u$ for all $u\in \Omega$. We say that $v$ is an \emph{ancestor} of $u$ (written $v\prec u$) if there exists $w\in \Omega$ such that $vw=u$, and write $|u|$ for the length (or \emph{generation}) of $u$, where $|u|=n$ if $u\in \N^n$. Then the above \emph{tree} condition simply means that: $T$ has an initial ancestor or \emph{root} $\emptyset$; $T$ contains all of the ancestors of all of its individuals; and finally, each individual $u\in T$ has a finite (possibly $0$) number $A_u$ of children, labelled in a consecutive way. We write $\mathbb{T}$ for the set of trees.

We will want to consider \emph{marked trees}, where the marks will correspond to the behaviour of particles in our branching diffusion. If we have a tree $T \in \mathbb{T}$, we will mark each $u\in T$ with a lifetime $l_u\in [0,\infty)$, and a motion in $D$, 
$$X_u: [s_u-l_u, s_u)\to D,$$
where $s_u = \sum_{v\leq u} l_v$ is the \emph{death} time of the particle $u$. 

We write 
$$ \mathcal{T} := \{(T,l, X)=(T,(l_u)_{u\in T},(X_u)_{u\in T}):\, T \in \mathbb{T}, \text{ and } l_u\in [0,\infty) \, , \, X_u: [s_u-l_u, s_u)\to D \text{ for all } u\in T\}$$
for the set of all marked trees on $D$. With an abuse of notation, if we have a marked tree $(T,l, X)$ and $u\in T$ we will also sometimes extend the definition of $X_u$ to the whole of the ancestry of $u$. That is we will set $X_u(t)$ to be equal to $X_v(t)$ if $t\notin [s_u-l_u,s_u)$, and $v$ is the unique ancestor of $u$ alive at time $t$. If $u$ has no children, we write $X_u(t)=(\dagger)$ for all $t\geq s_u$, where $(\dagger)$ is an additional cemetery point that we introduce for use later on. Finally, we write $$N_t:=\{u\in T: t\in [s_u-l_u, s_u)\}$$ for the set of particles alive in $T$ at time $t$, and let $|N_t|$ be the number of such particles. As in the introduction, we let 
$$\mathbb{X}_t := \{ X_u(t): \, u\in N_t\}$$ be the point process on $D$ corresponding to the marked tree $(T, l, X)$.

\subsubsection{Probability measures on marked trees}

Let $(\mathcal{F}_t)_{t\geq 0}$ be the filtration on the space of marked trees defined by 
$$\mathcal{F}_t:= \sigma\big( \{(u,X_u, l_u): \, s_u\leq t\}\cup\{(u,X_u(s)): s\in [s_u-l_u,t],\, t\in [s_u-l_u, s_u) \} \big) \;\; \text{for} \; t\ge 0 $$
and set $$\mathcal{F}_\infty=\sigma(\cup_{t\geq 0} \mathcal{F}_t).$$
Then $(\mathcal{F}_t)_{t\geq 0}$ is the natural filtration associated with the point process $(\mathbb{X}_t)_{t\geq 0}$. We let $\mathbb{P}^x$ be the probability measure on $(\mathcal{T}, \mathcal{F}_\infty)$ such that:
\begin{itemize}
	\item $X_\emptyset(t)$ evolves under the law $\mathbf{P}^x$ described in Section \ref{sec::spec_diffusion} for $0\leq t \leq l_{\emptyset}$, where $l_\emptyset= \tau^D_\emptyset \wedge \upsilon_\emptyset $, $\tau^D_\emptyset$ is the first time that $X_\emptyset$ hits $\partial D$ and $\upsilon_\emptyset \sim \text{Exp}(\beta)$ is an exponential random variable independent of $X_\emptyset$. 
	\item $A_\emptyset = 0$ on the event that $l_\emptyset=\tau^D_\emptyset$. On the complementary event, $A_\emptyset$ is distributed as an independent copy of the offspring distribution $A$.
	\item At any branching event where a positive number of children are born, all children repeat stochastically, and independently, the behaviour of their parent, starting from the point of fission.
\end{itemize}

That is, $\mathbb{P}^x$ is the law of the system described in the introduction, with offspring distribution $A$ and constant branching rate $\beta>0$. 

\subsubsection{The many-to-few formulae}

One particularly useful property of the branching diffusions considered in this section are the so-called \emph{many-to-few} formulae, which allow one to calculate certain expectations for the system with relative ease. We state the two simplest cases here; for the more general formula, see for example \cite[Lemma 1]{manytofew}. 

\begin{lemma}[Many-to-one]
	\label{lem::many_to_one}
	Suppose that $f$ is a measurable function on the Borel sets of $D$. Then 
	$$\mathbb{P}^x \big(\sum_{u\in N_t} f(X_u(t))\big)=\e^{(m-1)\beta t} \mathbf{P}^x(f(X(t))\nh) $$
	where we recall that $m=\mathbb{E}(A)$.
\end{lemma}

\begin{lemma}[Many-to-two]
	\label{lem::many_to_two}
	Suppose that $f$ and $g$ are measurable functions on the Borel sets of $D$. Then 
	\begin{align*}& \mathbb{P}^x\big(\sum_{u\in N_t} f(X_u(t))\sum_{v\in N_t} g(X_v(t))\big) \\
		& =\e^{(m-1)\beta t}\mathbf{P}^x\big(f(X(t))g(X(t))\I_{\{\tau^D>t\}}\big)+\beta (\mathbb{E}(A^2)-\mathbb{E}(A)) \int_0^t \e^{(2t-s)(m-1)\beta} \mathbf{P}^x\big( \I_{\{\tau^D>s\}} q(X(s),t-s)\big) \, ds \end{align*}
	where $q(y,r)=\mathbf{P}^y(f(X(r))\I_{\{\tau^D>r\}})\mathbf{P}^y(g(X(r))\I_{\{\tau^D>r\}})$
	for $y\in D$ and $r\geq 0$.
\end{lemma}

For the proof of the above lemmas, see \cite[Lemma 1]{manytofew}, which is stated in a more general setting. For an explanation of how this general statement gives the lemmas above, see \cite[\S 4.1, \S4.2]{manytofew}.

\subsubsection{The continuous genealogical plane tree}\label{sec::trees}

If we have a marked tree $(T,l, X)\in \mathcal{T}$ corresponding to a branching diffusion, we will also want to associate with it a \emph{continuous} genealogical tree $\mathbf{T}$. This tree (which we emphasise is \emph{different} from $T$) is the main object of Theorem \ref{thm::crtconv}: it is the plane tree with branch lengths given by the lifetimes of particles in the system. 

We first need to give a few definitions. A metric space $(\mathbf{T},d)$ is said to be a real tree if, for all $v,w\in \mathbf{T}$ the following two conditions hold \cite{legallrandomtrees}:
\begin{enumerate}[(1)]
	\item There exists a unique isometric map $\phi_{v,w}:[0,d(v,w)]\to \mathbf{T}$ with $\phi_{v,w}(0)=v$ and $\phi_{v,w}(d(v,w))=w$.
	\item Any continuous injective map $[0,1]\to \mathbf{T}$ that joins $v$ and $w$ has the same image as $\phi_{v,w}$.
\end{enumerate}
One way to define is a real tree is the following: take a continuous function $C:[0,\infty)\to [0,\infty)$ with $C(0)=0$ and define a ``distance" function on $[0,\infty)\times [0,\infty)$ by 
\[d_C(s,t)=C(s)+C(t)-2\min_{r\in[s,t]}C(r)\]
whenever $s\leq t$. It is easy to verify that this defines a pseudometric on $[0,\infty)$. Thus, quotienting by the equivalence relation $\sim$ that identifies points with $d_C(\cdot,\cdot)=0$ we obtain a metric space $(\mathbf{T}_C,d_C) := ([0,\infty)/\sim, d_C).$ One can prove, see for example \cite{legallrandomtrees}, that this metric space is a real tree. The function $C$ is called the \emph{contour function} of the tree. 

In our set up, if we have a marked tree $(T,l,X)\in \mathcal{T}$ we let $(\mathbf{T}(T,l,X),d(T,l,X))=(\mathbf{T}_{C(T,l,X)},d_{C(T,l,X)})$ be the real tree with contour function $(C(T,X,l)(t))_{t\geq0}=:(C(t))_{t\geq 0}$ described as follows. Let $\{\emptyset=:u_0, u_1, u_2, \cdot... u_{|T|}\}$ be the set of labels of $T$ in \emph{depth first} order. This is the ordering on $T$ such that $u$ is less than $v$ iff at the first coordinate where the labels of $u$ and $v$ differ, the coordinate of $u$ is less than that of $v$. For any two individuals $w,w'\in T$ we let $w\wedge w'$ denote their most recent common ancestor: that is, the (unique) $u$ with $|u|$ largest, such that $u\prec w$ and $u\prec w'$. We can then define, for $j\geq 0$ such that $u_j\nprec u_{j+1}$: \begin{itemize} \item $r_j:=s_{u_j}-(s_{u_j\wedge u_{j+1}}-l_{u_j\wedge u_{j+1}})$ to be the length of time between the birth time of $u_j\wedge u_{j+1}$ and the death time of $u_j$, and
	\item  $r_j':=l_{u_{j+1}}-(s_{u_j\wedge u_{j+1}}-l_{u_j\wedge u_{j+1}})$ to be the length of time between the birth  time of $u_{j} \wedge u_{j+1}$ and the birth time of $u_{j+1}$. 
\end{itemize} We set $r_j=r_j'=0$ if $u_j\prec u_{j+1}$. Let $R_j=\sum_{k<j} l_k + r_k + r_k'$, and for $t\in [0,R_{j+1}-R_j)$  set 
\begin{equation*}C(t+R_j)-C(R_j)= \begin{cases} 
		t & \text{if }t\in [0,l_j) \\
		2l_j-t & \text{if }t\in [l_j,l_j+r_j) \\
		- 2r_j + t & \text{if }t\in [l_j+r_j, l_j+r_j+r_j').\\
\end{cases} \end{equation*}
That is: $C(t)$ is positive and linear with unit speed on $[R_j,R_j+l_j)$ and $[R_j+l_j+r_j,R_j+l_j+r_j+r_j')$, and is negative and linear with unit speed on $[R_j+l_j, R_j+l_j+r_j)$. Finally, for $t\in [R_{|T|}, R_{|T|}+l_{|T|}+s_{|T|}]$ we let $C(t)$ interpolate linearly between $C(R_{|T|})$ and $0$. The definition of the function $C$ is probably clearest from a picture: we draw a tree with branch lengths corresponding to lifetimes of individuals in the system, and traverse it (with backtracking) at speed one. $C(t)$ measures how ``high" we are in the tree at time $t$. 

\begin{figure}[h]
	\centering
	\includegraphics[scale=0.5]{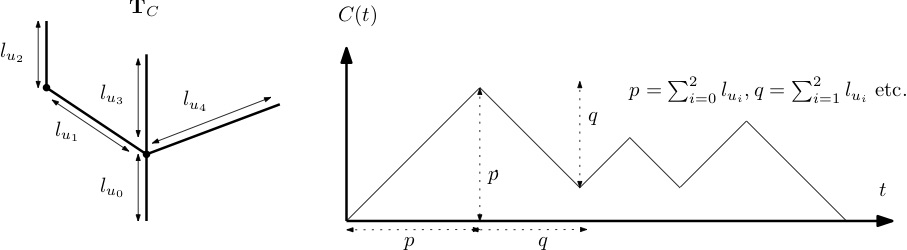}
	\caption{An example of the continuous tree $\mathbf{T}=\mathbf{T}_C$ generated by a branching diffusion. Here $u_0=\emptyset, u_1=(1), u_2=(1,1), u_3=(2)$ and $u_4=(3)$. Every branch of $\mathbf{T}$ corresponds to a particle $u\in T$, and this branch has length $l_u$. }
\end{figure}

\subsection{Martingales}\label{sec::mgales}

Suppose that $D\subset \R^d$ is bounded satisfying a uniform exterior cone condition, and $\mathbf{P}^x$, $(X(t))_{t\geq 0}$, $\tau^D$ and $p_t^D$ are as in Section \ref{sec::spec_diffusion}. Then by (\ref{eqn::kernel_decomp}), and using the fact that the eigenfunctions $\{\vp_i\}$ form an orthonormal basis of $L^2(D)$ we see that for any $f\in L^2(D)$ and $t>0$
\begin{equation}
	\label{eqn::f_decomp} \es{x}{f(X(t))\nh} = \sum_1^\infty \exp(-\lambda_i t) \vp_i(x) \ip{\vp_i}{f}
\end{equation}
where we write $\ip{f}{g}:=\int_D f(x)g(x) \, dx$ here and throughout the paper. In particular, 
\begin{equation} \label{eqn::exp_single_particle_martingale} 
	\mathbf{P}^x \left[\vp_i(X(t))\nh\right]=\e^{-\lambda_i t} \vp_i(x) \end{equation}
for all $x \in D.$ One consquence of this is that the process 
\begin{equation}
	\label{eqn::single_particle_mgale}
	\big(\exp(\lambda t)\vp(X(t))\nh\big)_{t\ge 0}
\end{equation} is a (positive) martingale under $\mathbf{P}^x$. Furthermore, this \emph{single-particle} martingale implies the existence of a martingale for the entire branching diffusion under $\mathbb{P}^x$. Indeed, a straightforward application of the Markov property for the branching diffusion and Lemma \ref{lem::many_to_one} yields the following.

\begin{lemma}
	\label{martingales}
	The process
	\[\left(M_t = \e^{(\lambda-\beta(m-1))t}\sum_{u\in N_t} \vp(X_u(t)) \right)_{t\ge 0}\]
	is a positive martingale under $\mathbb{P}^x$, for each $x\in D$.
	It therefore converges $\mathbb{P}^x-$almost surely as $t\to \infty$ to an almost surely finite limit $M_\infty$.
\end{lemma}

This martingale is the natural analogue of the well-known martingale $(Z_n/m^n)_{n\in \N}$ for Galton--Watson processes (with offspring mean $m$). Variants of the martingale for general branching processes have been studied extensively in the literature: see for example \cite{Biggins, Lyons} for the branching random walk case and \cite{chauvinrouault, Kyprianou, eng, harrissurvivalprobs, hesse} for branching Brownian motion, among many others. 

\subsection{Spine theory}
It turns out that a helpful approach in many parts of this paper will be to study the behaviour of the system under a change of measure. Precisely, the change of measure defined by the $\mathbb{P}^x$ martingale $(M_t)_{t\ge 0}$ from the previous section. To give a useful description of this, we need to view our process on a yet larger state space: the space of \emph{marked trees with spines}. This is a classical technique first introduced in \cite{chauvinrouault}, and since used extensively by many authors \cite{hesse,spineapproach,spineroberts,manytofew}. For a thorough exposition of the subject, we refer the reader to \cite{spineapproach}, and the overview in this section will closely follow that given in \cite{spineapproach,manytofew}.

Suppose we have a marked tree $(T,l,X)\in \mathcal{T}$. A spine $\psi$ on $(T,l,X)$ is a subset of $T\cup(\dagger)$ (where we recall that ($\dagger$) is an isolated cemetery point) such that:
\begin{itemize}
	\item $\emptyset\in \psi$ and $|\psi \cap (N_t \cup (\dagger))|=1$ for all $t$;
	\item $v\in \psi$ and $u\prec v$ implies that $u\in \psi$; and 
	\item if $v\in \psi$ and $A_v>0$, then there exists a unique $1\leq j \leq A_v$ with $vj\in \psi$. If $A_v=0$, then $\psi\cap N_t=\emptyset$ for all $t\geq s_v$. 
\end{itemize}
We write $$\tilde{\mathcal{T}}:=\{(T,l,X,\psi): (T,l,X) \in \mathcal{T}, \text{ and } \psi \text{ is a spine on } (T,l,X)\}$$ for the space of marked trees with spines. Given $(T,l,X,\psi)\in \tilde{\mathcal{T}}$ we let $\psi_t:=v$ be the unique element $v\in \psi \cap (N_t\cup (\dagger))$ and write $\xi_t:=X_v(t)$ for its position. With a slight abuse of notation, we set $X_v(t)=(\dagger)$ if $\xi_t=(\dagger)$ and also in this case say that $u\prec \psi_t$ if $u\in N_s$ and $u=\psi_s$ for some $s<t$.

\subsubsection{Filtrations}\label{sec::filtrations}

There are several different filtrations we can place on $\tilde{\mathcal{T}}$. We give brief descriptions of each of these below : see \cite{spineroberts} for more rigorous definitions.
\begin{itemize}
	\item $(\mathcal{F}_t)_{t\geq 0}$ is the natural filtration of the branching process as before, and does not contain any information about spines. We write $\mathcal{F}_\infty := \sigma (\cup_{t\geq 0}\Fs{t})$.  
	\item $(\tilde{\Fs{t}})_{t\geq 0}$ is the natural filtration of the branching process plus the spine. We write $\tilde{\Fs{\infty}}:=\sigma(\cup_{t\geq 0} \tilde{\Fs{t}})$. 
	\item $(\mathcal{G}_t)_{t\geq 0}:=(\sigma(\xi_s: s<t))_t$ is the filtration generated by the spatial motion of the spine. We write $\mathcal{G}_\infty:=\sigma (\cup_{t\geq 0} \mathcal{G}_t).$
	\item $(\tilde{\mathcal{G}}_t)_{t\geq 0}:= \sigma(\mathcal{G}_t\cup \{v\in \psi_s: 0\leq s \leq t\}\cup \{A_u: u<\psi_t\} )$ is the filtration that knows everything about the spine until time $t$: which individuals are in the spine, their motions, fission times, and family sizes at fission times along the spine. We write $\tilde{\mathcal{G}}_\infty := \sigma(\cup_{t\geq 0}\tilde{\mathcal{G}}_t)$. 
\end{itemize}

\subsubsection{Probability measures}

We first want to define the probability measure $\tilde{\mathbb{P}}^x$ on $(\tilde{\mathcal{T}},\tilde{\mathcal{F}}_\infty)$ under which, informally speaking, the law of the tree $(T,l,X)$ is the same as under  $\mathbb{P}^x$, and then the spine is chosen by picking one of the children uniformly at random at every branching event. More rigorously, if $Y$ is an $\tilde{\mathcal{F}}_t$-measurable random variable, then $Y$ can be written \cite{manytofew} as 
\begin{equation}
	\label{eqn::decomp_tildefmeasY}
	\sum_{v\in N_t\cup(\dagger)}Y(v) \I_{\{\psi_t=v\}} \end{equation}
where $Y(v)$ is measurable with respect to $\mathcal{F}_t$. Given this representation, we define the measure $\tilde{\mathbb{P}}^x$ on $(\tilde{\mathcal{T}},\tilde{\mathcal{F}}_\infty)$ by setting for each $\tilde{\mathcal{F}}_t$-measurable $Y$:
\begin{equation}
	\label{eqn::def_ptilde}
	\tilde{\mathbb{P}}^x(Y)=\mathbb{P}^x\big( \sum_{v\in N_t}( Y(v)\prod_{u<v} \frac{1}{A_u}) + Y((\dagger))\sum_{w\in D_t} \prod_{u<w} \frac{1}{A_u} \big),
\end{equation}
where $D_t:=\{w\in \cup_{s\leq t}N_s : A_w=0\}$. Note that $\tilde{\mathbb{P}}^x|_{\mathcal{F}_\infty}= \mathbb{P}^x$. 

\subsubsection{Change of measure}\label{sec::com}

Recall from Section \ref{sec::mgales} that $\e^{\lambda t}\frac{\vp(X(t))}{\vp(x)} \nh$ defines a mean-one martingale under $\mathbf{P}^x$. This implies, see \cite{spineapproach}, that
$$\zeta_t:= \I_{\{A_{\psi_s}>0 \; \forall s<t \}}\nhe\e^{(\lambda -\beta(m-1))t} \frac{\vp(\xi_t)}{\vp(x)} \prod_{v<\psi_t} A_v , $$
where $\tau^D_{\xi}$ is the first time that $\xi$ leaves the domain $D$,
is a mean-one $\tilde{\Fs{t}}$-martingale under $\tilde{\mathbb{P}}^x$. In fact, it is easy to check that $
\frac{M_t}{\vp(x)}= \tilde{\mathbb{P}}^x\big(\zeta_t\, |\, \mathcal{F}_t\big).$
Thus, if we define a new probability measure $\tilde{\mathbb{Q}}^x$ on $(\tilde{\mathcal{T}}, \tilde{\mathcal{F}}_\infty)$ via the martingale change of measure 
\begin{equation}
	\label{eqn::com_enhanced}
	\left.\frac{d\tilde{\mathbb{Q}}^x}{d\tilde{\mathbb{P}}^x}\right|_{\tilde{\mathcal{F}}_t} := \zeta_t,
\end{equation}
then we have, defining $\mathbb{Q}^x:=\tilde{\mathbb{Q}}^x|_{\mathcal{F}_\infty}$, that 
\begin{equation}
	\label{eqn::com_simple}
	\left.\frac{d\mathbb{Q}^x}{d\mathbb{P}^x}\right|_{\mathcal{F}_t} := \frac{M_t}{\vp(x)}. 
\end{equation}
We have the following description for how the branching diffusion with a distinguished spine behaves under $\tilde{\mathbb{Q}}^x$ (see for example \cite{chauvinrouault} or \cite{spineapproach}): 
\begin{itemize} 
	\item we begin with one particle at position $x$, which is the spine particle;
	\item the spine particle evolves as if under the changed measure 
	\begin{equation}\label{eqn::com_singleparticle}
		\left.\frac{d\mathbf{Q}^x}{d\mathbf{P}^x}\right|_{\mathcal{G}_t}:= \nhe \e^{\lambda t}\frac{\vp(\xi_t)}{\vp(x)};
	\end{equation}
	\item the spine particle branches at rate $m\beta$ and is replaced by a number of children having the size-biased distribution $\hat{A}$, where $$\mathbb{P}(\hat{A}=k)=\frac{k}{m}\mathbb{P}(A=k); \text{ and}$$
	\item given that there are $k$ children born at such a branching event, one is chosen uniformly at random to be the spine and stochastically repeats the behaviour of its parent. Non-spine particles initiate independent branching diffusions with law $\mathbb{P}^y$, from the point $y$ of fission.
\end{itemize}

We now make a number of remarks about this. Firstly, note that the spine particle under $\tilde{\mathbb{Q}}_x$ never leaves the domain $D$, and that it always has a positive number of children. This means that the branching diffusion under $\tilde{\mathbb{Q}}$ never becomes extinct, and that the spine particle is never in the cemetery state. Also note that under the change of measure (\ref{eqn::com_singleparticle}), the spine particle evolves as a diffusion with transition kernel $K^D_t(x,y)$: the kernel we defined in Lemma \ref{lem::spineconvergence}. Its motion is that of the diffusion $X(t)$ under $\mathbf{P}^x$, but \emph{conditioned to remain in $D$ for all time}, see \cite{pin}. In particular, Lemma \ref{lem::spineconvergence} tells us that if the domain $D$ is Lipschitz, its position converges quickly to an equilibrium distribution with density $\vp^2$. Finally, we record that (by an easy calculation):
\begin{equation}
	\label{eqn::spine_prob}
	\tilde{\mathbb{Q}}^x(v=\psi_t | \mathcal{F}_t)= \frac{\vp(X_v(t))}{\sum_{u\in N_t}\vp(X_u(t))}
\end{equation}
for $v\in N_t$. 

\section{The Phase Transition}
In this section we will provide a proof of Theorem \ref{phasetrans}. 
\\

\begin{proof2}{Theorem \ref{phasetrans}}
	First suppose that $\beta>\lambda/(m-1)$. Then we know by Lemma \ref{martingales} that $M_t \to M_\infty$ almost surely. Thus, by definition of $M_t$, we can conclude the proof in this case as soon as we can show that $\E{x}{M_\infty>0}>0$. However, for this it is enough to show that $(M_t)_t$ is uniformly integrable, and in fact, an elementary calculation using Lemma \ref{lem::many_to_two} gives that if $\beta>\lambda/(m-1)$, then $(M_t)_t$ is uniformly bounded in $L^2(\mathbb{P}^x)$. We leave this calculation to the reader. 
	
	Next suppose that $\beta<\lambda/(m-1)$. We write
	\begin{equation}
		\label{eqn::subcritical_survival_prob}
		\E{x}{|N_t|>0}\leq \E{x}{|N_t|}=\e^{(m-1)\beta}\mathbf{P}^x(\tau^D>t),\end{equation}
	where the second equality comes from Lemma \ref{lem::many_to_one}. Observe that by Lemma \ref{lem::spineconvergence}, we also have the existence of a constant $C$, depending only on $D$, such that 
	\begin{equation}
		\label{eqn::exp_kernel_bound}
		p_t^D(x,y)\leq C\vp(x)\vp(y)\e^{-\lambda t}
	\end{equation}
	for all $t\geq 1$. Since $\mathbf{P}^x(\tau^D>t)=\int_D p^D_t(x,y) \, dy$, this implies that when $\beta<\lambda/(m-1)$, the right hand side of (\ref{eqn::subcritical_survival_prob}) converges to $0$ as $t\to \infty$ (uniformly in $x\in D$). Hence we have almost sure extinction in this case. 
	
	Finally, we deal with the critical case $\beta=\lambda/(m-1)$. We make use of the following lemma, which can be found in \cite[Lemma 2.1]{watanabe}.
	
	\begin{lemma}[\cite{watanabe}]
		\label{0orinfty}
		For all $x\in D$ 
		\[\Prb{x}{|N_t| \to 0\text{ or } |N_t| \to \infty \text{ as } t\to \infty}=1.\]
	\end{lemma}
	
	The proof we give is the same as that in \cite{watanabe} (we include it only to show that it still works with our current assumptions on $L$ and $D$.) 
	\begin{proof}
		Since $|N_t|$ is integer-valued, and $\{|N_t|=0\}\Rightarrow \{|N_r|=0 \, \forall r\geq t\}$, it is sufficient to prove that 
		$\Prb{x}{|N_t|=k \; \text{i.o.}}=0$
		for every $k\in \mathbb{N}$. Fix $k$ and define a sequence of hitting and leaving times $(L_n,H_n)_{n\geq 1}$, by letting $L_1$ be the first time $t$ that $|N_t|\ne 1$, and $H_1$ be the first time (necessarily after $L_1$) that $|N_t|=k$. Then inductively, let $L_n$ be the first time after $H_{n-1}$ that $|N_t|\ne k$, and $H_n$ the first time after $L_n$ that $|N_t|=k$. We have to show that $\Prb{x}{H_n<\infty}\to 0$ as $n\to \infty$. 
		Set 
		$$\upsilon=\inf_{y\in D} \mathbf{P}^y(\tau^D\leq 1) >0$$
		which is strictly positive by Lemma \ref{lem::spineconvergence}, and let $p>0$ be the probability that an $Exp(\beta)$ random variable is bigger than $1$. 
		Then we have that  $\Prb{x}{H_1<\infty} \leq \mathbb{P}^x(|N_{L_1}|> 0) \leq (1-p\upsilon)$ (since the probability that the initial particle exits the domain without branching is greater than or equal to $p\upsilon$). Then inductively, using the Markov property at each time $H_j$, we see that $\Prb{x}{H_n < \infty} \leq \Prb{x}{|N_{L_n}|>0}\le (1-p\upsilon)(1-(p\upsilon)^k)^{n-1}\to 0$. This completes the proof.
	\end{proof}
	
	With this in hand, to prove that we have almost sure extinction in the critical case, it is enough to show that $\mathbb{P}^x(|N_t|\to \infty \text{ as } t\to \infty)=0$. To do this, we use the fact that $M_t=\sum_{u\in N_t}\vp(X_u(t))$ converges almost surely to $M_\infty<\infty.$ Then the idea is that if $|N_t|$ is very big then $M_t$ should be big as well, and this will give a contradiction. 
	
	Consider the event $A_{t,K,R}:=\{\{|N_t|\geq R\}\cap \{M_t\le K\}\}$, for $R,K\geq 1$. On this event, by definition of $M_t$, it must be the case that $N_t$ is non-empty, and that $\vp(X_u(t))\leq K/R$ for some $u\in N_t$. In other words, we have $A_{t,K,R} \subset \{\sum_{u\in N_t}\I_{\{\vp(X_u(t))\leq K/R\}}\geq 1\}$. We compute that 
	\begin{equation*}
		\mathbb{P}^x\big(\sum_{u\in N_t}\I_{\{\vp(X_u(t))\leq K/R\}}\geq 1)\leq \mathbb{P}^x\big( \sum_{u\in N_t}\I_{\{\vp(X_u(t))\leq K/R\}}\big) \leq \e^{(m-1)\beta t} \int_{y\in D:\vp(y)\leq K/R }p^D_t(x,y)\, dy 
	\end{equation*}
	where the second inequality follows from Lemma \ref{lem::many_to_one}. Finally, using (\ref{eqn::exp_kernel_bound}),  we see that this is less than $cK/R$ whenever $t\geq 1$, for some constant $c$ depending only on $D$. 
	
	Now we are ready to conclude. We have 
	\begin{align}\label{eqn::critical_phasetrans_proof}
		\E{x}{|N_t|\to \infty} & = \E{x}{\{|N_t|\to \infty\}\cap \{M_\infty<\infty\}} \\ \nonumber 
		& = \E{x}{\cup_K \cap_R \cup_T \cap_{t\geq T} A_{t,K,R}},
	\end{align}
	and we have just shown that $\E{x}{\cap_{t\geq T}A_{t,K,R}}\leq \E{x}{A_{T,K,R}}\leq cK/R$ for all $T\geq 1$. Taking limits on the right hand side of (\ref{eqn::critical_phasetrans_proof}) we see that $\E{x}{|N_t|\to \infty}=0$.

	To complete the proof of Theorem \ref{phasetrans} we must show that the decay of the survival probability in the critical case is uniform in $D$. 
	To do this we will show that if we set 
	\begin{equation}
		\label{eqn::def_u}
		u(x,t):= \begin{cases}
			\mathbf{P}^x(|N_t|>0) & x\in D \\
			0 & x\in \st D\\
		\end{cases}
	\end{equation} 
	then $u(x,t)$ is a continuous function in $\bar{D}$ for all $t>0$. Then since the $u(x,t)$ are decreasing in $t$ and converge to the continuous function $0$ as $t\to \infty$ for each $x \in \bar{D}$, by Dini's theorem \cite[Theorem 7.13]{rudin}), the decay must indeed be uniform. 
	
	To prove this, first we fix $t>0$ and $x\in D$, and show that $u(x,t)$ is continuous at $x$. Indeed, given any $\eps>0$, we can pick $T>0$ such that if $Z\sim \text{Exp}(\lambda/(m-1))$ then $\mathbb{P}(Z<T)\leq \eps/4$. This means that for any $y\in D$, by conditioning on whether or not the first branching time is less than or equal to $T$, we have 
	\begin{equation*}
		|u(x,t)-u(y,t)|\leq \eps/2 + \int_D |p^D_T(x,w)-p^D_T(y,w)|u(w,t-T) \, dw.
	\end{equation*}
	Then the continuity follows, since for fixed $T$ the second term on the right hand side will be less than $\eps/2$ whenever $|y-x|$ is small enough. This last claim holds since $p_T^D(\cdot,w)$ is continuous at $x$ for each $w\in D$ (e.g. by (\ref{eqn::kernel_decomp})), and by dominated convergence.  
	
	So to complete the proof, we just need to show that $u(y,t)\to 0$ as $y\to \st D$. However, this follows since  
	\begin{eqnarray*}
		\Prb{y}{|N_t|>0} & \leq & 1 - \Prb{y}{\text{the process becomes extinct before the first branching time}} \\
		&\leq & 1 - \mathbb{P}(Z>s)\mathbf{P}^y(\tau^D\leq s)
	\end{eqnarray*}
	for any $s>0$, where $Z$ is the same random variable described above. 
	The last line can be made arbitrarily small by first taking $s$ to 0, and then $y\to \partial D$. Thus $u(\cdot, t)\in C(\overline{D})$ for each $t>0$.  
\end{proof2}

We finish this section by recording a useful lemma in the critical case.

\begin{lemma}\label{lem::exp_nt2}
	For any $T>0$ there exists $C_T>0$ such that
	$\E{x}{|N_t|^2}\leq tC_T \vp(x) $ for all $t\geq T$ and $x\in D$. 
\end{lemma}

\begin{proof} This essentially follows from the many-to-two lemma, along with a couple of estimates that have been developed in this section. In the proof, $c_T, c_T'$ and $c_T''$ all represent constants depending only on $T$.  
	
	First, by (\ref{eqn::exp_kernel_bound}) we know that there exists $c_T$ such that for all $r\geq T/2$ and $y\in D$, 
	\begin{equation}\label{eqn::big_time_stayindomain}
		\mathbf{P}^y(\tau^D>r) \leq c_T \vp(y) \e^{-\lambda r}.\end{equation}
	We can also write, by Lemma \ref{lem::many_to_two}
	\begin{equation*}
		\E{x}{|N_t|^2} = \e^{\lambda t} \mathbf{P}^x(\tau^D>t) + \frac{\lambda}{m-1} (\mathbb{E}(A^2)-\mathbb{E}(A)) \int_0^t \e^{\lambda(2t-s)}\mathbf{P}^x(\I_{\{\tau^D>s\}} \mathbf{P}^{X(s)}(\I_{\{\tau^D>t-s\}})^2) \, ds
	\end{equation*}
	where by (\ref{eqn::big_time_stayindomain}) the first term on the right hand side of this expression is bounded by $c_T\vp(x)$. Furthermore, using (\ref{eqn::big_time_stayindomain}) again, the integrand in the second term is bounded by $c_T' \vp(x)$ for some $c_T'$, whenever $s\geq T/2$ and $t-s\ge T/2$. To conclude, notice that if $s<T/2$ then $t-s$ must be $\ge T/2$. Therefore, the integrand on this region is bounded by $c_T''\e^{\lambda s}\mathbf{P}^x(\vp(X_s)\I_{\{\tau^D>s\}})$ (for some $c_T''$), which is less than or equal to $c_T''\vp(x)$ by (\ref{eqn::exp_single_particle_martingale}). Finally, if $t-s<T/2$, then again we must have $s\ge T/2$, and the integrand on this region (also using that using that $\e^{\lambda (2t-s)}\leq \e^{\lambda T} \e^{\lambda s}$ in this case) is less than $\e^{\lambda T} \vp(x)$. Putting all of this together gives the lemma. 
\end{proof}


\section{Survival at Criticality: Proof of Theorem \ref{criticalsurvivalprob}}

Throughout this section, we will work in the critical case $\beta=\lambda/(m-1)$.


\subsection{Asymptotics for the Survival Probability}

We will first prove that the survival probability $\E{x}{|N_t|>0}$ decays asymptotically like $\vp(x)a(t)$, where 
\[a(t):=\int_D \Prb{z}{|N_t|>0} \vp(z) \, dz.\] In fact, we prove a more general Proposition (see below) from which we will obtain this as a special case. 

\begin{propn}
	\label{asymp}
	Suppose that the conditions of Theorem \ref{phasetrans} are satisfied. For a measurable function $f$ on $D$ with $0\leq f <1$, set \begin{equation*}
		a_f(t):= \int_D \vp(x) \E{x}{1-\hat{f}(N_t)} \, dx;
		\;\;\; \hat{f}(N_t):= \begin{cases}
			\prod_{u\in N_t} f(X_u(t)) & \text{ if } N_t\ne \emptyset \\
			1 & \text{ if } N_t=\emptyset.
		\end{cases}
	\end{equation*}
	Then 
	\begin{equation}
		\left| \frac{\E{x}{1-\hat{f}(N_t)}}{\vp(x)a_f(t)}-1 \right|\to 0
	\end{equation}
	as $t\to \infty$, uniformly in $x\in D$ and $f$ with $0\leq f < 1$. 
\end{propn}

\begin{remark0}
	When $f=0$ we see that $\E{x}{1-\hat{f}(N_t)}=\E{x}{|N_t|>0}$ and write $a_0(t):=a(t)$. 
\end{remark0}

\begin{remark0}\label{rmk::af_ne_0}
	Note that $a_f(t)\ne 0$ for all $t\geq 0$. Indeed for any $t\geq 0$, $1-\hat{f}(N_t)$ is a positive ($\mathcal{F}_t$-measurable) random variable, and for every $x\in D$ there is a set of strictly positive $\mathbb{P}^x$ measure on which $1-\hat{f}(N_t) >0$ (i.e. on the event that $|N_t|>0$). Thus, $\E{x}{1-\hat{f}(N_t)}>0$ for every $x\in D$. Since this is a measurable function of $x$, and $\vp(x)$ is also measurable and strictly positive in $D$, we obtain that $a_f(t)>0$. 
\end{remark0}

\begin{proof}
	The idea for the proof of this is to write, recalling the changes of measure from Section \ref{sec::com},
	\begin{equation}\label{eqn::key_asymptotic}\frac{\Prb{x}{1-\hat{f}(N_t)}}{\vp(x)}=\mathbb{Q}^x\big(\frac{1-\hat{f}(N_t)}{\sum_{u\in N_t} \vp(X_u(t))}\big)=\tQ{x}{\frac{1-\hat{f}(N_t)}{\sum_{u\in N_t} \vp(X_u(t))}}
	\end{equation} see (\ref{eqn::com_simple}), and then show that the right hand side essentially does not depend on $x$ for large $t$. The intuition behind this is that under $\tilde{\mathbb{Q}}^x$, the position of the \emph{spine} particle will converge very quickly to equilibrium. Moreover, contributions to $(1-\hat{f}(N_t))/\sum_{u\in N_t} \vp(X_u(t))$ from subprocesses branching off the spine before its position has become well mixed are unlikely to occur, as these have the law of branching diffusions under $\mathbb{P}^{\cdot}$, which we know are unlikely to survive for a long time. 
	
	To turn this into a rigorous proof, we first pick $1\leq t_0< t$ (we may assume $t>1$) and decompose \begin{equation*}N_t=N_{t}^1\cup N_{t}^2:=\{u\in N_t: \psi_{t_0}\prec u\}\cup \{u\in N_t: \psi_{t_0}\nprec  u\}.
	\end{equation*} Then we see using (\ref{eqn::key_asymptotic}) that 
	\begin{equation*}
		\frac{\E{x}{1-\hat{f}(N_t)}}{\vp(x)}\geq \tQ{x}{\frac{1-\hat{f}(N_t^1)}{\sum_{u\in N_t^1}\vp(X_u(t))} \I_{\{N_t^2=\emptyset\}}}=\tQ{x}{\tQ{x}{\frac{1-\hat{f}(N_t^1)}{\sum_{u\in N_t^1}\vp(X_u(t))} \I_{\{N_t^2=\emptyset\}}\, \big| \, \tilde{\mathcal{G}}_{t_0}}}
	\end{equation*}
	where we recall the definition of the $\sigma$-algebra $\tilde{\mathcal{G}}_{t_0}$ from Section \ref{sec::filtrations}, that knows everything about the spine's motion $(\xi_u)_{u\leq t_0}$ and about branching points along the spine up to time $t_0$. Using the description of the behaviour of the system under $\tilde{\mathbb{Q}}^x$, we see that this is equal to 
	\begin{equation}\label{eqn::lower_bound_af}
		\tQ{x}{\tQ{\xi_{t_0}}{\frac{1-\hat{f}(N_{t-t_0})}{\sum_{u\in N_{t-t_0}}\vp(X_u(t-t_0)) }} \prod_{v\prec \psi_{t_0}} (\E{X_v(s_v)}{|N_{t-s_v}|=0})^{A_v-1}}
	\end{equation}
	where recalling the notation from Section \ref{sec::trees}, this product is over all branching points along the spine before time $t_0$ (at times $s_v$, and positions $X_v(s_v)$ ($=\xi(s_v)$), with $A_v$ children being born). Now we make the simple bound
	\begin{equation*}
		1-\sum_{v\prec \psi_{t_0}}A_v \sup_{w\in D}\E{w}{|N_{t-t_0}|>0} \leq \prod_{v\prec \psi_{t_0}} (\E{X_v(s_v)}{|N_{t-s_v}|=0})^{A_v-1}
	\end{equation*}
	which tells us, by independence of $\xi_{t_0}$ and $\sum_{v\prec \psi_{t_0}}A_v$, that (\ref{eqn::lower_bound_af}) is greater than or equal to 
	\begin{equation*}
		\left[\tQ{x}{\tQ{\xi_{t_0}}{\frac{1-\hat{f}(N_{t-t_0})}{\sum_{u\in N_{t-t_0}}\vp(X_u(t-t_0)) }}}\right]\times \left[ 1-\sup_{w\in D} \E{w}{|N_{t-t_0}|>0} \tQ{x}{\sum_{v\prec \psi_{t_0}}A_v}\right].
	\end{equation*}
	We make a few observations. Firstly, $\tQ{x}{\sum_{v\prec\psi_{t_0}}A_v}\leq c t_0$ for some constant $c$ depending only on the branching diffusion. Secondly, 
	\begin{equation*}
		\tQ{x}{\tQ{\xi_{t_0}}{\frac{1-\hat{f}(N_{t-t_0})}{\sum_{u\in N_{t-t_0}}\vp(X_u(t-t_0)) }}}=\int_D K^D_{t_0}(x,y)\frac{\E{y}{1-\hat{f}(N_{t-t_0})}}{\vp(y)} \, dy
	\end{equation*}
	by (\ref{eqn::key_asymptotic}) again, and the fact that the spine particle has transition density $K^D$ (defined in Lemma \ref{lem::spineconvergence}). Using Lemma \ref{lem::spineconvergence}, which tells us that $K_{t_0}^D(x,y)$ converges to $\vp(y)^2$ exponentially fast, uniformly in $x$ and $y$, we see that the right hand side of the above is greater than or equal to $(1-c'\e^{-\gamma t_0})a_f(t-t_0)$, for another constant $0<c'<\infty$, where $\gamma>0$ is the spectral gap for $L$ on $D$. Overall, we obtain that 
	\begin{equation*}
		\frac{\E{x}{1-\hat{f}(N_t)}}{\vp(x)a_f(t-t_0)} \geq 1 - C (t_0\sup_{w\in D}\E{x}{|N_t|>0}+ \e^{-\gamma t_0})
	\end{equation*}
	where we emphasise that the constant $C$ depends only on the branching diffusion, and not on $x$ or $f$. 
	
	We now move on to the upper bound. This is simpler and we do not need to use the spine change of measure. Indeed we can write 
	\begin{align*}
		\E{x}{1-\hat{f}(N_t)} = \E{x}{\E{x}{1-\hat{f}(N_t)\, \big| \Fs{t_0}}} & = \E{x}{1-\prod_{u\in N_{t_0}} (1-\E{X_u(t_0)}{1-\hat{f}(N_{t-t_0})}} \\
		& \leq \E{x}{\sum_{u\in N_{t_0}} \E{X_u(t_0)}{1-\hat{f}(N_{t-t_0})} } ,\end{align*}
	where the last inequality follows because $0\leq f <1$. Then applying the Many-to-one Lemma \ref{lem::many_to_one} and dividing through by $\vp(x)$ we obtain 
	\begin{equation*}
		\frac{\E{x}{1-\hat{f}(N_t)}}{\vp(x)}\leq \int_D \frac{\e^{\lambda t_0}p_{t_0}^D(x,y)}{\vp(x)\vp(y)}\vp(y)\E{y}{1-\hat{f}(N_{t-t_0})} \, dy.
	\end{equation*}
	Finally, by Lemma \ref{lem::spineconvergence} again, we see this is bounded above by $a_f(t-t_0)\big(1+c'\e^{-\gamma t_0}\big)$, and so combining the upper and lower bounds we have 
	\begin{equation}\label{eqn::final_decay_asymptotic}
		\left| \frac{\E{x}{1-\hat f(N_t)}}{\vp(x)a_f(t-t_0)}-1 \right| \leq C \left(\e^{-\gamma t_0} + t_0 \sup_{w\in D} \E{w}{|N_{t-t_0}|>0} \right)
	\end{equation}
	where again $C$ is a constant not depending on $x$ or $f$.
	
	From here we can conclude. Because we know by Theorem \ref{phasetrans} that $\sup_{w\in D} \E{w}{|N_{s}|>0} \to 0$ as $s\to \infty$, we can pick a function $1\leq t_0(t)<t$ (defined for all $t>1$) such that the right hand side of (\ref{eqn::final_decay_asymptotic}) tends to $0$ as $t\to \infty$. Thus, to complete the proof we need only show that $a_f(t)\sim a_f(t-t_0(t))$ as $t\to \infty$. However, this follows since 
	\begin{equation}
		\left | \frac{a_f(t)}{a_f(t-t_0(t))}-1 \right|
		= \left| \int_D \left(\frac{\Prb{x}{1-\hat f(N_t)}}{\vp(x) a_f(t-t_0(t))}-1 \right) \vp(x)^2 \, dx \right| 
	\end{equation}
	which converges to $0$, uniformly in $f$, by (\ref{eqn::final_decay_asymptotic}) and the choice of $t_0(t)$.
\end{proof}


\subsection{Asymptotics for $a(t)$}

To complete the proof of Theorem \ref{criticalsurvivalprob}, we need to show that $a_0(t)=:a(t)\sim bt^{-1}$ as $t\to \infty$, where 
\begin{equation}
	\label{eqn::defnb}
	b:=\frac{2(m-1)}{\lambda (1,\vp^3) (\mathbb{E}(A^2)-\mathbb{E}(A))}.
\end{equation} The idea is the following: if we write $\E{x}{|N_t|>0}=u(x,t)$ as before, then $u$ will satisfy a certain partial differential equation in $D$ (this is the FKPP equation if $L=\frac{1}{2}\Delta$). This will give us an ordinary differential equation that is satisfied by $a(t)$, and consequently we will be able to deduce the desired asymptotic. In fact, we will never explicitly use the fact that $u$ is a solution of this PDE (and instead derive the equation for $a$ directly) but this is the motivation behind our approach. Again we state the relevant lemma in a more general setting, as this will be helpful later on. We let $u_f(x,t)=\E{x}{1-\hat{f}(N_t)}$ for any measurable $0\leq f<1$ on $D$, and note that $u_0$ corresponds to $u$ defined above.

\begin{lemma}\label{lem::aeqn} Suppose that $0\leq f< 1$ and $f\in C(\bar{D})$. Then $a_f(t):=\int_D \vp(x)u_f(x,t)$ is differentiable for all $t>0$ and 
	\begin{equation} \label{eqn::aeqn} \frac{da_f}{dt}(t)= -\frac{\lambda}{m-1} \int_D (G(1-u_f(x,t))+mu_f(x,t)-1)\vp(x) \, dx\end{equation}
	where $G(s)=\mathbb{E}(s^A)$ is the generating function for $A$. 
\end{lemma}

\begin{proof}
	The idea is to rewrite $u_f(x,t)=\E{x}{1-\hat{f}(N_t)}$, by conditioning on the first branching time for the system. This gives us that $a_f(t)$ is equal to 
	\begin{align*} & \int_D \vp(x)\e^{-\frac{\lambda t}{m-1}}\mathbf{P}^x(\nh (1-f(X(t))))\, dx \\ + & \int_D \vp(x) \int_0^t \frac{\lambda }{m-1} \e^{-\frac{\lambda s}{m-1}} \mathbf{P}^x\big(\I_{\{\tau^D>s\}} (1-G(1-u_f(X(s),t-s)))  \big) \, ds dx.
	\end{align*}
	The key fact that allows us to simplify this expression is that for all $y\in D$ and $t>0$ we have $\int_D p_t^D(x,y)\vp(x)\, dx= \e^{-\lambda t} \vp(y)$ by (\ref{eqn::f_decomp}) (and \eqref{eqn::kernel_decomp}, which shows that $p_t^D$ is symmetric in $x$ and $y$). Hence, by Fubini, we have 
	\begin{align}\label{eqn::a_nice_eq}
		a_f(t)= & \e^{-\lambda t (\frac{1}{m-1}+1)} \int_D\vp(y)(1-f(y)) \, dy \\ \nonumber &+ e^{-\lambda t (\frac{1}{m-1}+1)} \int_0^t \int_D \frac{\lambda}{m-1}\e^{\lambda s(\frac{1}{m-1}+1) } \vp(y) (1-G(1-u_f(y,s))) \, dy ds
	\end{align}
	where we have also made the change of variables $s\leftrightarrow t-s$ in the integral. Now we claim that for every $y\in D$, $s\mapsto u_f(y,s)$ is a continuous function on $[0,\infty)$. Assuming this is true (we will prove it momentarily) we see by dominated convergence (and the definition of $G$) that the time integrand in (\ref{eqn::a_nice_eq}) is continuous in $s$. Thus, by the fundamental theorem of calculus, $a_f(t)$ is differentiable and 
	\begin{equation*}
		\frac{da_f(t)}{dt} = -\lambda(\frac{1}{m-1}+1) a_f(t) + \frac{\lambda}{m-1} \int_D \vp(y)(1-G(1-u_f(y,t))) \, dy.
	\end{equation*}
	Writing $a_f(t)=\int_D \vp(x) u_f(x,t) \, dx$ to make the right hand side of the above into a single integral over $D$, the result follows.
	
	So, we only need to prove the claim concerning continuity. Fix $t\in [0,\infty)$ and $y\in D$. Then for any $s\in [0,\infty)$
	\begin{equation*}
		|u_f(y,s)-u_f(y,t)|=|\E{y}{\hat{f}(N_t)-\hat{f}(N_s)}|=|\E{y}{\I_{\{|N_{t\wedge s}|>0\}} \big( \prod_{u\in N_{t\wedge s}}  f(X_u(t\wedge s))-\prod_{u\in N_{t\wedge s}}\E{X_u(t\wedge s)}{\hat{f}(N_{|t-s|})}\big)}|
	\end{equation*} 
	where for the second equality we have conditioned on $\mathcal{F}_{t\wedge s}$. By the triangle inequality (and using the fact that $f<1$), this is less than 
	\begin{equation}\label{eqn::cont_uf}
		\E{y}{\sum_{u\in N_{t\wedge s}} |f(X_u(t\wedge s))-\E{X_u(t\wedge s)}{\hat{f}(N_{|t-s|})}|},
	\end{equation}
	which we can also write as 
	\begin{equation*}
		\e^{\lambda(t\wedge s)} \mathbf{E}^y\big( \I_{\{\tau^D>t\wedge s\}}\big|f(X(t\wedge s))-\E{X_u(t\wedge s)}{\hat{f}(N_{|t-s|})}\big| \big)=\e^{\lambda(t\wedge s)} \int_D p_{t\wedge s}^D(y,z) |f(z)-\E{z}{\hat{f}(N_{|t-s|})} | \, dz
	\end{equation*}
	by the many-to-one lemma. Therefore by dominated convergence, since by Lemma \ref{lem::spineconvergence} we can bound $p_r^D(y,z)$ above by something integrable in $z$, uniformly in  $r>t/2$ for example, it is enough to show that $$|f(z)-\E{z}{\hat{f}(N_r)}|\to 0$$ pointwise in $z\in D$ as $r\to 0$. 
	
	To do this, we bound  \begin{equation}\label{eqn::continuity_uf} \E{z}{f(z)-\hat{f}(N_r)}\leq (1-\e^{-\lambda/(m-1)r}) + f(z)\mathbf{P}^z(\tau^D\leq r) + \int_D p_r^D(z,w)|f(z)-f(w)| \, dw, \end{equation} which follows since the probability of the first branching time being less than $r$ is equal to $(1-\e^{(m-1)^{-1}\lambda r})$, and on this event $|f(z)-\hat{f}(N_r)|\leq 1$. Clearly the first two terms of (\ref{eqn::continuity_uf}) go to $0$ as $r\to 0$. Moreover so does the third. Indeed, by continuity of $f$, for any $\eps>0$ we can choose $\eta=\eta(\eps)$ such that $|w-z|<\eta \Rightarrow |f(w)-f(z)|<\eps$, and hence the third term can be made less than $\eps + \mathbf{P}^z(|X(r)-z|>\eta)$. By continuity of the process $X$, we can then choose $r$ small enough that this is less than $2\eps$. 
\end{proof}

Now let us see how this allows us to conclude the proof of Theorem \ref{criticalsurvivalprob}. \\

\begin{proof2}{Theorem \ref{criticalsurvivalprob}} 
	First we observe that since $A$ has finite variance, $s^{-2}(G(1-s)+ms-1)\to \frac{1}{2}(\mathbb{E}[A^2]-\mathbb{E}[A])$ as $s\to 0$ (by dominated convergence). This means, by Lemma \ref{lem::aeqn}, that we can write, for $a(t)=a_0(t)$ and denoting $\dot{a}(t)=\frac{da}{dt}(t)$:
	\begin{equation*}
		\frac{\dot{a}(t)}{a(t)^2}= -b^{-1} (1,\vp^3)^{-1} \int_D \vp(x) \frac{u^2(x,t)}{a(t)^2}(1+E(x,t)) \, dx
	\end{equation*}
	where $E(x,t)\to 0$ uniformly in $x$ as $t\to \infty$ (here we have also used that the survival probability decays uniformly to $0$ in $D$). Now by Proposition \ref{asymp}, we can conclude that 
	\begin{equation*}
		\frac{d}{dt}(\frac{1}{a(t)})=-\frac{\dot{a}(t)}{a(t)^2}=b^{-1} +  \hat{E}(t)
	\end{equation*}
	where $\hat{E}(t)\to 0$ as $t\to \infty$. Using the fundamental theorem of calculus, \cite[Theorem 6.21]{rudin} (note we only need integrability of the derivative), we deduce the result.
\end{proof2} 
\\


\section{The Conditioned System}

Theorem \ref{criticalsurvivalprob} allows us to study the law of our branching diffusions conditioned to survive for a long time in much greater depth. One aspect of the limiting behaviour is captured by what happens to the law of the process run up to some fixed time $T$, if it is then conditioned to survive until a much larger time $t$. It turns out that this limiting description is given precisely by the evolution of the process under $\mathbb{Q}^x$, as described in Section \ref{sec::com}.  
\medskip 

\begin{proof2}{Proposition \ref{thm::conditionedsystem}}
	Recall, we would like to prove that for any $T\geq 0$, $x\in D$ and $B\in \mathcal{F}_T$, we have that
	\[\lim_{t\to \infty} \Prb{x}{B||N_t|>0}=\mathbb{Q}^x(B).\]
	Conditioning on $\mathcal{F}_T$, we see that 
	\[
	\Prb{x}{B||N_t|>0} = \frac{\E{x}{\I_B\Prb{x}{|N_t|>0|\mathcal{F}_T} }}{\Prb{x}{|N_t|>0}} :=\frac{\E{x}{\I_B Y}}{\Prb{x}{|N_t|>0}}
	\]
	where we have defined, ordering $N_T=\{u_1,\cdots, u_{|N_T|}\}$ (for example with the depth first order described in Section \ref{sec::trees}),
	\[
	Y := \Prb{x}{|N_t|>0\,|\,\mathcal{F}_T}
	= \sum_{i=1}^{|N_T|} \Prb{X_{u_i}(T)}{|N_{t-T}|>0} \left(\prod_{j>i} \Prb{X_{u_j}(t)}{|N_{t-T}|=0}\right).
	\]
	Then, from the asymptotic for the survival probability, Theorem \ref{criticalsurvivalprob}, and the fact that $\frac{t}{t-T}\to 1$ as $t\to \infty$, it follows that 
	\[\frac{\I_B Y}{\Prb{x}{|N_t|>0}} \underset{t \to \infty}{\to} \frac{\sum_{u\in N_T} \vp(X_u(T))}{\vp(x)}\I_B=\frac{M_T}{M_0}\I_B\]
	almost surely, as $t\to \infty$. Moreover, we have that
	$Y \leq  \sum_{u\in N_T} \Prb{X_u(T)}{|N_{t-T}|>0} \leq C \frac{M_T}{t-T} $
	for all large enough $t$ and some $C>0$ depending only on the diffusion. This means
	we can dominate $\I_B Y/\Prb{x}{|N_t|>0}$ by an integrable random variable, namely a constant multiple of $M_T$. The dominated convergence theorem then
	provides the result.
\end{proof2}

Given the asymptotic for the survival probability, it is also not too much work to prove Theorem \ref{thm::convergenceconditionednt}, which gives some limiting information on the positions of particles at time $t$, given survival. Recall we would like to show that for any $f$ with $\ip{f^2}{\vp}<\infty$ (this condition is required because we will apply the many-to-two Lemma in the proof) we have 
\[ ( t^{-1}\sum_{u\in N_t}f(X_u(t)) \giv |N_t|>0 ) \to Z \] 
in distribution as $t \to \infty$, where $Z \sim \text{Exp}(b (\vp,f)^{-1})$ and $b=\frac{2(m-1)}{\lambda(1,\vp^3)(\mathbf{E}(A^2)-\mathbf{E}(A))}$ as before.
\\

We first need an auxilliary Lemma.

\begin{lemma}\label{lemma::decay_aft}
	Recall the definition of $a_f(t)$ from Lemma \ref{lem::aeqn}. Then 
	$$\frac{1}{t}(\frac{1}{a_f(t)}-\frac{1}{a_f(0)}) \to b^{-1},$$ uniformly over all $f\in C(\bar{D},[0,1)):= \{f\in C(\bar{D})\, : \, 0\leq f <1\}$.
\end{lemma}
\begin{proof}
	Note that $a_f(t)>0$ for all $t$, by Remark \ref{rmk::af_ne_0}. Exactly as in the proof of Theorem \ref{criticalsurvivalprob}, we obtain that 
	\begin{equation*}
		\frac{d}{dt}(\frac{1}{a_f(t)})=b^{-1}+\hat{E}_f(t)
	\end{equation*}
	where $\hat{E}_f(t)\to 0$ as $t\to \infty$, uniformly in $f\in C(\bar{D},[0,1))$. The uniformity comes from the fact that $u^f(x,t) \leq \E{x}{|N_t|>0}$ tends to $0$ uniformly over $D \times C(\bar{D},[0,1))$, and that the convergence in Proposition \ref{asymp} is uniform over $C(\bar{D},[0,1))$. The result then follows by the fundamental theorem of calculus. 
\end{proof}

\begin{proof2}{Theorem \ref{thm::convergenceconditionednt}}
	We first complete the proof in the special case $f=\vp$. We will show that for any $\alpha>0$ 
	\begin{equation}
		\label{eqn::convergence_conditioned_laplace}
		\E{x}{ \e^{-\frac{\alpha}{t}\sum_{u\in N_t} \vp(X_u(t))}\,\big|\, |N_t|>0}\to \frac{b}{b+\alpha}
	\end{equation}
	uniformly in $x$ as $t\to \infty$ (which is enough, by L\'{e}vy's continuity theorem for the Laplace transform). 
	
	Fix $\alpha>0$. To prove (\ref{eqn::convergence_conditioned_laplace}), we first observe that if we define $f_t(x)=\e^{-\frac{\alpha}{t} \vp(x)}$, then 
	\begin{equation*}
		\E{x}{ \e^{-\frac{\alpha}{t}\sum_{u\in N_t} \vp(X_u(t))}\,\big|\, |N_t|>0}= 1-\frac{u_{f_t}(x,t)}{\E{x}{|N_t|>0}}.
	\end{equation*}
	By Proposition \ref{asymp} and Theorem \ref{criticalsurvivalprob}, we also know that 
	\begin{equation*}\frac{u_{f_t}(x,t)}{\vp(x)a_{f_t}(t)} \to 1 \text{ and } \frac{b\vp(x)}{t\E{x}{|N_t|>0}}\to 1
	\end{equation*} as $t\to \infty$, uniformly in $x\in D$. Note that it doesn't matter that $f_t$ depends on $t$, since the convergence in Proposition \ref{asymp} is uniform over $C(\bar{D},[0,1))$. Hence, we are required to show that 
	$$ \frac{1}{ta_{f_t}(t)} \to \frac{1}{b}+\frac{1}{\alpha} $$
	as $t\to \infty$. However, this follows directly from Lemma \ref{lemma::decay_aft} and the fact that $t(1-\e^{-\frac{\alpha}{t}\vp(z)})\to \alpha\vp(z)$ as $t\to \infty$ for any $z\in D$.

	To deal with general $f$ such that $(f^2,\vp)<\infty$, we write $\tilde{f}=f-\ip{\vp}{f}\vp$. We will show that 
	for any $\eps>0$ 
	\begin{equation}\label{eqn::tilde_f_irrelevant} \mathbb{P}^x (|t^{-1}\sum\nolimits_{u\in N_t} \tilde{f}(X_u(t))|>\eps\giv |N_t|>0) \to 0 \end{equation}
	as $t\to \infty$, uniformly in $x$. This clearly implies the result by writing $f=\tilde{f}+(\vp,f)\vp$ and applying the special case of Theorem \ref{thm::convergenceconditionednt} (with $\vp$) that we have just proved. 
	
	To prove (\ref{eqn::tilde_f_irrelevant}), it is enough by conditional Markov's inequality to show that $\mathbb{P}^x\big((t^{-1}\sum_{u\in N_t}\tilde{f}(X_u(t)))^2\, | \, |N_t|>0\big)\to 0$ as $t\to \infty$, uniformly in $x$. Moreover, the many-to-two lemma, Lemma  \ref{lem::many_to_two}, tells us that 
	\begin{align}\label{eqn::huge_horrible} &\mathbb{P}^x\big((t^{-1}\sum_{u\in N_t}\tilde{f}(X_u(t)))^2\, | \, |N_t|>0\big)  = \frac{\vp(x)}{t \mathbb{P}^x(|N_t|>0)} \\ \nonumber 
		& \times \left(\frac{\e^{\lambda t} \mathbf{P}^x\big(\tilde{f}(X(t))^2\I_{\{\td>t\}}\big)+\frac{\mathbb{E}(A^2)-\mathbb{E}(A)}{m-1}\int_0^t \lambda \e^{\lambda (2t-s)}\mathbf{P}^x\big(\I_{\{\td>s\}}\left[\mathbf{P}^{X(s)}\big(\I_{\{\td>t-s\}} \tilde{f}(X(t-s))\big)\right]^2\big) \, ds}{t \vp(x)} \right)\end{align}
	where we know by Theorem \ref{criticalsurvivalprob} that the expression outside of the brackets on the right hand side is uniformly bounded in $t$ ($\ge 1$, say) and $x$. We also know, by (\ref{eqn::exp_kernel_bound}), that for all $r>1$ and some constant $C$ depending only on the branching diffusion, we have
	\begin{equation} \label{eqn::uniformparticlebound1}\frac{\e^{\lambda r} \mathbf{P}^x\big(\tilde{f}(X(r))^2\I_{\{\td>r\}}\big)}{\vp(x)} \leq C(\vp,\tilde{f}^2), \end{equation} and by Lemma \ref{lem::spineconvergence}, plus Cauchy--Schwarz and the fact that  that $(\vp,\tilde{f})=0$, that
	\begin{align} \label{eqn::uniformparticlebound2}\frac{\e^{\lambda r} \mathbf{P}^x\big(\tilde{f}(X(r))\I_{\{\td>r\}}\big)}{\vp(x)} & = \int_D \big( \frac{K_r^D(x,y)}{\vp(y)^2}-1 \big) \tilde{f}(y)\vp(y) \, dy + \int_D \tilde{f}(y)\vp(y) \, dy \nonumber \\ & \leq C\e^{-\gamma r}(\vp,\tilde{f}^2)^{1/2}. \end{align} 
Finally, we observe that for $s>1$, by Cauchy--Schwarz and (\ref{eqn::exp_kernel_bound}) again,
	\begin{align*}\label{eqn::bound_small_interval}
		\vp(x)^{-1}\mathbf{P}^x\big(\I_{\{\td>s\}}\left[\mathbf{P}^{X(s)}\big(\I_{\{\td>t-s\}} \tilde{f}(X(t-s))\big)\right]^2\big) & \leq \vp(x)^{-1} \mathbf{P}^x\big(\I_{\{\td>s\}}\mathbf{P}^{X(s)}\big(\I_{\{\td>t-s\}} \left[\tilde{f}(X(t-s))\right]^2\big)\big) \\ \nonumber  
		& = \vp(x)^{-1}\mathbf{P}^x\big(\nh \big[\tilde{f}(X(t))\big]^2) \\ \nonumber 
		& \leq C \e^{-\lambda t}(\vp, \tilde{f}^2)
	\end{align*}
	where $C$ is another, possibly different, constant. This tells us that 
	\begin{equation*}
		\frac{\int_{t-1}^t \lambda \e^{\lambda (2t-s)} \mathbf{P}^x\big(\I_{\{\td>s\}}\left[\mathbf{P}^{X(s)}\big(\I_{\{\td>t-s\}} \tilde{f}(X(t-s))\big)\right]^2\big) \, ds}{\vp(x)t}\to 0
	\end{equation*}
	as $t\to \infty$, uniformly in $x$. Using \eqref{eqn::uniformparticlebound1} and \eqref{eqn::uniformparticlebound2} on the remaining parts of the second line of (\ref{eqn::huge_horrible}), together with the fact that \begin{equation} \label{eqn::useful_ends_of_integrals} \mathbf{P}^x(\I_{\{\tau^D>s\}} \vp(X(s))^2)\leq \sup_{w\in D}|\vp(w)| \mathbf{P}^x( \I_{\{\tau^D>s\}}\vp(X(s))) \leq \vp(x)\sup_{w\in D}|\vp(w)| \e^{-\lambda s}\end{equation} for all $s>0$, it is easy to conclude that $\mathbb{P}^x\big((t^{-1}\sum_{u\in N_t}\tilde{f}(X_u(t)))^2\, | \, |N_t|>0\big)$ does indeed tend to $0$, uniformly in $x$.
\end{proof2}


We conclude by explaining how one can obtain Corollary \ref{uniformparticle} from here, which describes the asymptotic distribution of a particle picked at random from the population, given survival.
\medskip

\begin{proof2}{Corollary \ref{uniformparticle} }
	To prove the Corollary we first show that for any $f$ with $(\vp,f^2)<\infty$
	\begin{equation}\label{eqn::quotient} \mathbb{P}_x\left( \left|\frac{\sum_{u\in N_t}f(X_u(t))}{\sum_{u\in N_t}\vp(X_u(t))} - (\vp,f)\right|>\eps \, \big| \, |N_t|>0\right)\to 0 \end{equation} as $t\to \infty$. Defining $\tilde{f}$ as in the proof of Theorem \ref{thm::convergenceconditionednt}, the left hand side of (\ref{eqn::quotient}) is equal to 
	\begin{align}\label{eqn::quotient2} \mathbb{P}^x\left(\ \left|\frac{t^{-1}\sum_{u\in N_t}\tilde{f}(X_u(t))}{t^{-1}\sum_{u\in N_t}\vp(X_u(t))}\right|>\eps\, \big|\, |N_t|>0\right)\leq \,& \mathbb{P}^x\left(t^{-1}\sum_{u\in N_t}\tilde{f}(X_u(t))>\delta\, \big| \,|N_t|>0 \right) \\ \nonumber
		+\, &\mathbb{P}^x\left(t^{-1}\sum_{u\in N_t}\vp(X_u(t))<\frac{\delta}{\eps} \, \big| \, |N_t|>0 \right)    \end{align}
	for any $\delta>0$. From (\ref{eqn::tilde_f_irrelevant}) and Theorem \ref{thm::convergenceconditionednt}, if we take a limit as $t\to \infty$ on the right hand side, we are left with simply the probability that an exponential random variable is less than $\delta/\eps$. Taking $\delta\to 0$ proves (\ref{eqn::quotient}). The corollary then follows by applying the above with both $f$ and the constant function $1$, and writing $\sum f(X_u(t))/|N_t| = \sum f(X_u(t))/\sum \vp(X_u(t)) \times \sum \vp(X_u(t))/|N_t|.$ 
\end{proof2}
\medskip

Note that by (\ref{eqn::tilde_f_irrelevant}), for any fixed $\delta$ the first term on the right hand side of (\ref{eqn::quotient2}) converges to $0$ uniformly in $x$. Moreover, by Markov's inequality we can write
\begin{equation*}
	\mathbb{P}^x\left(t^{-1}\sum_{u\in N_t}\vp(X_u(t))<\frac{\delta}{\eps} \, \big| \, |N_t|>0 \right)  \leq \e^{1/\eps} \E{x}{ \e^{-\frac{1}{\delta t} \sum_{u\in N_t} \vp(X_u(t))} \, \big|\, |N_t|>0}
\end{equation*}
which by (\ref{eqn::convergence_conditioned_laplace}) converges to $\e^{1/\eps} b(b+\delta^{-1})^{-1}$ as $t\to \infty$, uniformly in $x$. This implies the following:

\begin{cor}\label{cor::unifaverage} For $f$ as in Corollary \ref{uniformparticle}, and any $\eps>0$, the convergence to 0 in (\ref{eqn::quotient}) is uniform in the starting position $x$.
\end{cor}

We will use this to prove a stronger version of Corollary \ref{uniformparticle}. Corollary \ref{uniformparticle} tells us that the average value of $f(X_v(t))$ among all vertices $v\in N_t$, given survival, converges to $\ip{f}{\vp}/\ip{1}{\vp}$. The next lemma will tell us that in fact we need only look at the average over a large enough subset of $N_t$ (we will see precisely what this means in a moment). First we need some notation. Recall from Section \ref{sec::bds} that we can view our branching diffusion as a marked tree $(T,l,X)\in \mathcal{T}$. This means that for every $t$ such that $N_t\ne \emptyset$, we can write $N_t=\{w_1(t),w_2(t),\cdots, w_{|N_t|}(t)\},$ where the indices correspond to the depth first ordering of the particles (see Section \ref{sec::trees}). For any $1\le M\le |N_t|$ we can then define the set
\begin{equation*}
	N_{t,M} = \{w_i(t) : 1\leq i \leq M\},
\end{equation*}
and have $|N_{t,M}|=M$.

\begin{lemma}
	\label{lemma::average_lln} Let $f$ be a bounded measurable function on $D$. Then for any $\eps,\rho>0$ and $x\in D$
	\[\mathbb{P}^x \big( B_t^{f,\eps}(\rho) \, \big| \, |N_t|>0 \big) :=\mathbb{P}^x\left(\left.\big\{\rho t \leq |N_t|\big\} \, \cap\,\big\{ \bigcup\nolimits_{\rho t\leq M\leq |N_t|} \big\{\big|\frac{\sum_{u\in N_{t,M}} f(X_u(t))}{M}-\frac{\ip{f}{\vp}}{\ip{1}{\vp}}\big|>\eps  \big\} \big\}\, \right| \, |N_t|>0\right)\]
	converges to $0$
	as $t\to \infty$. 
\end{lemma}

\begin{proof}
	We will prove the lemma by looking at the system conditioned to survive until time $t$, and dividing $N_t$ into families, depending on whether or not particles have the same ancestor at some earlier time. This earlier time will be chosen such that with high probability, the average value of $f(X_v(t))$ over $v$ in any one of these families is close to $\ip{f}{\vp}/\ip{1}{\vp}$. This already shows that there are many subsets of $N_t$ over which the average value of $f(X_v(t))$ is close to what we want. To extend this to \emph{all} large enough subsets of $N_t$ (as in the statement of the lemma), we will show that the size of each of these families is very small compared to $t$. 
	
	\begin{figure}[h]
		\centering
		\includegraphics[scale=0.5]{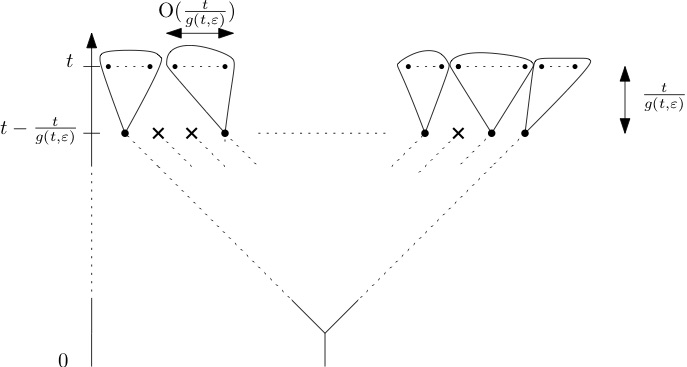}
		\caption{Sketch of the argument. There are $O(g(t,\eps))$ particles at time $t-t/g(t,\eps)$ with descendants at time $t$ (marked with dots). Each of these families is likely to be \emph{good} and has size $O(t/g(t,\eps))$.}
		
	\end{figure}	
	To do this, fix $\eps>0$ and write
	\[p(t,\eps):= \sup_{x\in D} \mathbb{P}^x \left(\left. \left|\frac{\sum_{u\in N_t} f(X_u(t))}{ |N_t|}-\frac{\ip{f}{\vp}}{\ip{1}{\vp}}\right|>\eps/2  \right| |N_t|>0 \right), \]
	which by Corollary \ref{cor::unifaverage}, converges to $0$ as $t\to \infty$. This means that we can choose a function $g(t,\eps)\ll t$ such that 
	$g(t,\eps)\to \infty$ as $t\to \infty$, but 
	\begin{equation} \label{eqn::scales} g(t,\eps) \, p\left(\frac{t}{g(t,\eps)},\eps\right)\to 0\end{equation}
	as $t\to \infty$. \footnote{Indeed, since $\hat{p}(t):=\sup\{p(u,\eps): u\geq t\}$ converges monotonically to $0$ as $t\to \infty$, we can choose $g(t,\eps) \le \sqrt{t}$ but still converging to $\infty$, such that $g(t,\eps)\hat{p}(\sqrt{t})\to 0$ as $t\to \infty$. Then since $p(t/g(t,\eps),\eps)\leq \hat{p}( t/g(t,\eps))\leq \hat{p}(\sqrt{t})$, the function $g$ will satisfy (\ref{eqn::scales}).}
	
	As mentioned above, we will break up the set $N_t$ into families. Two vertices will be in the same family if they have a common ancestor at time $t-t/g(t,\eps)$. For $1\leq i \leq |N_{t-t/g(t,\eps)}|$ we define \begin{equation*}
		m_i := \frac{\sum_{\{v\in N_t:w_i(t-t/g(t,\eps))\prec v\}} f(X_v(t)) }{\sigma_i} \;\;\; \text{ for } \;\;\; \sigma_i := |\{v\in N_t:w_i(t-t/g(t,\eps))\prec v\}|
	\end{equation*} to be the average value of $f$ among the $i$th of these families. If $w_i(t-t/g(t,\eps))$ has no descendants at time $t$ we set $m_i = (f,\vp)/(1,\vp)$. The key to the proof of this lemma will be to show that 
	\begin{equation} \label{eqn::averagesubtrees} \E{x}{A_t \, \big|\, |N_t|>0}:=\mathbb{P}^x \left(\left. \bigcup_{i=1}^{|N_{t-t/g(t,\eps)}|} \left\{\left|m_i-\frac{\ip{f}{\vp}}{\ip{1}{\vp}} \right|>\eps/2 \right\}\right||N_t|>0 \right) \to 0\end{equation}
	as $t\to \infty$.
	
To do this, it turns out to be more convenient to look at the unconditioned probability, and then use Theorem \ref{criticalsurvivalprob}. We write 
	\begin{equation*}
		\E{x}{A_t\cap \{|N_t|>0\}}  =\E{x}{\E{x}{A_t\cap\{|N_t|>0\} \, \big| \, \Fs{t-t/g(t,\eps)}}} 
	\end{equation*} which by Theorem \ref{criticalsurvivalprob}, a union bound, and the definition of $p(s,\eps)$, yields the inequality
	\begin{equation}\label{eqn::scales_arg_part1}
	\E{x}{A_t\cap \{|N_t|>0\}} \leq C \E{x}{|N_{t-t/g(t,\eps)}|}p(\frac{t}{g(t,\eps)},\eps) \frac{g(t,\eps)}{t}
	\end{equation}
	for some $C$ depending only on the branching diffusion. Moreover, by Lemma \ref{lem::many_to_one} and (\ref{eqn::exp_kernel_bound}), we have that $\E{x}{|N_{t-t/g(t,\eps)}|}\leq C'\vp(x)$ for some further $C'$, and so the expression on the right hand side of (\ref{eqn::scales_arg_part1}) is less than or equal to a constant times $t^{-1}\vp(x)p(t/g(t,\eps),\eps)g(t,\eps)$. Finally, dividing left hand side of (\ref{eqn::scales_arg_part1}) by $\E{x}{|N_t|>0}\sim t^{-1}\vp(x)$ and then applying (\ref{eqn::scales}), we obtain (\ref{eqn::averagesubtrees}).

	Now we prove that 
	\begin{equation} \label{eqn::shortgaps}
		\mathbb{P}_x\left(A_{t}' \giv |N_t|>0\right):=\mathbb{P}_x \big(\bigcup_{i=1}^{|N_{t-t/g(t,\eps)}|} \big\{\sigma_i > \frac{t}{(g(t,\eps))^{1/3}}\big\}\, \big| \, |N_t|>0\big) \to 0 \end{equation} as $t\to \infty$. For this we apply a similar argument to above. We write 
	
	\begin{align}\label{eqn::scales_arg_part2}
		\E{x}{A_t'\cap \{|N_t|>0\}}  & =\E{x}{\E{x}{A_t'\cap\{|N_t|>0\} \, \big| \, \Fs{t-t/g(t,\eps)}}} \\ \nonumber & \leq C \E{x}{|N_{t-t/g(t,\eps)}|} \sup_{y\in D} \E{y}{|N_{t/g(t,\eps)}|> \frac{t}{g(t,\eps)^{1/3}}}  \leq \frac{C\vp(x) }{g(t,\eps)^{1/3}t}
	\end{align}
	where the last inequality follows because for any $y\in D$, by Markov's inequality and the uniform bound on $\E{y}{|N_{t/g(t,\eps)}|^2}$ from Lemma \ref{lem::exp_nt2}, 
	\begin{equation*}
		\E{y}{|N_{t/g(t,\eps)}|> \frac{t}{g(t,\eps)^{1/3}}} \leq \frac{\E{y}{|N_{t/g(t,\eps)}|^2}g(t,\eps)^{2/3}}{t^2}\leq C\frac{\frac{t}{g(t,\eps)} g(t,\eps)^{2/3}}{t^2}.
	\end{equation*}
	 Now dividing through by $\E{x}{|N_t|>0}$ on the right hand side of (\ref{eqn::scales_arg_part2}) gives (\ref{eqn::shortgaps}).
	
	Finally, we are in a position to prove that $\mathbb{P}_x(B_t^{f,\eps}(\rho)\giv |N_t|>0)\to 0$ as $t\to \infty$. By the above work, and a union bound it is enough to show that \[B_t^{f,\eps}(\rho)\subset A_t\cup A_t'\] for all $t$ large enough. In fact, we will show that 
	\[ (A_t)^c \cap (A_t')^c \cap \{\rho t \le |N_t|\} \subset (B_t^{f,\eps}(\rho))^c,\] 
	for all $t$ large enough, which suffices, since by definition $(A_t)^c \cap (A_t')^c \cap \{\rho t >|N_t|\}\subset (B_t^{f,\eps}(\rho))^c$. 
	
	So, suppose we are on the event $(A_t)^c \cap (A_t')^c \cap \{\rho t \le |N_t|\}$ and for every $ M \leq |N_t|$, set $M\leq k(M):= \sigma_0+\sigma_1+\cdots + \sigma_{i+1}$, where $\sigma_0:= 0$, and $i$ is the unique integer such that $\sigma_0 + \cdots + \sigma_{i} < M \le \sigma_0+ \cdots + \sigma_{i+1}$. Then because we are on the event $(A_t)^c$ we have that 
	\[\left|\frac{\sum_{u\in N_{t,k(M)}} f(X_u(t))}{k(M)}-\frac{\ip{f}{\vp}}{\ip{1}{\vp}}\right| \leq \eps/2 \] for all $ \rho t\leq M \leq |N_t|$ simultaneously (in fact, for all $1\leq M \leq |N_t|$).
	Furthermore, since we are on the event $(A_t')^c$, for $M\geq \rho t$ we have 
	\begin{align*}\label{eqn::bounddifferences}\left|\frac{\sum_{u\in N_{t,M}} f(X_u(t))}{M} -\frac{\sum_{u\in N_{t,k(M)}} f(X_u(t))}{k(M)}\right|
		& \leq  \left|\frac{k(M)-M}{M} \right| \, \left(\big|\frac{\sum_{u\in N_{t,k(M)}} f(X_u(t))}{k(M)}\big|+\sup_{w\in D}|f(w)|\right) \\ \nonumber
		& \leq \frac{1}{\rho \, g(t,\eps)^{1/3}}(2\frac{(f,\vp)}{(1,\vp)}+\sup_{w\in D}|f(w)|).\end{align*}
	
	Thus, for all $t$ large enough, because $g(t,\eps)\to \infty$, this must be less than $\eps/2$ for all $M\geq \rho t$. Consequently, we must be on the event $(B_t^{f,\eps}(\rho))^c$. 
\end{proof}


\section{Convergence to the Brownian CRT}

From this point onwards, we will assume that $\vp\in C^1(\overline{D})$ as in the statement of Theorem \ref{thm::crtconv}. In particular this means that all the first order partial derivatives of $\vp$ are bounded on $D$.

Recall from Section \ref{sec::trees} that we associate to our branching diffusion $(T,l,X)$ a continuous plane tree $\mathbf{T}=\mathbf{T}(T,l,X)$. 
This is the continuum tree with branch lengths given by lifetimes of particles in the system. Let us first describe the continuous time \emph{depth first exploration} of $\mathbf{T}$; this is very closely related to the continuous exploration of $\mathbf{T}$ used to define the contour function in Section \ref{sec::trees}, but in this case we \emph{do not} ``backtrack'' along branches of the tree (see Figure \ref{fig::dfe}). 

Recall that if $|T|$ is the total number of individuals in $T$, then we can order them $\{u_\emptyset:=u_0,u_1,\cdots, u_{|T|}\}$ with respect to the depth first order. The depth first exploration is then the continuous time process $v_t$ on $[0,\sum_{0\leq j\leq |T|} l_{u_j}):=[0,L(T))$, taking values in $T$, defined by
\begin{equation}
	\label{eqn::dfe_defn}
	v_t := u_i \;\;\;\;\; \text{ if } \;\;\;\;\; \sum_{j<i} l_{u_j} \leq t <\sum_{j\le i} l_{u_j}
\end{equation}
where we recall the definition of the lifetime $l_v$ of an individual $v$ from Section \ref{sec::trees}. Informally, this process visits the individuals of $T$ in depth first order, and spends time $l_u$ visiting individual $u$. We let $\kappa_t=\sum_{j<i} l_{u_j}$ if $v_t=u_i$, so $\kappa_t=\inf\{s\leq t: v_s=v_t\}$ is the time at which the depth first exploration ``starts visiting" the individual $v_t$. Although this depth first exploration takes values in $T$, we should think of it as a continuous exploration of $\mathbf{T}$, where between visiting vertices, the branches of $\mathbf{T}$ are traversed at speed one. See Figure \ref{fig::dfe} for an illustration: note that we shall always refer to the exploration as an exploration of $\mathbf{T}$ rather than an exploration of $T$, in order to keep the correct intuition.

\begin{figure}[h]
	\centering
	\includegraphics[scale=0.5]{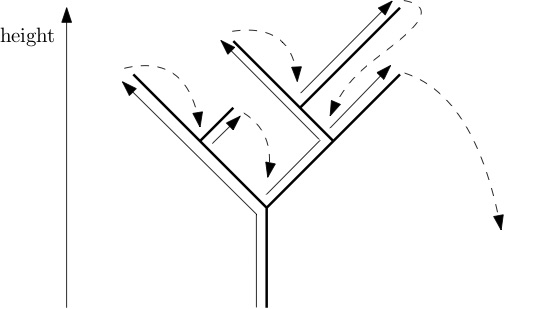}
	\caption{A sketch of the depth first exploration of a continuous tree. Full arrowed lines represent motion at speed one in the vertical direction. Dotted arrowed lines represent instantaneous jumps.}
	\label{fig::dfe}
\end{figure}

Now we describe the \emph{height function} associated to the depth first exploration. This is defined by 
\begin{equation} \label{eqn::height_function}
	H_t := (s_{v_t}-l_{v_t})+(t-\kappa_t),
\end{equation} 
so describes the evolution of \emph{height} as we traverse, if we think of our depth first exploration process as described in the previous paragraph (again see Figure \ref{fig::dfe}). We can think of this as a ``non-backtracking'' version of the contour function from Section \ref{sec::trees}. The height function will be useful for the proof of Theorem \ref{thm::crtconv} for the following reason: 

\begin{remark0} \label{rmk::h_metric} Consider the interval $[0,L(T))$ on which the depth first exploration is defined, and for $r,w\in [0,L(T))$ set \begin{equation}
		\label{defn::stard}
		d^\star(r,w)=H_r+H_w-2\inf_{s\in [r\wedge w, r\vee w]}H_s.
	\end{equation}
	Then this defines a pseudo-metric on $[0,L(T))$: we can have $d^\star(r,w)=0$ if $r$ and $w$ correspond to two times at which a ``branch point" on $\mathbf{T}$ is visited. We can therefore define a metric space by quotienting $[0,L(T))$ using the equivalence relation $r \sim^\star w \iff d^\star(r,w)=0$. It is easy to check that $(\frac{[0,L(T))}{\sim^\star}, d^\star)$ is $\mathbb{P}^x$-almost surely isometrically equivalent to $(\mathbf{T},d)$.
\end{remark0} In fact, for the proof of Theorem \ref{thm::crtconv}, in the end we will prove that if we look at $(\frac{[0,L(T))}{\sim^\star},\frac{1}{t}d^\star)$ under the law $\mathbb{P}^x(\cdot \, \big| \, |N_t|>0)$, then \emph{these} random metric spaces will converge, in distribution as $t\to \infty$, to the Brownian CRT $(\mathbf{T}_\ee, d_\ee)$. Remark \ref{rmk::h_metric} tells us that this implies the theorem. The reason we choose to study the height function rather than the contour function described in Section \ref{sec::trees} is because, as we will soon see, there is a natural way of defining a martingale that is very closely related to the height function (and as soon as we introduce backtracking into the exploration of the tree, the martingale structure is lost). 

 Finally, it should be clear that any time $t$ in our depth first exploration of the tree is naturally associated with a position in $D$. That is, if $v_t$ has associated motion $X_{v_t}: [s_{v_t}-l_{v_t},s_{v_t})\to D$, then we can write  
\begin{equation}
	\label{eqn::position_dfe}
	V_t:= X_{v_t} ((s_{v_t}-l_{v_t})+(t-\kappa_t)).
\end{equation}This process $V_t$ then describes the evolution of position in $D$, if we do a continuous depth first exploration of $\mathbf{T}$, and we follow the motion of an individual $u$ over its lifetime during the time that we are visiting it.
\\

Now we will start setting up for the proof of Theorem \ref{thm::crtconv}. A classical technique in proofs of this sort (see for example \cite{crtconvgd, legallrandomtrees}) is, instead of considering one tree conditioned to be large, to consider an i.i.d. sequence of trees without any conditioning. Describing the scaling limit of this process then allows one to also describe the scaling limit of a single ``large" tree.

So, we write $\bar{\mathbb{P}}^x$ for the law of a sequence of i.i.d trees $((T^1,l^1,X^1),(T^2,l^2,X^2), \cdots)$, where each $(T^i,l^i,X^i)$ has the same distribution as $(T,l,X)$ under $\mathbb{P}^x$. We concatenate the continuous depth first explorations of the trees in the natural way, and write $\bar{H}_t$ for the associated concatenation of the height processes. To show the convergence in Theorem \ref{thm::crtconv} it will be important to show that $\bar{H}_t$ when rescaled appropriately, looks like a reflected Brownian motion. To do this, we introduce a further process, $\bar{S}_t$, which will turn out to be a martingale. This can be thought of as an analogue of the Lukasiewicz path used in \cite{crtconvgd}.

In the following, we write 
\begin{equation}
	\label{eqn::index}
	\Lambda_t := \inf \{j\geq 1\, : \, \sum_{i=1}^j L(T^j) > t \}
\end{equation}
for the index of the tree being visited at time $t$ in the continuous depth first exploration. We also write $$\bar{v}_t:=(v_t,\Lambda_t)\in T^{\Lambda_t}\times \{\Lambda_t\} \subset\Omega \times \N,$$ where $v_t$ is the individual in $T^{\Lambda_t}$ being visited at time $t$ in the exploration of the sequence of trees. We still write $V_t$ for the associated position as defined in (\ref{eqn::position_dfe}). Finally, we say that for any $u,v\in \Omega$, $w$ is a \emph{younger sibling} of $u$, denoted $w<_{s} u$, if there exist $j<k\in \N$ and $v\in \Omega$ such that $u=vj$ and $w=vk$. If we have a tree $T$ and an individual $u\in T$ we write 
\begin{equation*}
	Y(u,T)=\{w\in T : w<_s v \text{ for some } v\prec u\}
\end{equation*}
for the set of younger siblings of ancestors of $u$.  
\begin{defn}\label{def::st} Suppose we have a sequence of trees $((T^i,l^i,X^i))_{i\geq 0}$. We define for $t\geq 0$
	\begin{equation} \label{eqn::bar_st_def}
		\bar{S}_t= \vp(V_t) - \sum_{i\le \Lambda_t}\vp(X^i_{\emptyset}(0)) + \sum_{w\in Y(v_t,T^{\Lambda_t})} \vp(X^{\Lambda_t}_{w}(s^{\Lambda_t}_w-l^{\Lambda_t}_w)).  
	\end{equation}
	
\end{defn}

Informally, the process $(\bar{S}_t)_{t\ge 0}$ can be defined as follows. We do our continuous depth first exploration of the sequence of trees and, until the first ``branching'' or ``death'' time, $\bar{S}_t$ is just equal to $-\vp(x)$, plus a term that follows $(\vp(V_t))_{t\ge 0}$: that is, $\vp$ applied to the position of the individual we are visiting at time $t$. The initial $-\vp(x)$ term is just included for convenience, so that $\bar{S}_0=0$. As the exploration evolves, we always keep the term that follows $\vp(V_t)$, but also, whenever we pass a branch point \footnote{that is, a point on the tree corresponding to one individual dying, and being replaced with a strictly positive number of offspring} at time $t$, we add on $\vp(V_t)$ times the number of particles born minus 1. Whenever we reach a leaf (but not the end of a tree) we jump down to the next particle to be visited in the depth first exploration, and then $\vp$ of the position of that particle, which was before included in the third term of (\ref{eqn::bar_st_def}), becomes the $\vp(V_t)$ term: the first term of (\ref{eqn::bar_st_def}). When we reach the end of a tree, we subtract $\vp(x)$, but then start the depth first exploration of the next tree (so in particular $\vp(V_t)$ at such a time becomes $\vp(x)$).

In fact, $(\bar{S}_t)_{t\ge 0}$ is very closely related to the martingale $(M_t)_{t\ge 0}$ defined in Section \ref{sec::mgales}. Essentially they are the same process, but ``explored" in different orders, and we will see that this is enough to preserve the martingale property. The overall idea is that we would like to approximate the height process $(\bar{H}_t)_{t\ge 0}$ by $(\bar{S}_t)_{t\ge 0}$, and then apply an invariance principle for the latter, which we can do because it is a martingale. This is an analogous idea to that used to prove convergence of Galton--Watson processes to the CRT in \cite{crtconvgd}, where $(\bar{S}_t)_{t\ge 0}$ here plays the role of the \emph{Lukasiewicz path.} We first record a property of this process, which will be essential to showing a relationship with the height function. It states that a slight modification of $(\bar{S}_t)_{t\ge 0}$ is equal to a certain explicit process, that will later turn out to be (roughly) proportional to $(H_t)_{t\ge 0}$.
\begin{lemma}\label{lemma::pathrel} For $t\ge 0$, let $\bar{S}'_t=\bar{S}_t-\vp(V_t)$ and $\bar{I}_t'=\inf_{0\leq s \leq t}\bar{S}_s'$. Then 
	\[ \bar{\mathbf{S}}_t:=\bar{S}_t'-\bar{I}_t'= \sum_{w\in Y(v_t,T^{\Lambda_t})} \vp(X^{\Lambda_t}_{w}(s^{\Lambda_t}_w-l^{\Lambda_t}_w))  \]
\end{lemma}
\begin{proof} Since $\vp$ is positive, it is clear that $\bar{I}_t'=-\sum_{i \le \Lambda_t} \vp(X^i_\emptyset(0))$. This implies the result.
\end{proof}

\begin{defn}\label{rmk::sonetree}
	If we have a single tree $(T,X,l)$, with continuous depth first exploration process $(v_t)_{t\ge 0}$, we define the corresponding processes $(S_t)_{t\ge 0}$ and $(\mathbf{S}_t)_{t\ge 0}$ by 
	\begin{equation*}
S_t:= \left\{
	\begin{array}{ll}
	\vp(V_t)+  \sum_{w\in Y(v_t,T)} \vp(X_w(s_w-l_w))-\vp(x);  &  t\leq L(T) \\
	-\vp(x); & t>L(T)
	\end{array}
	\right. \end{equation*} and 
		\begin{equation*}
	\mathbf{S}_t:= \left\{
	\begin{array}{ll}
	\sum_{w\in Y(v_t,T)} \vp(X_w(s_w-l_w));  &  t\leq L(T) \\
	0; & t>L(T).
	\end{array}
	\right. \end{equation*}

\end{defn}

\begin{remark0} \label{remark::excursions_trees} Note that $(\mathbf{S}_t)_{t\ge 0}$ is a process that starts and ends at $0$ (at time $L(T)$) and is positive in between, i.e. an \emph{excursion}. 
	Moreover, for $y>0$ the law of $(\mathbf{S}_t)_t$ under $\mathbb{P}^x$, conditioned on $\{\sup_{t} \mathbf{S}_t \geq y\}$, is the same as the law of $(\bar{\mathbf{S}}_t)_{t\ge 0}$ under $\bar{\mathbb{P}}^x$, restricted to the first excursion in which it exceeds $y$. 
\end{remark0}

\subsection{Martingale Convergence }

We let $(\bar{\mathcal{F}}_t)_{t\ge 0}$ be defined for each $t\ge 0$ by $\bar{\mathcal{F}}_t:=$
\begin{equation*}
	\label{eqn::df_filtration}
\sigma \{ \{\cup_{i<\Lambda_t} (T^i,l^i,X^i)\},\{(l^{\Lambda_t}_w,X^{\Lambda_t}_w,A^{\Lambda_t}_w) : \bar{v}_r=(w,\Lambda_t), r<\kappa_t\},\{X_{v_t}^{\Lambda_t}(r): r\in [s_{v_t}^{\Lambda_t}-l_{v_t}^{\Lambda_t}, s_{v_t}^{\Lambda_t}-l_{v_t}^{\Lambda_t}+(t-\kappa_t)]\} \}.
\end{equation*}
That is $(\bar{\mathcal{F}}_t)_{t\ge 0}$ is the filtration that knows everything about the continuous depth first search of the sequence of trees $((T^i,l^i,X^i))_i$ up to time $t$. It encodes: (i) which individuals in which trees are visited and when before time $t$; (ii) their spatial motions (although for $v_t$ only up until the point in its lifetime that has been explored by time $t$); and (iii) their progenies, or offspring, (apart from that of $v_t$).

\begin{lemma}
	\label{lemma::s_mgale}
	Under $\bar{\mathbb{P}}^x$ the process $(\bar{S}_t)_{t\geq 0}$ is a martingale with respect to $(\bar{\mathcal{F}}_t)_{t\geq 0}$. 
\end{lemma}

Before we prove the lemma, we  show that $(\bar{S}_t)_{t\ge 0}$ is locally square integrable, which will be of technical importance in what follows.

\begin{lemma}
	\label{lemma::s_sq_int} For every $x\in D$ and $R\ge 0$, 
	$\sup_{t \le R} \bar{\mathbb{P}}^x[(\bar{S}_t)^2]<\infty$.
\end{lemma}

\begin{proof2}{Lemma \ref{lemma::s_sq_int}}
	Fix $R\ge 0$. Since $\vp$ is bounded it suffices to show that $\Lambda_t$ and $|Y(v_t,\Lambda_t)|$ have uniformly bounded second moments under $\bar{\mathbb{P}}^x$ for $t\le R$. 
	\begin{itemize}
		\item For $\Lambda$, we simply note that we can stochastically dominate $\Lambda_t$ for any $t\le R$ by a geometric random variable with success probability $\mathbf{P}^x(\tau_D>R)$. 
	\item For $|Y|$, it is a little more complicated. First observe that, for any $t\ge 0$, if we condition on the time $L(T^1)+\cdots + L(T^{\Lambda_t-1})$ that we start exploring the $\Lambda_t^{\text{th}}$ tree in the sequence, the conditional law of the $\Lambda_t^{\text{th}}$ tree is just that of a single tree under $\mathbb{P}^x$, conditioned on having $|N_{t-L(T^1)+\cdots + L(T^{\Lambda_t-1})}|>0.$ Since the probability of this event is bounded below by $\mathbf{P}^x(\tau^D>R)$, and $|Y(v_t,\Lambda_t)|$ is bounded above by $\max\{|Y(u,\Lambda_t)|\, : \, u\in T^{\Lambda_t}, s_u^{\Lambda_t}-l_u^{\Lambda_t}\le R \}  $, it is enough to show that 
	\begin{equation}\label{eqn::Yl2bound} \mathbb{P}^x[(\max\{|Y(u,T)|\, : \, u\in T, s_u-l_u\le R \})^2]<\infty
	\end{equation} (note this is now an expectation for a single branching process). However this follows from the assumption that $A$ has finite variance. Indeed, consider a pure (i.e. no spatial motion) branching process in continuous time, that branches at rate $\lambda/(m-1)$ and has offspring distribution $A'=A+1$ (so that there is always a positive number of offspring born). Then if $\mathcal{N}_R$ is the total size of the population of this process at time $R$, it is clear that $\mathcal{N}_R$ stochastically dominates $\max\{|Y(u,T)|\, : \, u\in T, s_u-l_u\le R \}$. Moreover, it is easy to check that $\mathcal{N}_R$ has finite variance (for example, by conditioning on the first branching time). 
		\end{itemize}

\end{proof2}

\begin{proof2}{Lemma \ref{lemma::s_mgale}} Before commencing with the technicalities of the proof, we would like to emphasise that the reason $(\bar{S}_t)_{t\ge 0}$ is a martingale is exactly the same as the reason that $(M_t)_{t\ge 0}$ from \eqref{eqn::mgale_outline} is a martingale. The branching Markov structure of the process is preserved when we explore in the depth first order, and this means that $(\bar{S}_t)_{t\ge 0}$ being a martingale is equivalent to it having constant expectation. The fact that it has constant expectation is intuitively clear because, after a small time $\delta$, $\mathbb{P}^x(\vp(V_\delta))\approx \e^{-\lambda \delta}\vp(x)\approx \vp(x)(1-\lambda \delta)$ and branching contributes $\lambda \delta \vp(x)$ to the expectation (the fact that we are exploring in the depth first order does not affect this).
	
Let us now be more precise. The first step in the proof is to show that if $(S_t)_{t\ge 0}$ is the process $(\bar{S}_t)_{t\ge 0}$ for a single tree (as described in Definition \ref{rmk::sonetree}) then \begin{equation}\label{eqn::const_exp} \mathbb{P}^x[S_t]=0\;\; \forall \; t\ge 0\end{equation} We will then show how this implies the martingale property. 

To see \eqref{eqn::const_exp} we define a sequence of times $R_1, R_2,$... as follows: let $R_1=s_\emptyset$, and given $\{R_1,..., R_k\}$ set $R_{k+1}=R_k+l_{v_{R_k}}$. In words, $(R_k)_{k\ge 0}$ is the sequence of times at which the particle we are visiting during the depth first exploration of the tree either branches or exits the domain. We will prove by induction that for any $n\ge 1$ 
\begin{equation}\label{eqn::sonetree_constexp} \mathbb{P}^x[S_{t\wedge R_n}] = 0.  \end{equation}
Then since $\mathbb{P}^x[S_t]= \mathbb{P}^x[S_{t\wedge R_n}] + \mathbb{P}^x[S_t \I_{\{t> R_n\}}] - \mathbb{P}^x[S_{t\wedge R_n}\I_{\{t> R_n\}}]$ and the second two terms tend to $0$ as $n\to \infty$ (by Cauchy--Schwarz, \eqref{eqn::Yl2bound}, and the fact that $\mathbb{P}(t> R_n)\to 0$) this implies \eqref{eqn::const_exp}.

For the induction, first note that in the case $n=1$ we can immediately write $\mathbb{P}^x[S_{t\wedge R_1}]=\mathbb{P}^x[M_{t\wedge R_1}]-\vp(x)=0$.
Then for the inductive step we consider $\mathbb{P}^x[S_{t\wedge R_{n+1}}-S_{t\wedge R_n}]$, and condition on everything that has happened in the depth first exploration up to time $t\wedge R_n$. In the case that $t\le R_n$ or $t\ge L(T)$ the conditional expectation of $S_{t\wedge R_{n+1}}-S_{t\wedge R_n}$ is clearly equal to $0$. Moreover, on the complementary event we can use exactly the same argument as for the $n=1$ case to see that the conditional expectation is equal to $0$. This proves \eqref{eqn::sonetree_constexp} for all $n\ge 1$. 
	
Finally, we must use this to prove the martingale property. The idea is that for $t>s$, $\bar{\mathbb{P}}^x(\bar{S}_t-\bar{S}_s\, |\, \bar{\mathcal{F}}_s)$ can be written as a sum of expectations of the form \eqref{eqn::const_exp}: one for each of the subtrees rooted at the elements of $Y(v_s,\Lambda_s)$ plus one for each  $(T^j)_{j>\Lambda_s}$, where the time $t$ in \eqref{eqn::const_exp} is replaced in each expectation by a random time that is (importantly) independent of the tree/subtree being considered. More precisely, we can write for $0\leq s \leq t$ 
	\begin{align*}\label{eqn::mgale_s}
		\bar{\mathbb{P}}^x(\bar{S}_t-\bar{S}_s \, |\, \bar{\mathcal{F}}_s) & = \E{V_s}{S_{(t-s)}}-\vp(V_s) + \sum_{w\in Y(v_s,\Lambda_s)} \big((\mathbb{P}^{X^{\Lambda_t}_{w}(s^{\Lambda_t}_w-l^{\Lambda_t}_w)}\otimes \mu_w )(S_{\sigma(w)})-\vp(X^{\Lambda_t}_{w}(s^{\Lambda_t}_w-l^{\Lambda_t}_w))
		\big) \\ \nonumber 
		& + \sum_{i=1}^\infty \big((\mathbb{P}^x \otimes \mu_i)(S_{\sigma_i})-\vp(x)\big)
	\end{align*}
	where $\otimes$ represents the independent product measure, the $\mu_w$ are random laws on times $\sigma(w)$ for each $w\in Y(v_s,\Lambda_s)$, and the $\mu_i$ are random laws on times $\sigma_i$ for each $i\geq 1$. \footnote{For example, if $w$ is the oldest younger sibling of $v_t$ (so $v_t=uj$ and $w=u(j+1)$ for some $j\ge 1$) then $\sigma(w)$ would be $(t-s)$ minus the minimum of $(t-s)$ and the total length of the subtree rooted at $v_s$.}
By \eqref{eqn::const_exp} we see that the right hand side of the above is equal to $0$.
\end{proof2}

\begin{lemma}\label{lemma::mgaleqv}
	For any $x\in D$, $(\bar{S}_t)_{t\geq 0}$ is a locally square-integrable martingale under $\bar{\mathbb{P}}^x$. Its predictable quadratic variation is given by 
	\[\langle \bar{S} \rangle_t = \int_0^t \frac{\lambda}{m-1}\mathbb{E}(A^2-2A+1)\vp(V_s)^2+ 2\sum_{i,j} a^{ij}(V_s) \frac{\st \vp}{\st x_i} \frac{\st \vp}{\st x_j}(V_s)  \, ds.\]
\end{lemma}
\begin{proof}
We already know from Lemmas \ref{lemma::s_mgale} and \ref{lemma::s_sq_int} that $(\bar{S}_t)_{t\geq 0}$ is a locally square-integrable martingale. To calculate its predictable quadratic variation, let us look at the definition a little more closely, and try to break up $(\bar{S}_t)_{t\ge 0}$ into continuous and jump parts. 

It is clear that all the terms in (\ref{eqn::bar_st_def}), apart from the $\vp(V_t)$ term, are constant except at times where $\bar{v}_t$ (the vertex we are visiting at time $t$ in the depth-first exploration) changes.  It is also clear that $\vp(V_t)$ is continuous away from these times. Note that the number of such times is a.s. finite on any finite time interval (for example, since the expectation of the number of such times is finite). We write $V_{t-}=\lim_{s\uparrow t} V_s$, which exists since $(V_t)_t$ is cadlag. Similarly, since the process $t\mapsto \bar{v}_t=(v_t,\Lambda_t)$ is cadlag with respect to the discrete topology on $\Omega \times \N$, we can define $\bar{v}_{t-}=\lim_{s\uparrow t} \bar{v}_s$. By the above discussion, any discontinuities of $\bar{S}_t$ must occur at times $t$ such that $\bar{v}_{t-}\ne \bar{v}_t$. 
	
	We consider the different possibilities for such times: 
	\begin{enumerate}[(I)]
		\item $V_{t-} \in \partial D$. This means that $\vp(V_{t-})=0$, and so $ \vp(V_t)-\vp(V_{t-})=\vp(V_t)$. In this case, there are two further possibilities: (a) $\Lambda_{t-}=\Lambda_t$ and $\bar{v}_t=(w,\Lambda_{t-})$ for some $w\in Y(v_{t-},T^{\Lambda_{t-}})$ or (b) $\bar{v}_t=(\emptyset, \Lambda_t)$ with $\Lambda_t\ne \Lambda_{t-}$. In case (a) we see that the third term on the right hand side of (\ref{eqn::bar_st_def}) also decreases by $\vp(V_t)$ at time $t$, and so $ \bar{S}_t-\bar{S}_{t-} = 0$. Similarly, in case (b) the second term on the right hand side of (\ref{eqn::bar_st_def}) decreases by $\vp(x)=\vp(V_t)$ at time $t$, so again $\bar{S}_t-\bar{S}_{t-}$ is equal to $0$. Overall, \emph{such a situation does not actually correspond to a discontinuity of $\bar{S}_t$.}
		\item $V_{t-}\notin \partial D$ but $V_{t-}\ne V_t$. Then $A_{v_{t-}}=0$ and again $\bar{v}_t$ satisfies either (a) or (b) defined above. By  the same reasoning, we have that $\bar{S}_t-\bar{S}_{t-}=-\vp(V_{t-})=(A_{v_{t-}}-1)\vp(V_{t-})$.
		\item $V_{t-}=V_{t}$ but $\bar{v}_t\ne \bar{v}_{t-}$. This corresponds to the continuous depth first exploration reaching a branching point at time $t$ that is not the end of a branch. In this case we have $A_{v_{t-}}\ge 1$ and $\bar{S}_t-\bar{S}_{t-}=(A_{v_{t-}}-1)\vp(V_{t})$. 
	\end{enumerate} 
	With this is mind, we define 
	\begin{equation}\label{eqn::s_decomp}
		\bar{S}_t^{(c)}:=\vp(V_t)-\vp(x) -\sum_{s\le t} (\vp(V_s)-\vp(V_{s-})) \;\;\; \text{and} \;\;\; \bar{S}_t^{(j)}:= \sum_{\substack{ (w,i)=\bar{v}_s \\ \text{ for some } s<\kappa_t}} \vp(X_w^i(s^i_w))(A^i_w-1).
	\end{equation}
	Observe that \[\bar{S}_t-\bar{S}_t^{(c)}:= \sum_{s\le t} (\vp(V_s)-\vp(V_{s-}))- \sum_{i\le \Lambda_t}\vp(X^i_{\emptyset}(0))+ \sum_{w\in Y(v_t,T^{\Lambda_t})} \vp(X^{\Lambda_t}_{w}(s^{\Lambda_t}_w-l^{\Lambda_t}_w))\] is equal to the sum of the jumps of $\vp(V_s)$ that occur before time $t$, plus the sum of the jumps of $\bar{S}_s-\vp(V_s)$ that occur before time $t$. 
	This implies that $(\bar{S}^{(c)}_t)_{t\ge 0}$ is a continuous process. Moreover, together with our previous considerations, it tells us that $\bar{S}_t-\bar{S}_t^{(c)}$  is equal to $\bar{S}_t^{(j)}$ for each $t$. Indeed, at times where (I) is satisfied, the two types of jump that make up $\bar{S}_t-\bar{S}_t^{(c)}$ cancel each other out, and at times where (II) or (III) is satisfied, the jumps are precisely the jumps of $\bar{S}_t^{(j)}$. 
	
	In conclusion, we can write $\bar{S}_t = \bar{S}_t^{(c)}+\bar{S}_t^{(j)}$ for $t\ge 0$, thus decomposing $\bar{S}$ into a continuous process and a jump process. This means that if we write $\upsilon(\omega, dt, dx)$ for the predictable compensator of $\bar{S}_t^{(j)}$ (see \cite[\S II, Theorem 1.8]{limitthms} for the definition) then we have \cite[\S II, Proposition 2.29]{limitthms}
	\begin{equation}\label{eqn::predictable_quad_var}
		\langle \bar{S} \rangle_t = [\bar{S}^{(c)}]_t+\int_0^t \int_{\R} x^2 \upsilon(\omega, ds, dx). 
	\end{equation}
	where $[\cdot]$ denotes the ordinary quadratic variation. Since $\bar{S}^{(c)}$ is continuous, and its increments are that of $\vp$ applied to an $L$-diffusion in $D$ , we have that
	\begin{equation*}
		[\bar{S}^{(c)}]_t = 2\int_0^t \sum_{i,j} a^{ij}(V_s) \frac{\st \vp}{\st x_i} \frac{\st \vp}{\st x_j}(V_s)  \, ds.
	\end{equation*}
	Moreover, one can check (straight from the definition in \cite{limitthms}) that \begin{equation}
		\label{eqn::predic_compensator}
		\upsilon(\omega, ds, dx)=\sum_{k\geq 0} \mathbb{P}(A=k) \delta_{(k-1)\vp(V_{s-}(\omega))}(dx) \times \flm ds .
	\end{equation}
	From this and (\ref{eqn::predictable_quad_var}), the lemma then follows.
\end{proof}

\begin{propn}\label{propn::mgaleinv} 
	Let \[ \sigma^2 = \frac{2b^{-1} }{(1,\vp)} = \frac{\lambda (1,\vp^3) (\mathbb{E}(A^2)-\mathbb{E}(A))}{(m-1) (1,\vp)}. \] 
	Then \[ \left( \frac{\bar{S}_{nt}}{\sqrt{n}} \right)_{t\geq 0} \to \left( \sigma B_t \right)_{t\geq 0} \]
	in distribution as $n\to \infty$, with respect to the Skorohod topology, where $(B_t)_t$ is a standard one dimensional Brownian motion.
\end{propn}

\begin{remark0}\label{rmk::nnotint}
	In fact, $n$ need not be an integer here (ie. we can let $n\to \infty$ in $\R$). We use the notation $n$ so as to clearly distinguish it from the continuous time parameters used to describe the evolution of the process. The same remark applies to Propositions \ref{propn::convtoex} and \ref{propn::convtoex2} below.
\end{remark0}

\begin{proof}
	This follows from the functional central limit theorem for martingales \cite[\S VIII, Theorem 3.22]{limitthms} once we can show that for all $t\geq 0$
	\begin{equation}\label{eqn::pred_conv}\langle \bar{S}^n \rangle_t \to \sigma^2 t  \end{equation}
	in probability as $n\to \infty$. Here $\left(\bar{S}_t^n \right)_{t\geq 0} = \left( \bar{S}_{nt}/\sqrt{n} \right)_{t\geq 0}$. \footnote{In fact, since $(\bar{S}_t^n)_t$ may have jumps, we also need to verify an extra condition in order to apply the functional central limit theorem. However this condition, see \cite[\S VIII, Eq.(3.23)]{limitthms}, is simply that we have, for every $t>0$,
	$$\flm \int_0^{t} \vp(V_{s-}(w))^2\mathbb{E}\big((A-1)^2 \I_{\{(A-1)^2 \vp(V_{s-}(w))^2>n\eps\}}\big) \, ds \to 0 $$
	almost surely as $n\to \infty$ (where the expectation $\mathbb{E}$ is only over $A$). Since $A$ has finite variance, this does indeed hold.} 
	
	To show the convergence (\ref{eqn::pred_conv}), we observe that we can write
	\[ \frac{\langle \bar{S}^n \rangle_t}{t} = \frac{1}{nt} \int_0^{nt}\frac{\lambda}{m-1}\mathbb{E}(A^2-2A+1)\vp(V_s)^2+ 2\sum_{i,j} a^{ij}(V_s) \frac{\st \vp}{\st x_i} \frac{\st \vp}{\st x_j}(V_s)  \, ds. \]
	Then, since $\vp$ and all of its first order derivatives are bounded (recall we are assuming this in this section), the result follows immediately from Proposition \ref{propn::qvconv} below. Indeed the proposition gives that 
	\[ \frac{\langle \bar{S}^n \rangle_t}{t} \to \frac{\lambda \mathbb{E}[A^2-2A+1]}{(1,\vp) (m-1)} \int_D \vp(x)^3 \, dx + \frac{2}{(1,\vp)} \int_D \vp(x) \sum_{i,j} a^{ij} (x)\frac{\st \vp}{\st x_i} \frac{\st \vp}{\st x_j}(x) \, dx, \]
	as $t\to \infty$, where the right hand side, by applying (\ref{eqn::evalue}) with $v(x)=\vp(x)^2$, is equal to 
	\[ \frac{\lambda (1,\vp^3)}{(m-1)(1,\vp)} \big( \mathbb{E}[A^2-2A+1] - (m-1) \big)=\sigma^2.\]
\end{proof}

\begin{propn}
	\label{propn::qvconv}
	Suppose that $f$ is a bounded, measurable function. Then 
	\[Q_t := \frac{1}{t}\int_0^t f(V_s) \, ds \to \frac{\ip{f}{\vp}}{\ip{1}{\vp}}\] in $\bar{\mathbb{P}}^x$-probability as $t\to \infty$. 
\end{propn}

Before we prove Proposition \ref{propn::qvconv}, let us record some of the consequences of Proposition \ref{propn::mgaleinv}. 

\begin{propn} 
	\label{propn::convtoex}
	Under $\bar{\mathbb{P}}^x$, we have the joint convergence
	\begin{equation}\label{eqn::jointconv}
		\left( \frac{\bar{\mathbf{S}}_{nt}}{\sqrt{n}}, \frac{\Lambda_{nt}}{\sqrt{n}}\right)_{t\geq 0} \to \left(\sigma |\beta_t|, \frac{\sigma}{\vp(x)}L_t^0(\beta)\right)_{t\geq 0} 
	\end{equation}
	as $n\to \infty$, in distribution with respect to the Skorohod topology. Here, $\beta$ is a standard Brownian motion started at $0$ and $L_t^0(\beta)$ is the local time of $\beta$ at $0$.\end{propn}

\begin{propn} \label{propn::convtoex2} For any $y>0$, under $\mathbb{P}^x( \cdot \, \big| \, \{\sup_t \mathbf{S}_t\geq \sqrt{n} y\} )$ (note that we are now considering the law of a single tree), we have 
	\begin{equation}
		\label{eqn::excconv}
		\left(\frac{\mathbf{S}_{nt}}{\sqrt{n}}\right)_{t\geq 0} \overset{(d)}{\underset{n\to \infty}{\longrightarrow}} \left(\sigma \ee^{\ge y/\sigma}_t\right)_{t\geq 0} 
	\end{equation}
	where $(\ee^{\ge y/\sigma}_t)_{t\geq 0}$ is a Brownian excursion conditioned to reach (at least) height $y/\sigma.$ 
\end{propn}

\begin{proof2}{Proposition \ref{propn::convtoex}}
	From the definition of $\bar{S}'_t=\bar{S_t}-\vp(V_t)$, we have that $|\bar{S}_{nt}'/\sqrt{n}-\bar{S}_{nt}/\sqrt{n}|\leq \|\vp\|_\infty/\sqrt{n}$ for all $t\geq 0$ and so Proposition \ref{propn::mgaleinv} implies that 
	\[ \left( \frac{\bar{S}_{nt}'}{\sqrt{n}} \right)_{t\geq 0} \to \left( \sigma B_t \right)_{t\geq 0} \] as $n\to \infty$ as well. Writing $\underline{B}_t=\inf_{0\leq s\leq t} B_s$ this implies the joint convergence 
	\begin{equation}\label{eqn::jc1}  \left( \frac{\bar{\mathbf{S}}_{nt}}{\sqrt{n}}, -\frac{\bar{I}_{nt}'}{\sqrt{n}}\right)_{t\geq 0} =\left( \frac{\bar{S}_{nt}'-\bar{I}_{nt}'}{\sqrt{n}}, -\frac{\bar{I}_{nt}'}{\sqrt{n}}\right)_{t\geq 0} \overset{(d)}{\underset{n\to \infty}{\longrightarrow}} \left(\sigma (B_t-\underline{B}_t),-\sigma\underline{B}_t \right)_{t\geq 0} \end{equation}
	where the right hand  side by L\'{e}vy's theorem, see for example \cite[\S VI, Theorem VI.2.3]{revuzyor}, is equal in distribution to \[\left(\sigma |\beta_t|, \sigma L_t^0(\beta)\right)_{t\geq 0}. \] However, we know that $\Lambda_t=-\bar{I}_t'/\vp(x)$, and so (\ref{eqn::jointconv}) follows. 
\end{proof2}

\begin{proof2}{Proposition \ref{propn::convtoex2}} For this result, we follow \cite[Proposition 2.5.2]{crtconvgd}. It is well known that you can construct the process $\ee^{\ge y/\sigma}$ from a standard Brownian motion $\beta$ by taking \[\ee_t^{\ge y/\sigma}=|\beta_{(G+t)\wedge D}|\]where $T=\inf\{t\geq 0: |\beta_t|\geq y/\sigma\}$, $G=\sup\{t\leq T: \beta_t=0\}$ and $D=\inf\{t\geq T: \beta_t=0\}.$ By the Skorohod representation theorem and (\ref{eqn::jointconv}) we also know that there exists a process
	\[\left(Z^{(n)}_t, \Lambda^{(n)}_t\right)_{t\geq 0}\overset{(d)}{=}\left( \frac{\bar{S}_{nt}'-\bar{I}_{nt}'}{\sqrt{n}}, \frac{\Lambda_{nt}}{\sqrt{n}}\right)_{t\geq 0} \]
	such that 
	\[ \left(Z^{(n)}_t, \Lambda^{(n)}_t\right)_{t\geq 0}\underset{n\to \infty}{\longrightarrow} \left(\sigma |\beta_t|, \frac{\sigma}{\vp(x)}L_t^0(\beta)\right)_{t\geq 0} \]
	uniformly on every compact set almost surely. This is because Skorohod convergence is equivalent to local uniform convergence when the limit is continuous. Define $T^{(n)}=\inf\{t\geq 0: Z^{(n)}\geq y\}$ for this sequence of processes, and $G^{(n)}$, $D^{(n)}$ in the same way as $G$ and $D$ above. By Remark \ref{remark::excursions_trees} and (\ref{eqn::jointconv}), if we can prove that $G^{(n)}\to G$ and $D^{(n)}\to D$ almost surely, we will be done. To do this, first note that since $\beta$ must exceed $y/\sigma$ immediately after time $T$, we have that $T^{(n)}\to T$ almost surely. This implies straight away that for all $t<D$ we have $t\leq D^{(n)}$ for all $n$ large enough almost surely. Now we must show that for all $t>D$ we have $t\geq D^{(n)}$ for all $n$ large enough almost surely. These facts together (along with the corresponding results for $G$) are enough to prove the convergence. To see the final claim, we use the convergence of the local time. For any $t>D$, we have using basic properties of Brownian local time that $L_t^0>L_D^0=L_T^0$. The convergence of the local time therefore tells us that $\Lambda_t^{(n)}> \frac{\sigma}{\vp(x)} \,L_T^0$ for all $n$ large enough almost surely. Since \[\Lambda_{T^{(n)}}^{(n)}\to \frac{\sigma}{\vp(x)}\,L_T^0\] almost surely, this implies that we have also have $\Lambda_t^{(n)}>\Lambda_{T^{(n)}}^{(n)}$ for all $n$ large enough almost surely. Using the fact that $\Lambda^{(n)}$ stays constant on $[T^{(n)},D^{(n)})$, we see that $t\geq D^{(n)}$. 
\end{proof2}

\subsubsection{Proof of Proposition \ref{propn::qvconv}}

The rest of this subsection will be devoted to the proof of Proposition \ref{propn::qvconv}. We begin by stating some preliminary lemmas that we will need for the proof. The first lemma gives a lower bound of order $t^{-1/2}$ for the probability that a single tree $T$ under $\mathbb{P}^x$ has total size (as measured by the total length $L(T)$ of the continuous depth first exploration of $T$) bigger than $t$. This is a consequence (see the proof) of the fact that, if $T$ is conditioned on reaching height at least $s$ with $s$ large, $T$ will have size of order $s^2$: the same scaling relation that holds for a Brownian excursion. The second and third say that if we pick a time uniformly in $[0,t]$ and look at the height of the particle visited at this time in the concatenated depth first exploration of a sequence of trees under $\bar{\mathbb{P}}^x$, then with high probability the height will not be of order bigger than $\sqrt{t}$, but it will go to $\infty$. Again this is heuristically what one should expect if Theorem \ref{thm::crtconv} is true, since the same holds for the value of a reflected Brownian motion at a uniformly chosen time in $[0,t]$. 
\begin{lemma}
	\label{lemma::length}
	For any $t_0>0$, there exists a constant $c_{t_0}\in (0,\infty)$ such that for all $x\in D$ and all $t\ge t_0$
	\[\mathbb{P}^x(L(T)>t)\geq c_{t_0} \frac{\vp(x)}{\sqrt{t}}.\]
\end{lemma}

Now, let $U_t$ be a sample from the uniform distribution on $[0,t]$, and write $\mu^U_t$ for its law. In the following we write 
\begin{equation}
	H_t^*=H_{U_t} \;\;\; ; \;\;\; \Lambda_t^*= \Lambda_{U_t} \;\;\; \text{and} \;\;\; V_t^*=V_{U_t}
\end{equation}
to simplify notation.

\begin{lemma}\label{lemma::ubigheight} For all $x\in D$
	
	\[\bar{\mathbb{P}}^x \otimes \mu^U_t \left(H_t^* \geq C\sqrt{t}\right) \to 0\]
	as $C\to \infty$, uniformly in $\{t\ge t_0\}$ for any $t_0>0$, where $\otimes$ represents the independent product measure. 
\end{lemma}

	\begin{lemma}\label{lemma::rtoinfinity} For $\delta>0$, define $R^\delta(t)>0$ to be such that $\pbu{x}{H_t^*\leq R^\delta(t)} = \frac{\delta}{3}.$ Then $R^\delta(t)\to \infty$ as $t\to \infty$ for any $x\in D$. \end{lemma}

We will now show how we may deduce Proposition \ref{propn::qvconv} from these lemmas, and then go on to prove them. In the following we set $N_s^i=\{u\in T^i: s\in [s_u^i-l_u^i, s_u^i)\}$ and $\tilde{N}_{s,t}^{i} = \{u\in T^i: \{s\in [s_u^i-l_u^i, s_u^i)\} \cap \{(u,T^i)=\bar{v}_r \text{ for some } r\leq t\}\}$. In words, $N_s^i$ is simply the set $N_s$ for the $i$th tree, and $\tilde{N}_{s,t}^i\subset N_s^i$ is the set of individuals in $N_s^i$ that have been visited in the depth first exploration of the sequence of trees before time $t$.\medbreak

\begin{proof2}{Proposition \ref{propn::qvconv}} The main idea of the proof is to view $Q_t$ as the ``average value'' of $f$ taken over all the positions $(V_s)_{s\in [0,t]}$ discovered before time $t$ in the depth first exploration of the i.i.d. sequence of trees. Recall from Corollary \ref{cor::unifaverage} that if we take the average value of $f$ over all posiitons of particles at height $t$ in a single tree (conditioned to reach height at least $t$) then this converges to $(f,\vp)/(1,\vp)$ in probability as $t\to \infty$. Roughly speaking, this convergence carries over to the average $Q_t$ because the depth first exploration spends most of its time in large trees at relatively large heights.  
	
	 To prove this properly, we define 
	\begin{equation*}
		m_t := \big(\big|\tilde{N}_{H_t^*,\,t}^{\Lambda_t^*}\big|\big)^{-1} \sum\nolimits_{u\in \tilde{N}_{H_t^*,\,t}^{\Lambda_t^*}} f(X_u^{\Lambda_t^*}(H^*_t) )
	\end{equation*} to be the average value of $f$ among the vertices at height $H_t^*$ of the $\Lambda_t^{*\text{th}}$ tree, that have been visited \emph{before time $t$} in the depth first exploration. 
	Observe that  
	\[Q_t=\pbu{x}{m_t \, | \, \sigma((T^i,X^i,l^i)_{i\ge 1})}.\]
	This means that if $m_t$ converges in $\bar{\mathbb{P}}\otimes \mu_t^U$-probability to $(f,\vp)/(1,\vp)$, then it will also converge in $L^2(\bar{\mathbb{P}}\otimes \mu_t^U)$ (because $f$ is bounded), and so by conditional Jensen's inequality, $Q_t$ will converge in $L^2$ and hence in probability to $(f,\vp)/(1,\vp)$ as well. Thus if we can show that for any fixed $\eps>0$,
	\begin{equation}
		\label{eqn::convprob} \pbu{x}{A_t}:= \pbu{x}{| m_t - \frac{\ip{\vp}{f}}{\ip{\vp}{1}}|>\eps} \to 0\end{equation}
	as $t\to \infty$, then Proposition \ref{propn::qvconv} will follow. 
	
	To do this, also fix a $\delta>0$: we will show that $\pbu{x}{A_t}<\delta$ for all large enough $t$, and hence prove \eqref{eqn::convprob}.  First, by Lemma \ref{lemma::ubigheight} we can choose $C$ such that $\pbu{x}{H_t^*\geq C\sqrt{t}} < \frac{\delta}{3}$ for all $t\ge 1$. We can also define $R(t):=R^\delta(t)$ to be such that $\pbu{x}{H_t^*\leq R(t)} = \frac{\delta}{3}$, and by Lemma \ref{lemma::rtoinfinity}, we have that $R(t)\to \infty$ as $t\to \infty$. Now, for the $i$th tree in our exploration and $s,t>0$, let $$m_{s,t}^i=|\tilde{N}_{s,t}^i|^{-1} \sum\nolimits_{u\in \tilde{N}_{s,t}^i} f(X_u^i(s))$$ be the average value of $f$ among the individuals of tree $i$ at height $s$ that are visited before time $t$ in the depth first exploration (so that $m_t=m_{H_t^*,\, t}^{\Lambda_t^*}$). Also write $A_{s,t}^i$ for the event that $|m_{s,t}^i-\ip{\vp}{f}/\ip{\vp}{1}|>\eps$.  Then the above considerations tell us that we can write
	\begin{equation}\label{eqn::mainprob}
		\pbu{x}{A_t} \leq \frac{2\delta}{3} + \frac{1}{t} \pb{x}{\int_{R(t)}^{C\sqrt{t}} \sum_{i=1}^{\infty} \I_{\{i\leq \Lambda_t\}} \I_{A_{s,t}^i} |\tilde{N}_{s,t}^i| \, ds},
	\end{equation} for all $t\ge 1$,
	where we have obtained the expression for the second term on the right hand side by first conditioning on $\{H_t^*,\Lambda_t^*\}$, and then taking expectation over them.
	
	Now by Fubini (since everything is bounded by $1$) we can rewrite 
	\begin{equation}\label{eqn::breakdownprob} \frac{1}{t} \pb{x}{\int_{R(t)}^{C\sqrt{t}} \sum_{i=1}^{\infty} \I_{\{i\leq \Lambda_t\}} \I_{A_{s,t}^i} |\tilde{N}_{s,t}^i| \, ds}=\frac{1}{t} \int_{R(t)}^{C\sqrt{t}} \sum_{i=1}^\infty \pb{x}{\I_{\{i\leq \Lambda_t\}} \I_{A_{s,t}^i}|\tilde{N}_{s,t}^i|} \, ds. \end{equation}
	The idea is, for each $i\geq 1$, to condition on 
	$\mathcal{K}_{i-1}=\sigma(\cup_{j\leq i-1} (T^{j},X^{j},l^{j}))$. Note that $\mathcal{K}_{i-1}$ is independent of the $i$th tree. Moreover, the event $\{i\leq \Lambda_t\}$ and the start time $\tau_i=L(T^1)+\cdots+L(T^{i-1})$ of the $i$th tree, are measurable with respect to $\mathcal{K}_{i-1}$. Thus we can write, for each $i\geq 1$,
	\begin{align} \label{eqn::conditioned_seq_trees}
		\pb{x}{\I_{\{i\leq \Lambda_t\}} \I_{A_{s,t}^i} |\tilde{N}_{s,t}^i| \, \big| \, \mathcal{K}_{i-1}} & =\I_{{\{i\leq \Lambda_t\}}} \E{x}{\I_{A_s(t-\tau_i)} |N_s(t-\tau_i)|} \\ \nonumber
		& = \I_{{\{i\leq \Lambda_t\}}} \E{x}{ \I_{A_s(t-\tau_i)} |N_s(t-\tau_i)| \, \big| \, |N_s|>0}\E{x}{|N_s|>0}
	\end{align}
	where the expectation in the final term is now with respect to a \emph{single} branching diffusion under $\mathbb{P}^x$, $N_s(r):=\{u\in N_s: u=v_h \text{ for some } h\leq r\}$ is the number of particles at level $s$ in this tree that are visited before time $r$ in a depth first exploration of it, and $A_s(r)=\{ \big| |N_s(r)|^{-1}\sum_{u\in N_s(r)} f(X_u(s))-(\vp,f)/(\vp,1)\big|>\eps\}$ is the event that the average of $f$ among the positions of these particles is more than $\eps$ away from $(\vp,f)/(\vp,1)$. 
	
	We record here that by Theorem \ref{criticalsurvivalprob} and Lemma \ref{lem::exp_nt2}, there exists a $K$ such that 
	\begin{equation} \label{eqn::csbounds} \mathbb{P}^x\left[ \left. |N_t|^2 \right| |N_t|>0 \right]^{1/2} \leq Kt, \;\;\; \mathbb{P}^x(|N_t|>0)\leq K/t\end{equation} 
	for all $t\ge 1$. We can also, by Lemma \ref{lemma::length}, choose this $K$ such that 
	\begin{equation}\label{eqn::expT} \bar{\mathbb{P}}^x(\Lambda_t) \leq \frac{K\sqrt{t}}{\vp(x)}\end{equation} for all $t\ge 1$. Indeed, $\Lambda_t$ can be stochastically dominated by a geometric random variable, whose success rate is $\E{x}{L(T)>t}\ge c_{1}\vp(x)/\sqrt{t}.$ 
	
	Now, decomposing on whether or not $|N_s(t-\tau_i)|$ is bigger than $\delta s/6CK^2 $ (recall the definition of $C$ from earlier in the proof) we have 
	\begin{eqnarray*} \E{x}{\left. \I_{A_s(t-\tau_i)} |N_s(t-\tau_i)| \right| |N_s|>0}& \leq & \frac{\delta s}{6CK^2} + \E{x}{\left.\I_{B_s^{f,\eps}(\delta/6CK^2)} |N_s(t-\tau_i)| \right| |N_s|>0 }\\
		& \leq & \frac{\delta s}{6CK^2} + \E{x}{ \left. |N_s|^2 \right| |N_s|>0}^{1/2} \E{x}{\left.B_s^{f,\eps}(\delta/6CK^2)\right| |N_s|>0}^{1/2}
	\end{eqnarray*} 
	where $B_s^{f,\eps}(\cdot)$ is the event from Lemma \ref{lemma::average_lln}. Using (\ref{eqn::conditioned_seq_trees}), (\ref{eqn::csbounds}) and the above we can therefore deduce that (as long as $s\ge 1$ and $t$ is big enough that $R(t)\ge 1$)
	\begin{equation}
		\pb{x}{\I_{\{i\leq \Lambda_t\}} \I_{A_{s,t}^i}|\tilde{N}_{s,t}^i| }\leq \left(\frac{\delta}{6CK}+K^2\E{x}{B_s^{f,\eps}(\delta/6CK^2)\, \big| \, |N_s|>0}^{1/2} \right)\pb{x}{i\leq \Lambda_t}.
	\end{equation}
	
	This means that the right hand side of (\ref{eqn::breakdownprob}) is less than 
	\[ \left( \frac{\delta}{6CK}+K^2\sup_{s\geq R(t)} \left\{ \mathbb{P}^x\left( \left.B_s(\delta/6K^2)\right| \, |N_s|>0\right)^{1/2}\right\} \right)\frac{1}{t} \int_{R(t)}^{C\sqrt{t}} \sum_{i=1}^\infty \pb{x}{\I_{\{i\leq \Lambda_t\}}}.
	\]
	However, by Fubini and (\ref{eqn::expT}), this is less than 
	\[ \vp(x)^{-1} \big(\frac{\delta}{6}+ CK^3\sup_{s\geq R(t)} \left\{ \mathbb{P}^x\left( \left.B_s(\delta/6CK^2)\right| N_s>0\right)^{1/2}\right\}\big). \] Using Lemma \ref{lemma::average_lln}, and the fact that $R(t)\to \infty$, we see than this is less than $\frac{\delta}{3}$ for all $t$ large enough. Substituting in to (\ref{eqn::mainprob}) proves that $\pbu{x}{A_t}<\delta$ for such $t$, and hence completes the proof.
	
\end{proof2}

\begin{proof2}{Lemma \ref{lemma::length}}  In fact, we will prove a slightly stronger statement, as it will also be of use later on. We will show that
	\begin{equation}
		\label{eqn::biglength}
		\mathbb{P}^x(L(T)\leq cs^2 \, \big| \, |N_s|>0) \to 0 
	\end{equation}
	as $c\to 0$, uniformly in $s\ge 0$ and $x\in D$. Then by choosing $c$ such that $\mathbb{P}^x(L(T)\geq cs^2 \, \big| \, |N_s|>0)\ge 1/2$ for all $s\ge 0$ we can write 
	\begin{equation*}
		\E{x}{L(T)\ge t} \geq \E{x}{L(T)\ge t \, \big| \, |N_{\sqrt{\frac{t}{c}}}|>0}\E{x}{|N_{\sqrt{\frac{t}{c}}}|>0}\geq \frac{1}{2} \E{x}{|N_{\sqrt{\frac{t}{c}}}|>0}
	\end{equation*} for all $t$. Applying Theorem \ref{criticalsurvivalprob}, Lemma \ref{lemma::length} follows.
	
	To prove (\ref{eqn::biglength}) we will show that for any $\delta>0$ there exists a $c>0$ and $S>0$ such that $\mathbb{P}^x(L(T)\leq cs^2 \, \big| \, |N_s|>0)<\delta$ for all $x\in D$ and $s\geq S$. Clearly we can then redefine $c$ such that this holds for all $s\ge 0$ (indeed for all $s\leq S$, we have  $\mathbb{P}^x(L(T)\leq cs^2 | |N_s|>0)\leq \mathbb{P}^x(|N_S|>0)^{-1} \mathbb{P}^x(L(T) \leq cS^2)$ which tends to $0$ as $c\to 0$). 
	
	The idea of the proof is that, conditionally on $|N_s|>0$, $N_s$ will have size $O(s)$ and in fact, $O(s)$ of the particles alive at time $s$ will be well away from the boundary of $D$ (cf. Theorem \ref{thm::convergenceconditionednt}). Then, at least in expectation, each of subtrees rooted at the individuals of $N_s$ will contribute $O(s)$ to $L(T)$. Putting this together with a law of large numbers type argument tells us that $L(T)\ge cs^2$ with high probability for small enough $c$.
	
	So, we fix $\delta>0$ and also pick some $D'\Subset D$ (meaning that $\bar{D'}$ is compact and contained in $D$). We write 
	$N_s^{\subset D'}:=\{u\in N_s: X_u(s)\in D'\}$ and first claim that we can choose $\eta>0$ and $S_1> 0$ such that 
	\begin{equation}
		\label{eqn::lotsinD'}
		\E{x}{|N_s^{\subset D'}|<\eta s \, \big| \, |N_s|>0 } < \delta/2
	\end{equation}
	for all $x\in D$ and $s\ge S_1$. To see that this is possible, write 
	\begin{equation*}
		\E{x}{|N_s^{\subset D'}|< \eta s \, \big| \, |N_s|>0 }  = \E{x}{\e^{-\frac{1}{\eta s} \sum_{u\in N_s} \I_{\{X_u(s)\in D'\}}} > \e^{-1}\, \big| \, |N_s|>0} \end{equation*}	
	and note that by conditional Markov's inequality, the right hand side is less than or equal to 
	\begin{equation} \label{eqn::lotsinD'2}  \e \E{x}{\e^{-\frac{1}{\eta s} \sum_{u\in N_s} \I_{\{X_u(s)\in D'\}}} \, \big|\, |N_s|>0}.
	\end{equation}
	Since we know from Theorem \ref{thm::convergenceconditionednt} that \eqref{eqn::lotsinD'2} converges to $\e (b\int_{D'} \vp)(b\int_{D'}\vp + \eta^{-1})^{-1}$ as $s\to \infty$, uniformly in $x$ (uniformity is a consequence of the proof), we can choose $\eta>0$ small enough and $S_1>0$ large enough so that (\ref{eqn::lotsinD'}) holds for all $x\in D$ and $s\ge S_1$.
	
	Next we pick $a>0$ (very large) such that $\inf_{y\in D'} \E{y}{|N_r|}\ge 2/a\eta$ for all $r\geq 0$ (which is possible by the many-to-one Lemma \ref{lem::many_to_one}, Lemma \ref{lem::spineconvergence}, and the fact that $D' \Subset D$). For $c>0$ and $u\in \Omega$, we set $L^u_{c,s} = \int_0^{acs} |\{ w\in N_{s+r}: u\prec w \}| \, dr$, so that no matter what the value of $c$, $L(T)\geq \sum_{u\in N_s^{\subset D'}} L^u_{c,s}$. This means that we can write 
	\begin{equation}\label{eqn::final_big_length_eqn}
		\E{x}{L(T)\leq cs^2 \giv |N_s|>0} \le \frac{\delta}{2} + \frac{\E{x}{\{\sum\nolimits_{u\in N_s^{\subset D'}} L_{c,s}^u \leq cs^2 \} \cap \{|N_s^{\subset D'}|\geq \eta s\} }}{\E{x}{|N_s|>0}}
	\end{equation}
	
	To proceed, we will apply a (conditional) law of large numbers type argument (in fact in the end we will just use conditional Markov's inequality) to the random variables $\{L_{c,s}^u; u\in N_s^{\subset D'}\}$. For this, we make the following observations: 
	\begin{itemize}
	\item conditionally on $\mathcal{F}_s$, the random variables $\{L^u_{c,s}\, : \, u\in N_s^{\subset D'}\}$ are \textbf{independent};
	\item by definition of $a$ they all have \textbf{conditional mean} $\ge 2cs/\eta$; 
	\item for any $r\ge r'\ge 1$ and $y\in D$, by conditioning on $\mathcal{F}_{r'}$ we have
	\begin{equation*}
	\E{y}{|N_r||N_{r'}|}\leq \sup_{z\in D}\sup_{u>0}\E{z}{|N_u|} \times \E{y}{|N_{r'}|^2},
	\end{equation*} 
	which by Lemma \ref{lem::exp_nt2} and the fact that $\sup_{z\in D}\sup_{u>0}\E{z}{|N_u|}<\infty$ (an easy consequence of (\ref{eqn::exp_kernel_bound})) is less than or equal to $K_1 r'$ for some $K_1$ not depending on $y,r,r'$. From here, an easy integration gives that each of the $\{L^u_{c,s}\, : \, u\in N_s^{\subset D'}\}$, has  \textbf{conditional variance} given $\mathcal{F}_s$ less than $ K c^3 s^3 a^3$ for some $K<\infty$ not depending on $c$ or $s$. 
	\end{itemize}
	
	Putting all of this together we see that 
	\begin{equation}
		\label{eqn::cond_lln}
		\E{x}{\sum\nolimits_{u\in N_s^{\subset D'}} L^u_{c,s}\leq cs^2 \, \big| \, \mathcal{F}_s}
	\end{equation} is the probability that a sum of $N_s^{\subset D'}$ independent positive, random variables (with means and variances as above) is less than or equal to $cs^2$. Let us suppose that $N_s^{\subset D'}=n$ and $n\geq  \eta s$. Then since the conditional mean of $\sum_{u\in N_s^{\subset D'}} L^u_{c,s}$ is bigger than or equal to $2csn/\eta$, we see that the probability in (\ref{eqn::cond_lln}) is less than or equal to the conditional probability that  ``the absolute value of ($\sum_{u\in N_s^{\subset D'}} L^u_{c,s}$ minus its conditional mean) is greater than or equal to $2csn/\eta - cs^2$". By conditional Markov's inequality, using conditional independence and the upper bound on the conditional variances, we see that this is less than or equal to 
	$$ \frac{Kn c^3 s^3a^3}{(2csn/\eta - cs^2)^2} 
	\leq \frac{cK\eta a^3}{2-\frac{\eta s}{n}} \leq cK\eta a^3 $$
	where for both inequalities we have used the fact that $n\geq \eta s$. Substituting this in, we see that the second term on the right hand side of (\ref{eqn::final_big_length_eqn}) is less than or equal to $\frac{\delta}{2}$ for all $c$ small enough. This concludes the proof. 
\end{proof2}

\medskip

\begin{proof2}{Lemma \ref{lemma::ubigheight}}
	Pick $t_0>0$. As explained in the proof of Proposition \ref{propn::qvconv}, see \eqref{eqn::expT}, Lemma \ref{lemma::length} immediately implies that
	$\pb{x}{\Lambda_t}\le \frac{K\sqrt{t}}{\vp(x)}$ for some $K=K(t_0)$ and all $x\in D$, $t\geq t_0$.  By Theorem \ref{criticalsurvivalprob} we can also choose this $K$ such that $\E{x}{|N_r|>0} \leq K\vp(x)/r$ for all $x\in D$ and $r\geq \sqrt{t_0}$. 
	To proceed, we observe that $\pbu{x}{H_t^*\geq C\sqrt{t}}$ is less than the probability that \emph{some} tree visited before time $t$ in the depth first exploration has height exceeding $C\sqrt{t}$. Thus, by a union bound we can write
	\[  \pbu{x}{H_t^*\geq C\sqrt{t}}\leq \pb{x}{\sum_{i=1}^\infty \I_{\{i\leq \Lambda_t\}} \I_{\{|N_{C\sqrt{t}}^i|>0\}}}. \]
	We will show that we can apply Fubini to the right hand side of the above, and that the expression we get converges to $0$ as $C\to \infty$. To do this, we write for each $i\geq 1$, since the event $\{i\leq \Lambda_t\}$ does not depend on the $i$th tree, $$\pb{x}{ \{i\leq \Lambda_t\} \cap \{|N_{C\sqrt t}^i|>0\}}\le \pb{x}{i\leq \Lambda_t}\E{x}{|N_{C\sqrt{t}}|>0}.$$
	Then as long as $t\geq t_0$ and $C>1$, our definition of $K$ means that the right hand side of the above is less than or equal to $K \pb{x}{i\leq \Lambda_t} \frac{\vp(x)}{C\sqrt{t}}$. Therefore by Fubini, since $\sum_{i=1}^\infty \pb{x}{i\leq \Lambda_t}=\pb{x}{\Lambda_t}\leq K\sqrt{t}/\vp(x)$, we have 
	\begin{equation*} 
		\pbu{x}{H_t^*\geq C\sqrt{t}}\leq \frac{K^2}{C}
	\end{equation*}
	for all $x\in D$, $t\geq t_0$ and $C\geq 1$. This concludes the proof.
\end{proof2}	
 
 \medskip 
 
\begin{proof2}{Lemma \ref{lemma::rtoinfinity}} This follows from the law of large numbers. Indeed, observe that for any fixed $K>0$, the amount of time in $[0,t]$ that the depth first exploration spends below height $K$ in the $i$th tree is equal to $\int_0^K |\tilde{N}_{r,t}^i| \, dr$ for every $i\leq \Lambda_t$. Thus by conditioning on the entire sequence of trees, we have
	\begin{equation*}
	\pbu{x}{H_t^*\leq K} = \pb{x}{\frac{1}{t} \sum_{i=1}^{\Lambda_t} \int_0^K |\tilde{N}_{r,t}^i| \, dr}.
	\end{equation*}
	 Moreover, since $t\geq \sum_{i\leq \Lambda_t} L(T^i)$, by definition of $\Lambda_t$, we can bound
	\begin{equation}\label{eqn::lln_height_df}
	\frac{1}{t} \sum_{i=1}^{\Lambda_t} \int_0^K |\tilde{N}_{r,t}^i| \, dr \le \frac{1}{t} \sum_{i=1}^{\Lambda_t} \int_0^K |N_r^i| \, dr \leq \frac{\sum_{i=1}^{\Lambda_t} \int_0^K |N_r^i| \, dr}{\sum_{i=1}^{\Lambda_t-1} L(T^i)}.
	\end{equation}
	However, for each $i$ the random variables $\int_0^K |N_r^i| \, dr$ are i.i.d with finite variance (by Fubini), and the random variables $L(T^i)$ are i.i.d with infinite mean (by Lemma \ref{lemma::length}). Since $\Lambda_t\to \infty$ a.s. as $t\to \infty$ (each tree is finite almost surely), the strong law of large numbers allows us to conclude that the right (and therefore left) hand side of (\ref{eqn::lln_height_df}) converges to $0$ almost surely under $\bar{\mathbb{P}}^x$. By dominated convergence (the random variable on the left hand side of (\ref{eqn::lln_height_df}) is a.s. less than or equal to $1$) we therefore have $\pbu{x}{H_t^*\leq K}\to 0$, and hence $R(t)\to \infty$ \end{proof2}

\subsection{Connection with the height function}
\label{sec::connection}

In order to make use of the above invariance principle, we will now make the connection between the height function $(H_t)_{t\ge 0}$, and the process $(\mathbf{S}_t)_{t\ge 0}$ from Definition \ref{rmk::sonetree}, corresponding to the depth first exploration of a \emph{single tree $T$ conditioned to be large.} For a vertex $u\in T$ we use the notation $$\mathbf{S}(u):= \sum_{w\in Y(u,T)} \vp(X_w(s_w-l_w))$$ so that $\mathbf{S}_t=\mathbf{S}(v_t)$.
We will show that for large $t$, and for an overwhelming proportion of the individuals $u\in N_t$ (given survival), $\mathbf{S}(u)$ is close to a constant times $t$.  Our approach will use an ergodicity property for the spine particle in the system under $\tilde{\mathbb{Q}}^x$, and is inspired from \cite{spinelln}.

In the following, given $\eta>0$, we will say that a pair $(u,t)$ with $u\in N_t$ is $\eta$-\emph{bad} if 
\begin{equation}\label{eqn::etabad}
	\left|\frac{\mathbf{S}(u)}{t}-b^{-1}\right|>\eta
\end{equation} 
(recalling the definition of $b$ from (\ref{eqn::defnb})). We also say, for given $R\geq 0$ and $t\geq R$, that $(u,t)$ is \emph{$\eta_R$-bad} if some $(v,s)$ with $R\leq s \leq t$, $v\in N_s$ and $v\prec u$ is $\eta$-bad. That is, $(u,t)$ is $\eta_R$-good if all of the ancestors of $u$ after time $R$ are $\eta$-good. We have the following estimate for the proportion of $u\in N_t$ such that $(u,t)$ is $\eta_R$-bad:

\begin{propn}\label{propn::etabad} Fix $\eps,\eta>0$ and write $N_t^{\eta_R}:=
	\{u\in N_t: (u,t) \text{ is } \eta_R-\text{bad }\}$. Then we have
	\begin{equation}
		\label{eqn::etafbad}
		\sup_{t \ge R} \mathbb{P}^x \left( \left.\frac{|N_t^{\eta_R}|}{|N_t|}>\eps \right| |N_t|>0 \right) \to 0
	\end{equation}
	as $R\to \infty$, for any $x\in D$.
\end{propn} 

\begin{proof}
	We will first show that for any $\eps>0$, setting $E_{R,t}^\eps:= \big\{\frac{\sum_{u\in N_t} \vp(X_u(t))\I_{\{(u,t) \text{ is } \eta_R-\text{bad}\}}}{\sum_{u\in N_t} \vp(X_u(t))}>\eps\big\}$, we have
	\begin{equation} \label{eqn::weighted_average_ntbad} \sup_{t\geq R} \mathbb{P}^x\left(E_{R,t}^\eps\, \big| \, |N_t|>0\right)\to 0\end{equation} 
	as $R\to \infty$. The reason for this approach is that it technically more convenient to work with $E_{R,t}^\eps$ than $\{|N_t^{\eta_R}|/|N_t|>\eps\}$ and, as we will see in the end, the events are essentially the same. 
	
	To show \eqref{eqn::weighted_average_ntbad}, we will use the description of the system under the measure $\tilde{\mathbb{Q}}^x$ given in Section \ref{sec::com}. Recalling that $\mathbb{Q}^x=\tilde{\mathbb{Q}}^x\big|_{\mathcal{F}_\infty}$ has 
	$ (d\mathbb{Q}^x/d\mathbb{P}^x)|_{\mathcal{F}_t}=\vp(x)^{-1}\sum_{u\in N_t} \vp(X_u(t))$ we see that
	\[ \mathbb{P}^x(E_{R,t}^\eps \giv |N_t|>0 ) = \mathbb{Q}^x \left[\frac{\vp(x)/\mathbb{P}^x(|N_t|>0)}{\sum_{u\in N_t} \vp(X_u(t))} \I_{E_{R,t}^\eps} \right]= \mathbb{Q}^x \left[ Y_t \,\I_{E_{R,t}^\eps}\right],\]
	where $Y_t:=\frac{\vp(x)/\mathbb{P}^x(|N_t|>0)}{\sum_{u\in N_t} \vp(X_u(t))}$.
	To see that this converges to $0$ it is enough to prove that: 
	\begin{enumerate}[(i)]
		\item $\sup_{t \ge R}\mathbb{Q}_x(E_{R,t}^\eps)\to 0$ as $R\to \infty$; and 
		\item For every $\delta>0$, there exists $R'$ and $K$ positive, such that 
		$\mathbb{Q}^x\left( Y_t \I_{|Y_t|>K} \right) \leq \delta$ for all $t\geq R'$.
	\end{enumerate}
	Point (i) comes, roughly speaking, from the fact that if $\psi_t$ is the \emph{spine particle} under $\tilde{\mathbb{Q}}^x$ at time $t$, and $t\geq R$, then $(\psi_t, t)$ is very unlikely to be $\eta_R$-bad for large $R$. More precisely, by (\ref{eqn::spine_prob}), we have that 
	\begin{equation*}
		\frac{\sum_{u\in N_t} \vp(X_u(t))\I_{\{(u,t) \text{ is } \eta_R-\text{bad}\}}}{\sum_{u\in N_t} \vp(X_u(t))} = \tilde{\mathbb{Q}}^x((\psi_t,t) \text{ is } \eta_R-\text{bad} \giv \mathcal{F}_t)
	\end{equation*}
	and so by Markov's inequality 
	\begin{equation*}
		\mathbb{Q}^x(E_{R,t}^\eps)= \tilde{\mathbb{Q}}^x(E_{R,t}^\eps) \leq \eps^{-1} \tilde{\mathbb{Q}}^x(\frac{\sum_{u\in N_t} \vp(X_u(t))\I_{\{(u,t) \text{ is } \eta_R-\text{bad}\}}}{\sum_{u\in N_t} \vp(X_u(t))}) = \eps^{-1} \tilde{\mathbb{Q}}^x((\psi_t,t) \text{ is } \eta_R-\text{bad}).
	\end{equation*}
	Moreover, we claim (see explanation below) that \begin{equation} \label{eqn::spine_not_bad} \sup_{t\geq R}\tilde{\mathbb{Q}}^x((\psi_t,t) \text{ is } \eta_R-\text{bad})\to 0\end{equation} as $R\to \infty$. To see this, recall that under $\tilde{\mathbb{Q}}^x$ we know (from Section \ref{sec::com}) that the spine particle evolves as a diffusion conditioned to remain in $D$, and that at constant rate $\frac{m}{m-1}\lambda $ a random number of ``younger siblings'' (see just before Definition \ref{def::st}) is born along its trajectory (each with initial position equal to the position of the spine at the time that they are born). The random number of younger siblings born at each time is always independent from everything else and has finite mean $(2m)^{-1}(\mathbb{E}(A^2)-\mathbb{E}(A))$ (coming from the description in Section 
	\ref{sec::com} again). On the other hand, we have that for any $r\ge 0$, $\mathbf{S}(\psi_r)=\sum_{y\in Y} \vp(y)$, where $Y$ is the set of initial positions (counted with multiplicity) of younger siblings born along the trajectory of the spine before time $r$. Since by Lemma \ref{lem::spineconvergence} and \cite{pin}, the trajectory of the spine is a Markov process with invariant density $\vp^2$, (\ref{eqn::spine_not_bad}) follows by a straightforward ergodicity argument. This proves (i).
	
Point (ii) essentially says that $(Y_t)_{t\geq 0}$ is $\mathbb{Q}^x$- uniformly integrable (in fact it is, but we only need the weaker statement). To prove it, one can use the change of measure between $\mathbb{Q}^x$ and $\mathbb{P}^x$ again to write 
	\[ \mathbb{Q}^x ( Y_t \I_{\{|Y_t|>K\}} ) = \frac{\mathbb{P}^x\left( \{|Y_t|>K\}\cap\{|N_t|>0\}\right)}{\mathbb{P}^x(|N_t|>0)}=\mathbb{P}^x\left(|Y_t|>K \giv |N_t|>0\right).\]
	Since $\vp(x)/(t\mathbb{P}_x(N_t>0))$ is uniformly bounded above for $t\geq 1$ (say) by Theorem \ref{criticalsurvivalprob}, we just need to show that for any $\delta>0$ there exists $K$ and $R'$ such that  
	\begin{equation}\label{eqn::unifexample} \sup_{t\geq R'}\,\mathbb{P}^x\left( \frac{\sum_{u\in N_t}\vp(X_u(t))}{t}< 1/K\giv|N_t|>0 \right) \leq \delta . \end{equation}
	However, this is a direct consequence of the convergence given by Theorem \ref{thm::conditionedsystem}, since we know that for fixed $K$ the probability in \eqref{eqn::unifexample} converges, as $t\to \infty$, to the probability that an exponential random variable is less than $1/K$. This completes the proof of (\ref{eqn::weighted_average_ntbad}). 
	
	We must now deduce from (\ref{eqn::weighted_average_ntbad}) that \[ \sup_{t \ge R} \mathbb{P}^x \left( \left.\frac{|N_t^{\eta_R}|}{|N_t|}>\eps \right| |N_t|>0 \right) \to 0 \text{ as } R\to \infty.\] The idea behind this is that $\sum_{u\in N_t} \vp(X_u(t))\I_{\{(u,t) \text{ is } \eta_R-\text{bad}\}}/\sum_{u\in N_t} \vp(X_u(t))$ is a reasonable approximation to $|N_t^{\eta_R}|/|N_t|$ at large times. Roughly speaking they should at least be proportional, since $\vp$ is bounded above, and on any subdomain (where most of the particles will be at large times, given survival) it is also bounded below. 
	
	More precisely, by Corollary \ref{uniformparticle} we know that given $\delta>0$, if $D_r=\{y\in D: \vp(y)<1/r\}$ and $N_t^{D_r}:=\{u\in N_t: X_u(t)\in D_r\}$, there exists $r>0$ such that
	\begin{equation}\label{eqn::unif3}
		\sup_{t\geq R} \, \mathbb{P}^x \big(  \frac{|N_t^{D_r}|}{|N_t|}>\frac{\eps}{2} \giv |N_t|>0 \big) \leq \delta/2,
	\end{equation} 
	for all $R$ large enough. Moreover, we can write $|N_t^{\eta_R}|\leq |N_t^{D_r}|+|N_t^{\eta_R}\setminus (N_t^{\eta_t}\cap N_t^{D_r})|$, where we have 
	$$\frac{|N_t^{\eta_R}\setminus (N_t^{\eta_t}\cap N_t^{D_r})|}{|N_t|}\leq \frac{\sum_{u\in N_t} \vp(X_u(t))\I_{\{(u,t) \text{ is } \eta_R-\text{bad}\}}}{\sum_{u\in N_t} \vp(X_u(t))} \|\vp\|_\infty r.$$ Finally, by (\ref{eqn::weighted_average_ntbad}) we can bound
	$$ \sup_{t\ge R}\mathbb{P}^x\big( \frac{|N_t^{\eta_R}\setminus (N_t^{\eta_t}\cap N_t^{D_r})|}{|N_t|} > \eps/2 \big) < \delta/2$$
	for all $R$ large enough. Putting this together with (\ref{eqn::unif3}) allows us to conclude.
\end{proof}

\begin{remark0}\label{rmk::rangeoft}
	Now suppose we are given $c>0$ and $x\in D$. Then for any $t$ such that $ct\wedge t\geq R$ 
	\begin{align*} \mathbb{P}^x \big( \{\frac{|N_{ct}^{\eta_R}|}{|N_{ct}|}>\eps\}\, \cap \, \{|N_{ct}|>0\}\giv |N_t|>0 \big) & =  \frac{\mathbb{P}^x \big( \{\frac{|N_{ct}^{\eta_R}|}{|N_{ct}|}>\eps\}\, \cap \, \{|N_{ct}|>0\}\, \cap \, \{ |N_t |>0 \} \big)}{\mathbb{P}^x(|N_{t}|>0)} \\
		& \leq  \mathbb{P}^x \big( \{\frac{|N_{ct}^{\eta_R}|}{|N_{ct}|}>\eps\}\, \cap \, \{|N_{t}|>0\}\giv |N_{ct}|>0 \big) \frac{\E{x}{|N_{ct}|>0}}{\E{x}{|N_t|>0}} \\
		& \leq \left(\sup_{s\geq R} \mathbb{P}^x \big( \frac{|N_s^{\eta_R}|}{|N_s|}>\eps \giv |N_s|>0 \big)\right)\frac{\E{x}{|N_{ct}|>0}}{\E{x}{|N_t|>0}} \\
		& \leq \frac{F(R)}{c} \left(\sup_{s\geq R} \mathbb{P}^x \big( \frac{|N_s^{\eta_R}|}{|N_s|}>\eps \giv |N_s|>0 \big)\right)
	\end{align*} where
	$$F(R)=\frac{\sup_{s\geq R} sb^{-1}\vp(x)^{-1}\E{x}{|N_s|>0} }{\inf_{s\ge R} sb^{-1}\vp(x)^{-1}\E{x}{|N_s|>0} }\longrightarrow 1$$ as $R\to \infty$, by Theorem \ref{criticalsurvivalprob}. Note that we are allowing $c<1$ here.
\end{remark0}

The next lemma provides the key connection between $(\mathbf{S}_t)_{t\ge 0}$ and the height function $(H_t)_{t\ge 0}$ for a critical branching diffusion under $\mathbb{P}^x$, that is conditioned to survive for a long time. We write $\mathcal{P}^x_1$ for the law of a tree $(T,X,l)$ under $\mathbb{P}^x$ plus a random variable $t^1$, which conditionally on $(T,X,l)$ is chosen uniformly from $[0,L(T)):=[0,L)$.
\begin{lemma} \label{lemma::goodcondtree}
	For any $\eta>0$ we have 
	\begin{equation*} 
		\lim_{R\to \infty}\lim_{t \to \infty} \mathcal{P}_1^x\big( (v_{t^1},H_{t^1}) \text{ is } \eta_R-\text{bad} \giv |N_t|>0 \big) \to 0.
	\end{equation*}
\end{lemma}
\begin{remark0}In the case that $H_{t^1}< R$ we also say that $(v_{t^1},H_{t^1})$ is $\eta_R$-bad (with an abuse of notation; recall we only defined the notion of $\eta_R$-badness for $(u,t)$ with $t\geq R$). However, since the probability of this event goes to $0$ as $t\to \infty$ for any fixed $R$, this will not play a role; we only introduce the convention for notational convenience. \end{remark0}

\begin{proof} By decomposing the probability space into the three disjoint events $\{L\le at^2\}$, $\{L>at^2\}\cap \{H_{t^1}\notin [ct,Ct]\}$ and $\{L>at^2\}\cap \{H_{t^1}\in [ct,Ct]\}$ we can write for any $a,c,C>0$
	\begin{align}\label{eqn::uparticleRbad}	\mathcal{P}_1^x\big( & (v_{t^1},H_{t^1}) \text{ is } \eta_R-\text{bad} \giv |N_t|>0 \big) \leq \mathcal{P}^x_1\big( L\leq a t^2 \giv |N_t|>0\big) + \mathcal{P}^x_1\big( \{H_{t^1} \notin [ct,Ct]\}\cap\{L>at^2\}\giv |N_t|>0\big) \nonumber \\
		& + \mathcal{P}^x_1\big( \{(v_{t^1},H_{t^1}) \text{ is } \eta_R-\text{bad}\}\cap \{L>at^2\} \cap \{H_{t^1}\in[ct,Ct]\}\giv |N_t|>0\big).
	\end{align} We deal with each term separately. To begin, by putting together (\ref{eqn::biglength}) and the fact that
	$\E{x}{|N_{Ct}|>0 \giv |N_t|>0} \to 0$ we see that the first term on the right hand side of \eqref{eqn::uparticleRbad} converges to $0$, uniformly in $t\ge 1$ (say), as $a\to 0$ and then $c\to 0,C\to \infty$. For the second, by definition of the conditional probability we have 
	\begin{equation*}
		\mathcal{P}^x_1\big(\{H_t^1\leq ct\}\cap\{L>at^2\} \giv |N_t|>0 \big) \le (\E{x}{|N_t|>0})^{-1} \, \E{x}{\{L\geq at^2\}\cap\{H_{t^1}\leq ct\} } 
		\end{equation*} which, by conditioning on $\mathcal{F}_\infty = \sigma((T,l,X))$, is equal to 
	\begin{equation*} 
	\E{x}{\I_{\{L\geq at^2\}} \mathcal{P}^x_1(H_{t^1}\leq ct \giv \mathcal{F}_\infty)}. 
	\end{equation*}
	Using the lower bound on $L$ from the indicator function, the lower bound on $\mathbb{P}^x(|N_t|>0)$ that one obtains from Theorem \ref{criticalsurvivalprob}, and the fact that $\sup_{z\in D} \sup_{u>0} \E{z}{|N_u|}<\infty$ (which we saw in the proof of Lemma \ref{lemma::length}) we obtain that the above is less than or equal to 
		\begin{equation*} (\E{x}{|N_t|>0})^{-1} \,  \E{x}{\I_{\{L\geq at^2\}} \frac{\int_0^{ct}|N_s| \, ds}{L}}  \leq \frac{cK}{a\vp(x)}
	\end{equation*}
	for some finite $K$ and all $t\geq 1$. This means that the second term on the right hand side of \eqref{eqn::uparticleRbad} also converges to $0$ uniformly in $t\ge 1$, as $a\to 0$ and then $c\to 0,C\to \infty.$
	
	Thus, we are left to prove that for any fixed $a,c,C$
	\begin{equation}\label{eqn::u_vertex_bad_finaleq}
		\lim_{R\to \infty}\lim_{t\to \infty}\mathcal{P}^x_1\big( \{(v_{t^1},H_{t^1}) \text{ is } \eta_R-\text{bad}\}\cap \{L>at^2\} \cap \{H_{t^1}\in[ct,Ct]\}\giv |N_t|>0\big)=0.
	\end{equation}
	From now on we assume that $ct\geq R$ (which is without loss of generality, since we are letting $t\to \infty$ first). By conditioning on $\mathcal{F}_\infty$ again, we see that 
	\begin{align}\label{eqn::nrbadeq} & \mathcal{P}^x_1\big( \{(v_{t^1},H_{t^1}) \text{ is } \eta_R-\text{bad}\}\cap \{L>at^2\} \cap \{H_{t^1}\in[ct,Ct]\}\giv |N_t|>0\big) \nonumber \\ & = \E{x}{\I_{\{L> at^2\}}\int_{ct}^{Ct} \frac{|N_u^{\eta_R}|}{|N_u|}\frac{|N_u|}{L}\, du \, \giv |N_t|>0} \nonumber \\
		& \leq \delta + (at^2)^{-1}\E{x}{\int_{ct}^{Ct} \I_{\{\frac{|N_u^{\eta_R}|}{|N_u|}>\delta\}}|N_u| \, du \giv |N_t|>0}
	\end{align}
	for any $\delta>0$. Therefore, we just need to prove that the second term in the last expression of \eqref{eqn::nrbadeq}, for fixed $\delta$, converges to 0 as $t\to \infty$ and then $R\to \infty$. However, we can write this term (by Fubini) as 
	\begin{equation*}
		(at^2)^{-1} \int_{ct}^{Ct} \E{x}{|N_u|\I_{\{\frac{|N_u^{\eta_R}|}{|N_u|}>\delta\}} \giv |N_t|>0} \, du \end{equation*} which by conditional Cauchy--Schwarz, Remark \ref{rmk::rangeoft}, and the assumption that $ct\geq R$ is less than 
		\begin{equation*}\frac{\sqrt{F(R)}}{\sqrt{c}}\sup_{s\geq R} \E{x}{\frac{|N_s^{\eta_R}|}{|N_s|}>\delta \giv |N_s|>0}^{1/2} \times (at^2)^{-1}\int_{ct}^{Ct} \E{x}{|N_u|^2||N_t|>0}^{1/2} \, du.
	\end{equation*}  Finally, we observe that by (\ref{eqn::csbounds}) and Theorem \ref{criticalsurvivalprob}, we have $\E{x}{|N_u|^2||N_t|>0}^{1/2}\leq Mu$ for some constant $M=M(c,C')$ and for any $u\in [ct,Ct]$, $ct\geq R\geq 1$. Hence, by integrating and applying Proposition \ref{propn::etabad}, we obtain (\ref{eqn::u_vertex_bad_finaleq}).
\end{proof}

Now we extend this result. Let $\mathcal{P}^x_k$ be the law of a tree $(T,X,l)$ under $\mathbb{P}^x$, together with $k$ random variables $(t^1,\cdots, t^k)$ chosen conditionally independently and uniformly at random from $[0,L(T))=[0,L)$. We also define the $k\times k$ matrices 
\begin{align}\label{eqn::matrices} (D_t^S)_{ij} & := t^{-1}\big( \mathbf{S}(v_{t^i})+\mathbf{S}(v_{t^j})-2\mathbf{S}(v^{ij}) \big) \text{ and } \\ \nonumber
	(D_t^H)_{ij} & := t^{-1} \big(H_{t^i}+H_{t^j}-2h^{ij}\big)
\end{align}
where $v^{ij}=v_{t^i}\wedge v_{t^j}$ is the most recent common ancestor of $v_{t^i}$ and $v_{t^j}$, and $h^{ij}=s_{v^{ij}}$ is its ``death time" (see Section \ref{sec::bds} for definitions of these objects.) The next proposition says that conditioned on survival up to a large time $t$, these matrices are essentially the same up to a constant.

\begin{propn}\label{prop::distancematrices}
	Let $k\geq 1$ and $D_t^S$ and $D_t^H$ be as defined above. Then for any $\eps>0$ 
	\begin{equation}\label{eqn::distsecondconv}
		\mathcal{P}^x_k\left(\left. \|b^{-1}D_t^H-D_t^S\|>\eps \right| |N_t|>0 \right) \to 0
	\end{equation}
	as $t\to \infty$, where the distance is the Euclidean distance between $k\times k$ matrices.
\end{propn} 

\begin{proof} 
	We prove this in the case $k=2$: the general result following by a union bound. Note that by symmetry, and since $(D_t^H)_{ii}=(D_t^S)_{ii}=0$ for $i=1,2$, we need only control the distance $|b^{-1}(D_t^H)_{12}-(D_t^S)_{12}|$. Given $\delta>0$, we:
	\begin{itemize}
		\item choose $M\geq 1$ such that $\lim_{t \to \infty} \E{x}{|N_{Mt}|>0\giv |N_t|>0}\leq \delta/4$, which is possible by Theorem \ref{criticalsurvivalprob};
		\item set $\eta:=\eps/(4M)$ and pick $R$ large enough that $\lim_{t \to \infty}\mathcal{P}^x_2( \{(v_{t^1},H_{t^1}) \text{ is } \eta_R-\text{bad}\}\cup \{(v_{t^2},H_{t^2}) \text{ is } \eta_R-\text{bad}\} \giv |N_t|>0)\leq \delta/4$, which is possible by a union bound and Lemma \ref{lemma::goodcondtree};
		\item pick $K$ large enough that $\lim_{t \to \infty} \E{x}{\int_{0}^R |N_s|\geq K \giv |N_t|>0} \leq \delta/4$, which is possible by Theorem \ref{thm::conditionedsystem} and Markov's inequality.
	\end{itemize} 
	
	Putting this together, we see that for all $t$ large enough, if we set
	\begin{equation*} A_t:= \{|N_{Mt}|>0\} \cup \{\int\nolimits_{[0,R]} |N_s| \geq K\} \cup\{(v_{t^1},H_{t^1}) \text{ is } \eta_R-\text{bad}\}\cup \{(v_{t^2},H_{t^2}) \text{ is } \eta_R-\text{bad}\} \end{equation*} then $	\mathcal{P}^x_2(A_t\giv |N_t|>0)  \leq \delta $. This implies that $\mathcal{P}^x_2(A_t\giv |N_t|>0) \to 0$ as $t\to \infty$.
	
	Now consider the complementary event $A_t^c$. Straight from the definitions, we know that on this event we have
	\begin{equation}\label{eqn::bounds_1_matrices}
		H_{t^1}\leq Mt\, , \;\;
		H_{t^2}\leq Mt \, , \;\; 
		|\frac{\mathbf{S}(v_{t^1})}{H_{t^1}}-b^{-1}|\leq \eta\, , \; \text{ and }\; |\frac{\mathbf{S}(v_{t^2})}{H_{t^2}}-b^{-1}|\leq \eta
	\end{equation}
	where $\eta=\eps/(4M)$. Then there are two possibilities: we are either on the event $B:=\{h^{ij}>R\}$ or the event $B^c:=\{h^{ij}\leq R\}$. We will show that for large enough $t$, $B\cap A_t^c=B^c\cap A_t^c=\emptyset$, thus completing the proof.
	
	To do this, observe that on $\{B\cap A_t^c\}$, by definition of $A_t$ and ``$\eta_R$-badness", we have that
	\begin{equation*}
		h^{12}\leq Mt \; \text{ and } |\frac{\mathbf{S}(v^{12})}{h^{12}}-b^{-1}|\leq \eta.
	\end{equation*}
	Putting this together with (\ref{eqn::bounds_1_matrices}) gives the deterministic bound
	$$|(D_t^S)_{12}-b^{-1}(D_t^H)_{12}| \leq \frac{1}{t} \times 4Mt \times \eta \leq \eps \quad \text{on} \quad \{B\cap A_t^c\}$$ and so we have $B\cap A_t^c = \emptyset$.  Furthermore, on the event $\{B^c\cap A_t^c\}$, we have $h^{ij}\leq R$ and $\mathbf{S}(v^{ij})\leq \|\vp\|_\infty K$ (by a very crude bound). This means that 
	$$ |\mathbf{S}(v^{12})-b^{-1}h^{12}| \leq K + Rb^{-1} $$ and so 
	$$ |(D_t^S)_{12}-b^{-1}(D_t^H)_{12}| \leq \frac{\eps}{2}+ \frac{K + Rb^{-1}}{t}. $$  Since the second term on the right hand side above is less than $\eps/2$ for all $t$ large enough, we see that $B^c\cap A_t^c=\emptyset$ for all such $t$. 
\end{proof}

\subsection{Convergence to the CRT}

\subsubsection{Preliminaries on convergence of metric measure spaces}

Before we can prove Theorem \ref{thm::crtconv}, we need to introduce various notions of convergence for metric spaces, and more generally, for metric measure spaces. Although our aim is to prove convergence of conditioned genealogical trees in the sense of Gromov--Hausdorff distance between metric spaces, it turns out to be helpful to go through the framework of metric measure spaces. We first recall the definition of the Gromov--Hausdorff metric on $\mathbb{X}_c$: the space of (isometry classes of) compact metric spaces. 

\begin{defn}\label{defn::ghmetric}
	The \emph{Gromov--Hausdorff} distance between $(X,r_X)$ and $(Y,r_Y)$ in $\mathbb{X}_c$ is given by 
	\[ d_{GH}\left((X,r_X),(Y,r_Y)\right)=\inf_{g_X,g_Y,Z} d_H^{(Z,r_Z)}(g_X(X),g_Y(Y)),\]
	where the infinum is taken over all isometric embeddings $g_X$ and $g_Y$ from $X$ and $Y$ to a common metric space $(Z,r_Z)$, and $d_H^{(Z,r_Z)}$ is the usual Hausdorff distance on $(Z,r_Z)$. 
\end{defn}

For us, a \emph{metric measure space} $(X,r,\mu)$ will be a compact metric space $(X,r)$ equipped with a finite Borel measure $\mu$. These will be considered modulo the equivalence $\sim$, where $(X,r,\mu)\sim (X',r',\mu')$ if there exists a measure preserving isometry between $X$ and $X'$. We denote the set of (equivalence classes) of these spaces by $\mathbb{X}$. We will be interested in the \emph{Gromov--Prohorov} metric and the \emph{Gromov--Hausdorff--Prohorov} metric on $\mathbb{X}$. We begin by defining the so-called \emph{Gromov--weak} topology.

\begin{defn}{\cite[Definition 2.3]{mmspaces}}
	\label{defn::poly}
	We will call a function $\Phi:\mathbb{X}\to \mathbb{R}$ a \emph{polynomial} if there exists an $k\in \mathbb{N}$ and a bounded continuous function $\phi:[0,\infty)^{k\choose2}\to \R$ such that 
	\[\Phi((X,r,\mu))=\int \mu^{\otimes k}(d(x_1,\cdots, x_n))\phi((r(x_i,x_j))_{1\leq i<j\leq k}),\]
	where $\mu^{\otimes k}$ is the product measure of $\mu$. We write $\Pi$ for the set of all polynomials.
\end{defn}
\begin{defn}{\cite[Definition 2.8]{mmspaces}}\label{defn::gwtop}
	A sequence $\mathcal{X}_n\in \mathbb{X}$ is said to converge to $\mathcal{X}\in \mathbb{X}$ with respect to the \emph{Gromov--weak topology} if and only if $\Phi(\mathcal{X}_n)$ converges to $\Phi(\mathcal{X})$ in $\mathbb{R}$, for all polynomials $\Phi \in \Pi$. 
\end{defn}
It is known, see \cite[Theorem 5]{mmspaces}, that this topology is metrised by the \emph{Gromov--Prohorov metric}, that we now define. In the following, if $f:X_1\to X_2$ is an embedding and $\mu$ is a measure on $X_1$, we write $f_*\mu$ for the image measure of $\mu$ under $f$ on $X_2$: defined by $f_*\mu(A)=\mu(f^{-1}(A))$ for $A\subset X_2$. 
\begin{defn} \label{defn::gpmetric} The \emph{Gromov--Prohorov} distance between $\mathcal{X}=(X,r_X,\mu_X)$ and $\mathcal{Y}=(Y,r_Y,\mu_Y)$ in $\mathbb{X}$ is given by 
	\[ d_{GP}(\mathcal{X},\mathcal{Y})=\inf_{g_X,g_Y,Z} d_{Pr}^{(Z,r_Z)}((g_X)_*(\mu_X),(g_Y)_*(\mu_Y)),\]
	where the infinum is as in Definition \ref{defn::ghmetric} and $d_P^{(Z,r_Z)}$ is the Prohorov distance between probability measures on $(Z,r_Z)$. 
\end{defn} 
Finally, we define the Gromov--Hausdorff--Prohorov metric \cite{ghptop, mieghp} on $\mathbb{X}$.
\begin{defn}\label{defn::ghpmetric} Let $\mathcal{X},\mathcal{Y}$ be as in Definition \ref{defn::gpmetric}.  The \emph{Gromov--Hausdorff--Prohorov} distance between $\mathcal{X}$ and $\mathcal{Y}$ is defined by 
	\[ d_{GHP}(\mathcal{X},\mathcal{Y})=\inf_{g_X,g_Y,Z} \left(d_{Pr}^{(Z,r_Z)}((g_X)_*(\mu_X),(g_Y)_*(\mu_Y)) + d_H^{(Z,r_Z)}(g_X(X),g_Y(Y))\right). \]
\end{defn}
\begin{remark0}\label{rmk::ghpimplies}
	It is clear from the above definitions that convergence in the Gromov--Hausdorff--Prohorov metric implies convergence in both the Gromov--Hausdorff metric and the Gromov--Prohorov metric.
\end{remark0}

We will need a couple of facts for our proof:
\begin{lemma}{\cite[Corollary 3.1]{mmspaces}}\label{lemma::gwconv}
	A sequence $\{\mathbb{P}_n\}_{n\in \N}$ of probability measures on $\mathbb{X}$ converges weakly to a probability measure $\mathbb{P}$ with respect to the Gromov--weak topology, if and only if 
	\begin{enumerate}[(i)]
		\item The family $\{\mathbb{P}_n\}_{n\in \N}$ is relatively compact in the space of probability measures on $\mathbb{X}$. 
		\item For all polynomials $\Phi\in \Pi$, $\mathbb{P}_n[\Phi]\to \mathbb{P}[\Phi]$ in $\R$ as $n\to \infty$.
	\end{enumerate}
\end{lemma}

\begin{lemma}{\cite[Theorem 2.4]{ghptop}, \cite[Theorem 7.4.15]{hauscomp}}\label{lemma::ghptight}
	A set $\mathbb{K}\subset \mathbb{X}$ is relatively compact with respect to the Gromov--Hausdorff--Prohorov metric if and only if \begin{enumerate}[(i)]
		\item There is a constant D such that $\emph{\text{diam}}(\mathcal{X})<D$ for all $\mathcal{X}\in \mathbb{K}$.
		\item For every $\delta>0$ there exists $N=N_\delta$ such that for all $\mathcal{X}\in \mathbb{K}$, $X$ can be covered by $N_\delta$ balls of radius $\delta$.
		\item $\sup_{\mathcal{X}\in \mathbb{K}} \mu_X(X) < + \infty$
	\end{enumerate}
\end{lemma}

\subsubsection{Proof of the main theorem}

We first need to reformulate the main theorem into a statement that we are better equipped to prove. Recall the definitions of $d^*$ and $\sim^*$ from Remark \ref{rmk::h_metric}: in this remark we observed that $(\frac{[0,L(T))}{\sim^*}, d^*)$ is $\mathbb{P}^x$-almost surely isometrically equivalent to $(\mathbf{T}(T,X,l),d(T,X,l))$ as defined in Section \ref{sec::trees}. Also write $\mathbf{e}$ for a Brownian excursion conditioned to reach height $1$ and $\hate$ for $b\sigma$ times a Brownian excursion conditioned to reach height $(b\sigma)^{-1}$ (where $\sigma$ is the constant from Proposition \ref{propn::mgaleinv}). Then it is clear from Brownian scaling that $(\mathbf{T}_{\mathbf{e}},d_{\mathbf{e}})$ is isometrically equivalent to $(\mathbf{T}_{\hate}, d_{\hate} )$ (where as in Section \ref{sec::trees}, $(\mathbf{T}_C, d_C)$ is the real tree with contour function $C$). 

Therefore, to prove Theorem \ref{thm::crtconv}, it is enough to prove that if $(\mathbf{T}^{(t)},d^{(t)})$ has the law of $\big(\frac{[0,L(T))}{\sim^*}, \frac{d^*}{t}\big)$ under $\mathbb{P}^x\big(\cdot \, \giv \, |N_t|>0\big)$ then 
\begin{equation} \label{eqn::main_thm_reformulate}
	\big(\mathbf{T}^{(t)},d^{(t)}\big) \overset{(d)}{\longrightarrow} \big(\mathbf{T}_{\hate}, d_{\hate} \big)
\end{equation}
as $n\to \infty$, with respect to the Gromov--Hausdorff topology.

We equip $(\mathbf{T}^{(t)},d^{(t)})$ with the measure $\mu^{(t)}$, which is defined to be the push forward of uniform measure on $[0,L(T))$ under the equivalence relation $\sim^*$. We also equip $(\mathbf{T}_{\hate},d_{\hate})$ with the measure $\mu_{\hate}$, where if $\hat{\tau}$ is the length of $\hate$ and $(\mathbf{T}_{\hate}, d_{\hate})=(\frac{[0,\hat\tau)}{\sim},d_{\hate})$ as described in Section \ref{sec::trees}, then $\mu_{\hate}$ is the push forward of uniform measure on $[0,\hat{\tau})$ under $\sim$. For the proof of Theorem \ref{thm::crtconv} we then need the following two ingredients: 
\begin{lemma}\label{lem::mainproof_1}
	The family \[\big( \mathbf{T}^{(t)},d^{(t)},\mu^{(t)}\big)_{t\ge 0}\] is tight with respect to the \emph{Gromov--Hausdorff--Prohorov} metric.
\end{lemma}
\begin{lemma}\label{lem::mainproof_2}
	$(\mathbf{T}^{(t)},d^{(t)},\mu^{(t)})_t$ converges to  $(\mathbf{T}_{\hate},d_{\hate},\mu_{\hate})$ as $t\to \infty$, with respect to the \emph{Gromov--Prohorov metric}.
\end{lemma}

Let us first see how this gives Theorem \ref{thm::crtconv}.
\\

\begin{proof2}{Theorem \ref{thm::crtconv}} 
	Since Gromov--Hausdorff--Prohorov convergence implies Gromov--Prohorov convergence, Lemma \ref{lem::mainproof_2} characterises subsequential limits with respect to the Gromov--Hausdorff--Prohorov metric. Thus we have the convergence in distribution 
	\[(\mathbf{T}^{(t)},d^{(t)},\mu^{(t)}) \underset{t\to \infty}{\longrightarrow} (\mathbf{T}_{\hate},d_{\hate},\mu_{\hate}) \] with respect to the Gromov--Hausdorff--Prohorov metric. This then implies (\ref{eqn::main_thm_reformulate}) by Remark \ref{rmk::ghpimplies} (GHP convergence implies GH convergence).
\end{proof2}
\\

All that remains therefore is to verify Lemmas \ref{lem::mainproof_1} and \ref{lem::mainproof_2}. In the following, we write $\mathbb{P}^x_{t}$ for the law of $(\mathbf{T}^{(t)},d^{(t)},\mu^{(t)})$ and $\mathbb{P}_{\hate}$ for the law of $(\mathbf{T}_{\hate},d_{\hate},\mu_{\hate})$. 
\\

\begin{proof2}{Lemma \ref{lem::mainproof_1}}
	We need to show that for any $\eps>0$ there exists a relatively compact $\mathbb{K}\subset \mathbb{X}$ (with respect to the Gromov--Hausdorff--Prohorov metric) such that \[\inf_{t\ge 0} \mathbb{P}^x_{t}((\mathbf{T}^{(t)},d^{(t)},\mu^{(t)})\in \mathbb{K})\geq 1-\eps. \] To do this, fix $\eps>0$ and let $\mathbb{K}_{R,M}$ be the relatively compact subset of $\mathbb{X}$ (by Lemma \ref{lemma::ghptight}), defined by $$\{ (X,r,\mu) \in \mathbb{X}: \text{diam}(X)\leq 2R\, ,\, \mu(X)=1\, , \, \forall k\geq 1 \text{ can cover } X \text{ with less than } 2^{4k}M \text{ balls of radius } 2^{-k} \}. $$

	We will prove that we can pick $R$ and $M$ large enough such that $\mathbb{P}^x_t((\mathbf{T}^{(t)},d^{(t)},\mu^{(t)})\in \mathbb{K}_{R,M})\geq 1-\eps$ for all $t\ge 0$. The key to the proof is the following claim:
	\begin{claim}\label{claim::tightness} For $M=M(\eps)$ and $R=R(\eps)$ large enough, we have 
	\begin{align}\label{eqn::tightness_diam}
		& \mathbb{P}^x_{t}\big( \text{diam}(T^{(t)})>2R\big)\leq \eps/2; \text{ and }
	\\
	 \label{eqn::tightness_main} & \mathbb{P}^x_{t}\big(\{\text{\emph{cannot} cover } (\mathbf{T}^{(t)},d^{(t)}) \text{ with } < \delta^{-4}M \text{ balls of radius } \delta\}\cap \{\text{diam}(\mathbf{T}^{(t)})\le 2R\}\big) \leq \delta \eps/2 \end{align} for all $\delta>0$ and $t\ge 0$. 
	\end{claim} With this in hand, since $\mu^{(t)}(\mathbf{T}^{(t)})=1$ for all $t\ge 0$, summing over $\delta=2^{-k}$ concludes the proof of tightness. 
	\medskip
	
	To prove Claim \ref{claim::tightness}, we will reformulate the probabilities in (\ref{eqn::tightness_diam}) and (\ref{eqn::tightness_main}) as conditional probabilities under $\mathbb{P}^x(\cdot \giv |N_t|>0)$. Recall that the law of $(\mathbf{T}^{(t)},d^{(t)})$ under $\mathbb{P}^x_t$ is that of $(\frac{[0,L(T))}{\sim^*},\frac{d^*}{t})$ under $\E{x}{\cdot \giv |N_t|>0}$. It is therefore immediate from the definition of $d^*$ that 
	$$\mathbb{P}^x_t\big( \text{diam}(T^{(t)})>2R\big) \leq \E{x}{|N_{Rt}|>0 \giv |N_t|>0}$$
	and so by Theorem \ref{criticalsurvivalprob}, we can take $R=R(\eps)$ large enough that this is less than $\eps/2$ for all $t$.
	
	Next, we look at \eqref{eqn::tightness_main}. Fixing $R$ from the previous paragraph, and given $\delta>0$, we divide $[0,2Rt]$ into intervals of length $t\delta/2=:b^{t,\delta}$ (so there are $4R/\delta$ of them.) 
	We will show that on the event
	\begin{equation} \label{eqn::event1} 
	\{\text{diam}(\frac{[0,L(T))}{\sim^*},\frac{d^*}{t})\leq 2R\}\, \bigcap \, \cap_{j=0}^{(4R/\delta)-1}\{ |S_j|\leq M(4R\delta^3)^{-1}\},
	\end{equation}
	where $S_j:=\{u\in N_{jb^{t,\delta}}: \exists w\in N_{(j+1)b^{t,\delta}} \text{ with } u\prec w \}$, we can cover $(\frac{[0,L(T))}{\sim^*},\frac{d^*}{t})$ with fewer than $\delta^{-4}M$ balls of radius $\delta$. 
	To do this,
	let $S$ be the set of equivalence classes of $\frac{[0,L(T))}{\sim^*}$ that contain a point $s$ such that $(v_s,H_{s})=(u,jb^{t,\delta})$ for some $u\in S_j$ and $j\in\{0,\cdots, (4R/
	\delta)-1\}$. On the event \eqref{eqn::event1} we have $|S|\leq \delta^{-4}M$, and so it suffices to show that any element of $(\frac{[0,L(T))}{\sim^*})$ has $(d^*/t)$ distance less than $\delta$ from some point in $S$. However, this holds since any point of $\frac{[0,L(T))}{\sim^*} $ corresponds to an individual $u\in T$, which must be in $N_r$ for some $r\in [jb^{t,\delta},(j+1)b^{t,\delta}]$, and $0\le j \le (4R/\delta)-1$. Such an individual must be a descendant of an individual $v\in S_{j-1}$ (or of $\emptyset$ if $j=0$), which in turn corresponds to a point of $\frac{[0,L(T))}{\sim^*}$ contained in $S$. Because $u\in N_r$ and $v\in N_{r'}$ with $|r-r'|\le \delta t$, the ($d^*/t$) distance between the corresponding points of $\frac{[0,L(T))}{\sim^*}$ is less than $\delta$, as required. 
	
	The above considerations mean that the probability on the left hand side of (\ref{eqn::tightness_main}) is  less than or equal to 
	\begin{equation}\label{eqn::reformulate_tightness_prob}
		\E{x}{\bigcup_{j=0}^{(4R/\delta)-1} \big\{|S_j| > \frac{M}{4R\delta^3}\big\} \giv |N_t|>0}
	\end{equation}
	and hence, we just need to estimate (\ref{eqn::reformulate_tightness_prob}). For any given $j,t,\delta$ we can bound, using Theorem \ref{criticalsurvivalprob} and the fact that $\sup_r \E{x}{|N_r|}<\infty$,
	\begin{equation*}
		\E{x}{|S_j|}  =\E{x}{\E{x}{|S_j|  \giv \mathcal{F}_{jb^{t,\delta}}}}=\E{x}{\sum\nolimits_{u\in N_{jb^{t,\delta}}} \E{X_u(jb^{t,\delta})}{|N_{b^{t,\delta}}|>0}} \le C(t\delta)^{-1} \end{equation*}
	for a constant $C$ that does not depend on $j,t$ or $\delta$ (and may now change from line to line). Then by Theorem \ref{criticalsurvivalprob} again, conditioning on the event $|N_t|>0$, we have
	$\E{x}{|S_j| \giv |N_t| >0} \leq C\delta^{-1},$ 
	and so by conditional Markov's inequality,
	\begin{equation*}
		\E{x}{|S_j| > \frac{M}{4R\delta^3} \giv |N_t|>0} \leq \frac{4RC\delta^2}{M}.
	\end{equation*} Finally, by a union bound we can write
	\begin{equation*}
		\E{x}{\bigcup_{j=0}^{(4R/\delta)-1} \big\{|\{u\in N_{jb^{t,\delta}}: \exists w\in N_{(j+1)b^{t,\delta}} \text{ with } u\leq w \}| > \frac{M}{4R\delta^3}\big\} \giv |N_t|>0} \leq \frac{16R^2 C \delta}{M},
	\end{equation*}
	and so choosing $M=M(R,\eps)=M(\eps)$ such that $16R^2CM^{-1}\leq \eps/2$ then gives (\ref{eqn::tightness_main}) and proves Claim \ref{claim::tightness}.
\end{proof2}
\\

\begin{proof2}{Lemma \ref{lem::mainproof_2}} For this we would like to show that $(\mathbf{T}^{(t)},d^{(t)},\mu^{(t)})$ converges to $(\mathcal{T}_{\hate},d_{\hate},\mu_{\hate})$ in the Gromov--Prohorov metric, or equivalently, with respect to the Gromov--weak topology. We will consider the latter formulation, in order to use the characterisation given by Lemma \ref{lemma::gwconv}. Since convergence in the Gromov--Hausdorff--Prohorov sense implies convergence in the Gromov--Prohorov/Gromov--weak sense, Lemma \ref{lem::mainproof_1} already shows that part (i) of the characterisation in Lemma \ref{lemma::gwconv} (relative compactness of the laws) is satisfied. Therefore, we need only show that for any polynomial $\Phi \in \Pi$, we have $\mathbb{P}^x_t[\Phi]\to \mathbb{P}_{\hate}[\Phi]$ as $t\to \infty$. To do this, we use Proposition \ref{propn::convtoex2}, and Proposition \ref{prop::distancematrices}.  
	
	Fix a polynomial 
	\begin{equation*}\Phi((X,r,\mu))=\int \mu^{\otimes k}(d(x_1,\cdots, x_n))\phi((r(x_i,x_j))_{1\leq i<j\leq k}),\end{equation*}
	where $\phi:[0,\infty)^{k\choose 2}\to \R$ is continuous and bounded.
	Examining the definitions, we see that 
	\begin{equation}\label{eqn::conv_1}
		\mathbb{P}^x_t[\Phi]=\mathcal{P}^x_k \big( \phi(D_t^H) \giv |N_t|>0 \big) 
	\end{equation}
	where $\mathcal{P}^x_k$ and the matrix $D_t^H$ are defined just before Proposition \ref{prop::distancematrices} \footnote{interpreting $\phi(D_t^H)=\phi(((D_t^H)_{ij})_{1\leq i<j\leq k})$}. Using this proposition, and also recalling the definition of the matrix $D_t^S$, we see that since $\phi$ is continuous and bounded,
	\begin{equation}\label{eqn::conv_2}
		\mathcal{P}^x_k \big( \phi(D_t^H) \giv |N_t|>0 \big)- \mathcal{P}^x_k \big( \phi(bD_t^S) \giv |N_t|>0 \big)\to 0
	\end{equation}
	as $t\to \infty$. Let us suppose for now that conditioning on $\{|N_t|>0\}$ is comparable to conditioning on the event $A_t:=\{\sup_t \mathbf{S}_t \geq b^{-1}t \}$, in the sense that we have 
	\begin{equation}\label{eqn::conv_3}
		\mathcal{P}^x_k \big( \phi(bD_t^S) \giv |N_t|>0 \big)- \mathcal{P}^x_k \big( \phi(bD_t^S) \giv A_t \big)\to 0
	\end{equation}
	as $t\to \infty$. Then since by \eqref{eqn::excconv} and Remark \ref{rmk::nnotint} we have 	\begin{equation*}
	\mathcal{P}^x_k \big( \phi(bD_t^S) \giv \sup_t \mathbf{S}_t \geq b^{-1}t \big) \to \mathbb{P}_{\hate}[\Phi]
	\end{equation*} as $t\to \infty$, we can combine (\ref{eqn::conv_1}), (\ref{eqn::conv_2}) and (\ref{eqn::conv_3}) to conclude that $\mathbb{P}^x_t[\Phi]\to \mathbb{P}_{\hate}[\Phi]$ as desired.
	
	Therefore, we are left to justify \eqref{eqn::conv_3}. In fact, it is enough (because $\phi$ is bounded by assumption) to show that 
	\begin{equation}\label{eqn::equiv_events}
	\E{x}{A_t \giv |N_t|>0}\to 1 \text{ and } \E{x}{|N_t|>0 \giv A_t}\to 1
	\end{equation}
	as $t\to \infty$. For the first statement, observe that $$\E{x}{A_t \giv |N_t|>0}\geq \E{x}{A_t \giv |N_{t(1+\delta)}|>0} \frac{\E{x}{|N_{t(1+\delta)}|>0}}{\E{x}{|N_{t}|>0}}$$ for any $\delta>0$, where for any $\delta>0$ the first term in the product on the right hand side converges to $1$ as $t\to \infty$ by Proposition \ref{propn::etabad}, and the second converges to $(1+\delta)^{-1}$ as $t\to \infty$ by Theorem \ref{criticalsurvivalprob}. To finish showing (\ref{eqn::equiv_events}), it is therefore enough to show that $\E{x}{A_t}\sim \E{x}{|N_t|>0}$ as $t\to \infty$. For this, we observe that we have an asymptotic for $\E{x}{|N_t|>0}$ from Theorem \ref{criticalsurvivalprob}, and also that Proposition \ref{propn::convtoex} allows us to compute an asymptotic for $\mathbb{P}^x(A_t)$, just as in for example \cite[Section 1.4, p263]{legallrandomtrees}. It is easy to check that these asymptotics agree, and hence the proof is complete.
\end{proof2}


\bibliographystyle{alpha_abbrvsort}
\bibliography{bbm_bibliography}

\begin{thebibliography}{BDMM07}

\bibitem[ADH13]{ghptop}
R.~Abraham, J.-F. Delmas, and P.~Hoscheit.
\newblock A note on the {G}romov-{H}ausdorff-{P}rokhorov distance between
  (locally) compact metric measure spaces.
\newblock {\em Electron. J. Probab.}, 18(14):21pp, 2013.

\bibitem[Ald91]{crti}
D.~Aldous.
\newblock The continuum random tree. {I}.
\newblock {\em Ann. Probab.}, 19(1):1--28, 1991.

\bibitem[Ald93]{crtiii}
D.~Aldous.
\newblock The continuum random tree. {III}.
\newblock {\em Ann. Probab.}, 21(1):248--289, 1993.

\bibitem[AH83]{ah}
S.~Asmussen and H.~Hering.
\newblock {\em Branching processes}, volume~3 of {\em Progress in Probability
  and Statistics}.
\newblock Birkh\"auser Boston, Inc., Boston, MA, 1983.

\bibitem[Bn91]{baniuc}
R.~Ba\~nuelos.
\newblock Intrinsic ultracontractivity and eigenfunction estimates for
  {S}chr\"odinger operators.
\newblock {\em J. Funct. Anal.}, 100(1):181--206, 1991.

\bibitem[BBS11]{nearcriticalabsorptionbbs}
J.~Berestycki, N.~Berestycki, and J.~Schweinsberg.
\newblock Survival of near-critical branching {B}rownian motion.
\newblock {\em J. Stat. Phys.}, 143(5):833--854, 2011.

\bibitem[BBS13]{bbmabsorption}
J.~Berestycki, N.~Berestycki, and J.~Schweinsberg.
\newblock The genealogy of branching {B}rownian motion with absorption.
\newblock {\em Ann. Probab.}, 41(2):527--618, 2013.

\bibitem[BBS14]{criticalabsorptionbbs}
J.~Berestycki, N.~Berestycki, and J.~Schweinsberg.
\newblock Critical branching {B}rownian motion with absorption: survival
  probability.
\newblock {\em Probab. Theory Related Fields}, 160(3-4):489--520, 2014.

\bibitem[BBS15]{criticalabsorptionbbsparticles}
J.~Berestycki, N.~Berestycki, and J.~Schweinsberg.
\newblock Critical branching {B}rownian motion with absorption: particle
  configurations.
\newblock {\em Ann. Inst. Henri Poincar\'e Probab. Stat.}, 51(4):1215--1250,
  2015.

\bibitem[Ber09]{NBcoasurvey}
N.~Berestycki.
\newblock {\em Recent progress in coalescent theory}, volume~16 of {\em Ensaios
  Matem\'aticos [Mathematical Surveys]}.
\newblock Sociedade Brasileira de Matem\'atica, Rio de Janeiro, 2009.

\bibitem[Big77]{Biggins}
J.~D. Biggins.
\newblock Martingale convergence in the branching random walk.
\newblock {\em J. Appl. Probability}, 14(1):25--37, 1977.

\bibitem[BS98]{bscoa}
E.~Bolthausen and A.-S. Sznitman.
\newblock On {R}uelle's probability cascades and an abstract cavity method.
\newblock {\em Comm. Math. Phys.}, 197(2):247--276, 1998.

\bibitem[BDMM06]{bdmm1}
\'E. Brunet, B.~Derrida, A.~H. Mueller, and S.~Munier.
\newblock Noisy traveling waves: effect of selection on genealogies.
\newblock {\em Europhys. Lett.}, 76(1):1--7, 2006.

\bibitem[BDMM07]{bdmm2}
\'E. Brunet, B.~Derrida, A.~H. Mueller, and S.~Munier.
\newblock Effect of selection on ancestry: an exactly soluble case and its
  phenomenological generalization.
\newblock {\em Phys. Rev. E (3)}, 76(4):041104, 20, 2007.

\bibitem[BBI01]{hauscomp}
D.~Burago, Y.~Burago, and S.~Ivanov.
\newblock {\em A course in metric geometry}, volume~33 of {\em Graduate Studies
  in Mathematics}.
\newblock American Mathematical Society, Providence, RI, 2001.

\bibitem[CR88]{chauvinrouault}
B.~Chauvin and A.~Rouault.
\newblock K{PP} equation and supercritical branching {B}rownian motion in the
  subcritical speed area. {A}pplication to spatial trees.
\newblock {\em Probab. Theory Related Fields}, 80(2):299--314, 1988.

\bibitem[Dav89]{davies}
E.~B. Davies.
\newblock {\em Heat kernels and spectral theory}, volume~92 of {\em Cambridge
  Tracts in Mathematics}.
\newblock Cambridge University Press, Cambridge, 1989.

\bibitem[DS84]{daviessimon}
E.~B. Davies and B.~Simon.
\newblock Ultracontractivity and the heat kernel for {S}chr\"odinger operators
  and {D}irichlet {L}aplacians.
\newblock {\em J. Funct. Anal.}, 59(2):335--395, 1984.

\bibitem[DLG02]{crtconvgd}
T.~Duquesne and J.-F. Le~Gall.
\newblock Random trees, {L}\'evy processes and spatial branching processes.
\newblock {\em Ast\'erisque}, (281):vi+147, 2002.

\bibitem[EK04]{eng}
J.~Engl\"ander and A.~E. Kyprianou.
\newblock Local extinction versus local exponential growth for spatial
  branching processes.
\newblock {\em Ann. Probab.}, 32(1A):78--99, 2004.

\bibitem[Eva10]{evans}
L.~C. Evans.
\newblock {\em Partial differential equations}, volume~19 of {\em Graduate
  Studies in Mathematics}.
\newblock American Mathematical Society, Providence, RI, second edition, 2010.

\bibitem[GT83]{gilbargtrudinger}
D.~Gilbarg and N.~S. Trudinger.
\newblock {\em Elliptic partial differential equations of second order}, volume
  224 of {\em Grundlehren der Mathematischen Wissenschaften [Fundamental
  Principles of Mathematical Sciences]}.
\newblock Springer-Verlag, Berlin, second edition, 1983.

\bibitem[GPW09]{mmspaces}
A.~Greven, P.~Pfaffelhuber, and A.~Winter.
\newblock Convergence in distribution of random metric measure spaces
  ({$\Lambda$}-coalescent measure trees).
\newblock {\em Probab. Theory Related Fields}, 145(1-2):285--322, 2009.

\bibitem[HH09]{spineapproach}
R.~Hardy and S.~C. Harris.
\newblock A spine approach to branching diffusions with applications to
  {$L^p$}-convergence of martingales.
\newblock In {\em S\'eminaire de {P}robabilit\'es {XLII}}, volume 1979 of {\em
  Lecture Notes in Math.}, pages 281--330. Springer, Berlin, 2009.

\bibitem[HH07]{harrissurvivalprobs}
J.~W. Harris and S.~C. Harris.
\newblock Survival probabilities for branching {B}rownian motion with
  absorption.
\newblock {\em Electron. Comm. Probab.}, 12:81--92, 2007.

\bibitem[HHK16]{hesse}
S.~C. Harris, M.~Hesse, and A.~E. Kyprianou.
\newblock Branching {B}rownian motion in a strip: survival near criticality.
\newblock {\em Ann. Probab.}, 44(1):235--275, 2016.

\bibitem[HR14]{spinelln}
S.~C. Harris and M.~I. Roberts.
\newblock A strong law of large numbers for branching processes: almost sure
  spine events.
\newblock {\em Electron. Comm. Probab.}, 19(28):6pp, 2014.

\bibitem[HR17]{manytofew}
S.~C. Harris and M.~I. Roberts.
\newblock The many-to-few lemma and multiple spines.
\newblock {\em Ann. Inst. Henri Poincar\'e Probab. Stat.}, 53(1):226--242,
  2017.

\bibitem[Her74]{hering1d}
H.~Hering.
\newblock Limit theorem for critical branching diffusion processes with
  absorbing barriers.
\newblock {\em Math. Biosci.}, 19:355--370, 1974.

\bibitem[Her78a]{heringmulti}
H.~Hering.
\newblock Multigroup branching diffusions.
\newblock In {\em Branching processes ({C}onf., {S}aint {H}ippolyte, {Q}ue.,
  1976)}, volume~5 of {\em Adv. Probab. Related Topics}, pages 177--217.
  Dekker, New York, 1978.

\bibitem[Her78b]{heringup}
H.~Hering.
\newblock Uniform primitivity of semigroups generated by perturbed elliptic
  differential operators.
\newblock {\em Math. Proc. Cambridge Philos. Soc.}, 83(2):261--268, 1978.

\bibitem[JS87]{limitthms}
J.~Jacod and A.~N. Shiryaev.
\newblock {\em Limit theorems for stochastic processes}, volume 288 of {\em
  Grundlehren der Mathematischen Wissenschaften [Fundamental Principles of
  Mathematical Sciences]}.
\newblock Springer-Verlag, Berlin, 1987.

\bibitem[Kes78]{kesten}
H.~Kesten.
\newblock Branching {B}rownian motion with absorption.
\newblock {\em Stochastic Processes Appl.}, 7(1):9--47, 1978.

\bibitem[Kol38]{kolmogorov}
A.~N. Kolmogorov.
\newblock Zur {L}\"{o}sung einer biologischen {A}ufgabe.
\newblock {\em Izv. Nau\v{cn}.-Issled. Inst. Mat. i Meh. Tomsk. Gosudarstv.
  Univ.}, 2:1--6, 1938.

\bibitem[Kyp04]{Kyprianou}
A.~E. Kyprianou.
\newblock Travelling wave solutions to the {K}-{P}-{P} equation: alternatives
  to {S}imon {H}arris' probabilistic analysis.
\newblock {\em Ann. Inst. Henri Poincar\'e Probab. Stat.}, 40(1):53--72, 2004.

\bibitem[LG05]{legallrandomtrees}
J.-F. Le~Gall.
\newblock Random trees and applications.
\newblock {\em Probab. Surv.}, 2:245--311, 2005.

\bibitem[Lyo97]{Lyons}
R.~Lyons.
\newblock A simple path to {B}iggins' martingale convergence for branching
  random walk.
\newblock In {\em Classical and modern branching processes ({M}inneapolis,
  {MN}, 1994)}, volume~84 of {\em IMA Vol. Math. Appl.}, pages 217--221.
  Springer, New York, 1997.

\bibitem[McK75]{mckean}
H.~P. McKean.
\newblock Application of {B}rownian motion to the equation of
  {K}olmogorov-{P}etrovskii-{P}iskunov.
\newblock {\em Comm. Pure Appl. Math.}, 28(3):323--331, 1975.

\bibitem[Mie08]{mie}
G.~Miermont.
\newblock Invariance principles for spatial multitype {G}alton-{W}atson trees.
\newblock {\em Ann. Inst. Henri Poincar\'e Probab. Stat.}, 44(6):1128--1161,
  2008.

\bibitem[Mie09]{mieghp}
G.~Miermont.
\newblock Tessellations of random maps of arbitrary genus.
\newblock {\em Ann. Sci. \'Ec. Norm. Sup\'er. (4)}, 42(5):725--781, 2009.

\bibitem[Pin85]{pin}
R.~G. Pinsky.
\newblock On the convergence of diffusion processes conditioned to remain in a
  bounded region for large time to limiting positive recurrent diffusion
  processes.
\newblock {\em Ann. Probab.}, 13(2):363--378, 1985.

\bibitem[RY99]{revuzyor}
D.~Revuz and M.~Yor.
\newblock {\em Continuous martingales and {B}rownian motion}, volume 293 of
  {\em Grundlehren der Mathematischen Wissenschaften [Fundamental Principles of
  Mathematical Sciences]}.
\newblock Springer-Verlag, Berlin, third edition, 1999.

\bibitem[Rob10]{spineroberts}
M.~I. Roberts.
\newblock Spine changes of measure and branching diffusions.
\newblock {\em PhD thesis, University of {B}ath,
  http://people.bath.ac.uk/mir20/thesis.pdf}, 2010.

\bibitem[Rud76]{rudin}
W.~Rudin.
\newblock {\em Principles of mathematical analysis}.
\newblock McGraw-Hill Book Co., New York-Auckland-D\"usseldorf, third edition,
  1976.
\newblock International Series in Pure and Applied Mathematics.

\bibitem[Sev58]{sevastyanov}
B.~A. Sevast'yanov.
\newblock Branching stochastic processes for particles diffusing in a bounded
  domain with absorbing boundaries.
\newblock {\em Teor. Veroyatnost. i Primenen.}, 3:121--136, 1958.

\bibitem[Wat65]{watanabe}
S.~Watanabe.
\newblock On the branching process for {B}rownian particles with an absorbing
  boundary.
\newblock {\em J. Math. Kyoto Univ.}, 4:385--398, 1965.

\end{thebibliography}

\end{document}